\documentclass[a4paper,12pt]{article}
\usepackage[utf8]{inputenc}

\usepackage{wrapfig}
\usepackage{amsmath} 
\usepackage{amsthm} 
\usepackage{amssymb} 
\usepackage{enumerate} 
\usepackage{esint} 
\usepackage{pgf,tikz} 
\usetikzlibrary{arrows} 
 \usepackage{yfonts} 
 \usepackage{mathrsfs} 
 \usepackage{mathabx} 
 \usepackage{graphicx}
\usepackage{caption}
\usepackage{subcaption}
\usepackage{mathtools} 
\usepackage[titletoc,toc]{appendix} 

\definecolor{ffffff}{rgb}{1.0,1.0,1.0}
\definecolor{qqqqff}{rgb}{0.0,0.0,1.0}
\definecolor{ffqqqq}{rgb}{1.0,0.0,0.0}
\definecolor{zzzzqq}{rgb}{0.6,0.6,0.0}
\definecolor{marronet}{rgb}{0.6,0.2,0}
\definecolor{negre}{rgb}{0,0,0}
\definecolor{vermell}{rgb}{0.8,0.05,0.05}
\definecolor{blau}{rgb}{0.3,0.2,1.}
\definecolor{blauclar}{rgb}{0.,0.,1.}
\definecolor{grisfosc}{rgb}{0.5,0.5,0.5}
\definecolor{verd}{rgb}{0.05,0.7,0.05}
\definecolor{taronja}{rgb}{0.9,0.5,0.05}
\definecolor{vermellclar}{rgb}{1.,0.,0.}
\definecolor{verdet}{rgb}{0,0.8,0.1}
\definecolor{blauverd}{rgb}{0,0.4,0.2}
\definecolor{grisclar}{rgb}{0.6274509803921569,0.6274509803921569,0.6274509803921569}
\definecolor{cqcqcq}{rgb}{0.7529411764705882,0.7529411764705882,0.7529411764705882}
\definecolor{aqaqaq}{rgb}{0.6274509803921569,0.6274509803921569,0.6274509803921569}
\definecolor{xdxdff}{rgb}{0.49019607843137253,0.49019607843137253,1.}
\definecolor{uuuuuu}{rgb}{0.26666666666666666,0.26666666666666666,0.26666666666666666}

\newcommand*\squared[1]{\tikz[baseline=(char.base)]{
            \node[shape=rectangle,draw,inner sep=2.4pt] (char) {#1}; \node[shape=rectangle,draw,inner sep=1pt] (char) {#1};}}

\newcommand{\C}{{\mathbb C}}       
\newcommand{\R}{{\mathbb R}}       
\newcommand{\N}{{\mathbb N}}       
\newcommand{\Z}{{\mathbb Z}}       
\newcommand{\D}{{\mathbb D}}
\newcommand{\T}{{\mathbb T}}
\newcommand{\W}{{\mathcal W}} 
\newcommand{\diam}{{\rm diam}}
\newcommand{\dist}{{\rm dist}}
\newcommand{\Dist}{{\rm D}}

\newcommand{\rf}[1]{{(\ref{#1})}}
\newcommand{\supp}{{\rm supp}}

\newcommand{\Beurling}{{\mathbf {B}}}
\newcommand{\Cauchy}{{\mathbf {C}}}

\newcommand{\norm}[1]{{\left\| {#1} \right\|}}

\newtheorem{theorem}{Theorem}
\newtheorem*{theorem*}{Theorem}
\newtheorem*{conjecture*}{Conjecture}
\newtheorem{lemma}[theorem]{Lemma}
\newtheorem{claim}[theorem]{Claim}

\newtheorem{corollary}[theorem]{Corollary}
\newtheorem*{corollary*}{Corollary}
\newtheorem{proposition}[theorem]{Proposition}
\newtheorem{conjecture}[theorem]{Conjecture}
\newtheorem{definition}[theorem]{Definition}

\newtheorem{remark}[theorem]{Remark}

\numberwithin{subsection}{section}
\numberwithin{theorem}{section}
\numberwithin{equation}{section}
\numberwithin{figure}{section}

\usepackage[affil-it]{authblk}

\usepackage{xcolor}
\newcommand{\raja}{A}

\title{Global smoothness of quasiconformal mappings in the Triebel-Lizorkin scale}

\author{Kari Astala
\thanks{KA (Department of Mathematics and Statistics, University of Helsinki,  Finland): \texttt{kari.astala@helsinki.fi}} \,  Mart\'i Prats
\thanks{MP (De\-par\-ta\-ment de Ma\-te\-m\`a\-ti\-ques, U\-ni\-ver\-si\-tat Au\-t\`o\-no\-ma de Bar\-ce\-lo\-na, Bellaterra, Catalonia; Centre de Recerca Matem\`atica, Bellaterra, Catalonia): \texttt{marti.prats@uab.cat}}  \, Eero Saksman
\thanks{ES (Department of Mathematics and Statistics, University of Helsinki, Finland): \texttt{eero.saksman@helsinki.fi}}}

\begin{document}
\maketitle
\bibliographystyle{alpha}

\begin{abstract} 
We study quasiconformal mappings in planar domains $\Omega$ and their regularity properties described  in terms of Sobolev, Bessel potential or Triebel-Lizorkin scales.
This leads to optimal conditions, in terms of the geometry of the boundary  $\partial \Omega$ and of the smoothness of the Beltrami coefficient, that guarantee the global regularity of the mappings in these classes. In the Triebel-Lizorkin class with smoothness below $1$, the same conditions give global regularity in $\Omega$  for the principal solutions with Beltrami coefficient supported in $\Omega$.
\end{abstract}

\renewcommand{\abstractname}{Resum\'e}
\begin{abstract}
Nous \'etudions les applications quasiconformes dans les domaines planaires $\Omega$ et leurs propri\'et\'es de r\'egularit\'e d\'ecrites en termes d'\'echelles de Sobolev, de potentiel de Bessel ou de Triebel-Lizorkin. Cela conduit \`a des conditions optimales, en termes de g\'eom\'etrie de la fronti\`ere $\partial \Omega$ et de finesse du coefficient de Beltrami, qui garantissent la r\'egularit\'e globale des applications dans ces classes. Dans la classe de Triebel-Lizorkin avec une r\'egularit\'e inf\'erieure \`a $1$, les m\^emes conditions donnent une r\'egularit\'e globale en $\Omega$ pour les solutions principales avec coefficient de Beltrami support\'e en $\Omega$.
\end{abstract}

\section{Introduction} Quasiconfomal mappings in planar domains $\Omega \subset \C$ are homeomorphisms  that satisfy the Beltrami equation
\begin{equation}\label{eqBeltrami}
\bar{\partial}f=\mu \, \partial f \mbox{\quad\; almost everywhere, \, with  \,} \| \mu \|_{\infty} \leq k < 1.
\end{equation}
Here, a priori,  $f \in W^{1,2}_{loc}(\Omega)$ and we often call the mapping $\mu$-quasiconformal when \eqref{eqBeltrami} holds.
 In many respects the local smoothness of the {\it Beltrami coefficient }$\mu = \mu_f$ locally dictates  the regularity of $f$. For instance, from \eqref{eqBeltrami} alone one has $f \in C^\alpha_{loc}(\Omega)$, where $\alpha = \frac{1-\| \mu \|_{\infty}}{1+\| \mu \|_{\infty}}$. On the other hand, it follows from the classical 
Schauder estimates, see e.g. \cite[Chapter 15]{AstalaIwaniecMartin},  that $f \in  C^{\ell +1,\alpha}_{loc}(\Omega)$ whenever $ \mu \in C^{\ell, \alpha}_{loc}(\Omega)$ and $\ell \in \N, 0 < \alpha < 1$. Similar relations hold  \cite{CruzMateuOrobitg} for the Sobolev regularity, or regularity measured in terms of the Besov or Triebel-Lizorkin spaces.

In this setting it is  natural   to ask in which domains and for which function spaces the local smoothness extends to a  {\it global regularity}, regularity in all of $\Omega$. The question has been studied, for instance, for the global higher integrability of the derivative of the mapping in \cite{AstalaKoskela}, \cite{Nieminen}, for the global H\"older continuity of $f$ in \cite{GehringMartio}, or for global $C^{\ell,\alpha}$-regularity in \cite{Kalaj}.
\smallskip

In this present paper we look for  optimal conditions for the global regularity in the more subtle smoothness scales described in terms of Besov, Bessel potential, Triebel-Lizorkin  or Sobolev spaces. There are actually two different ways to approach this question. Namely, given two domains $\Omega, \Omega'\subset \C$ one can study the   global regularity of quasiconfomal homeomorphisms $f: \Omega \to \Omega'$.
Another, and as it turns out, more difficult question is 
the   regularity of  the {\it principal mappings}. These are homeomorphic solutions to \eqref{eqBeltrami} in all of $\C$, where for a bounded domain $\Omega \subset \C$ we have
$$ f(z) = z + {\mathcal O}(1/z), \quad z \notin \Omega, \quad {\rm with} \quad \supp \, \mu \subset \overline{\Omega},
$$
so that $f$ is conformal in $\C \setminus \Omega$. Now the question is how the geometry of $\Omega$ and the smoothness of $\mu$ reflect on the global regularity  of $f{|_{\Omega}}$.
\smallskip

Let us begin with the first question, when a quasiconformal mapping $f$ between two domains lies in the  Sobolev\footnote{{We emphasize that for us $W^{s,p}$ means the Bessel potential space $W^{s,p}=(1-\Delta)^{-s/2}L^p$, while some authors (most notably Triebel) use this symbol for the diagonal Besov space  $B^{s}_{p,p}$ for non-integer values of $s.$}} (i.e. Bessel-potential) space $W^{s,p}(\Omega)$, {which includes {the case of} Hilbert spaces $H^s(\Omega)=W^{s,2}(\Omega)$.} 
We 
say that $\Omega$ is a $B^{s+1-\frac1p}_{p,p}${\it-domain} if the boundary $\partial \Omega$ admits a bi-Lipschitz parameterization contained in the Besov space
$B^{s+1-\frac1p}_{p,p}$,  c.f. Definition \ref{defBesovDomain}.
\smallskip

\begin{theorem}\label{theoSobolevStabilityOfDomains}
Let  $s > 0 $ and $1 < p < \infty$ with $sp>2$. Suppose $\Omega, \Omega'$ are simply connected, bounded $B^{s+1-\frac1p}_{p,p}$-domains and  $f:\Omega \to\Omega'$ is a quasiconformal mapping, with $\mu_f\in W^{s,p}(\Omega)$.  

Then $f\in W^{s+1,p}(\Omega)$.  
\end{theorem}

 The result holds for finitely connected domains as well, see Section \ref{secSobolev}. 
 Our assumptions on the boundary $\partial \Omega$ are, in fact,  optimal 
 for the global  $W^{s+1,p}$-regularity, as shown by the following. 
\smallskip
 
\begin{theorem}\label{theoSobolevRiemann}
Let  $\, s > 0 $ and $1 < p < \infty$ with $sp>2$, and suppose $\Omega$ is a bounded simply connected domain with Riemann map $\varphi :\D \to \Omega$. 

Then $\Omega$ is a  $B^{s+1-\frac1p}_{p,p}$-domain if and only if  $\varphi \in W^{s+1,p}(\D)$ and $\varphi^{-1} \in W^{s+1,p}(\Omega)$.
\end{theorem}
 \noindent For further aspects  see, in particular, Lemma \ref{intrinsic} and the discussion preceding it.  
Kellogg's classical result on H\"older regularity {states that every simply connected domain with a  bi-Lipschitz parameterization in the H\"older class $C^{s+1}(\T)$ has its Riemann mapping in $C^{s+1}(\D)$ whenever $s\in \R_+\setminus\Z$ (see  \cite[Theorem 3.6]{PommerenkeConformal}). It }can be interpreted as a particular case of a Besov scale version of Theorem  \ref{theoSobolevRiemann} for $p=\infty$. In particular,  Theorem  \ref{theoSobolevRiemann} characterizes the regularity of the mapping in terms of smoothness of the boundary. Other  problems related to Theorem  \ref{theoSobolevRiemann} on spaces of analytic (or univalent) functions  with less global smoothness on the unit disc   have been considered in the literature, see e.g. \cite{Walsh, PerezGonzalezRattya}.  For the many fascinating relations of the boundary smoothness $W^{3/2,2}$ see \cite{Bishop} and its references.
 
\medskip

In order to  sketch  the proof of Theorem \ref{theoSobolevStabilityOfDomains} 
we first note that  $\Omega$ is a Lipschitz-domain, so that one can extend $\mu_f$ to a Beltrami coefficient $\tilde \mu \in W^{s,p}(B_R)$, where the disc 
$B_R \supset \overline{\Omega}$.  
Solving \eqref{eqBeltrami}  in $B_R$ with $\mu = \tilde \mu$ gives us a quasiconformal map $F \in W^{s+1,p}_{loc}(B_R)$ \cite{CruzMateuOrobitg}. In particular, $F|_\Omega \in  W^{s+1,p}(\Omega)$, while Stoilow's theorem shows that $f = h \circ F$, where $h: F(\Omega) \to \Omega'$ is conformal  (see Figure \ref{figStoilow}). 
 In fact, a similar basic strategy was  applied in the H\"older scale in \cite{Kalaj}, who proved the H\"older    space analogue of Theorem \ref{theoSobolevStabilityOfDomains}.
 
 The argument in \cite{Kalaj} relies on Kellogg's regularity result, while in our setting  in order to prove Theorem \ref {theoSobolevRiemann} and the more general Theorem \ref{theoTriebelStabilityOfDomains} we need to establish  sharp analogues  both for Sobolev spaces (Theorem \ref{theoSobolevRiemann})  and in the setting of Triebel spaces (see \ref{theoTriebelRiemann} below). These results  are of independent interest.  In proving them
 we  make use of  the approach applied  by Pommerenke  in \cite{PommerenkeConformal} for the H\"older scale. However, to carry through this argument in the case of general Sobolev and Triebel-Lizorkin spaces is far from trivial, and we refer to {Section \ref{secRiemann} below for details. }

\begin{figure}[ht]
 \centering
\begin{tikzpicture}[line cap=round,line join=round,>=triangle 45,x=0.45cm,y=0.45cm]
\clip(-9.942839905217182,-9.858320421771625) rectangle (16.057380359564185,6.663805363507869);
\draw[line width=0.8pt] (-1.0320568603213844,1.704081582200247) -- (-1.023303249771372,1.6013813457699932) -- (-1.027455608202946,1.5158731245447317) -- (-1.0407886496899157,1.4430201502603222) -- (-1.0598268253027163,1.3790217290948383) -- (-1.0813591541499958,1.3207308144199295) -- (-1.1024500510763406,1.2655783994130652) -- (-1.1204464914331136,1.2115043407771255) -- (-1.1329818400442009,1.1568942361278387) -- (-1.1379766581933195,1.100521988923617) -- (-1.1336367891643497,1.0414977061263728) -- (-1.1184490095710256,0.979220585095931) -- (-1.0911745204171222,0.9133364575347116) -- (-1.050840538533138,0.8436996696133721) -- (-0.9967302357402931,0.7703389887221563) -- (-0.9283712597975099,0.693427238606728) -- (-0.845523057891877,0.6132543759613054) -- (-0.7481632101379336,0.5302037328659585) -- (-0.6364729672559455,0.44473116076895874) -- (-0.5108221733041891,0.35734682302912696) -- (-0.37175374104508796,0.2685993943471454) -- (-0.21996783422988864,0.1790624367288524) -- (-0.05630589779139872,0.08932273293757283) -- (0.029680935879016344,0.044561362747361104) -- (0.1182653363608567,-2.9630293423624495E-5) -- (0.3026697535438605,-0.08840962570320574) -- (0.4957375693986799,-0.17524428070103282) -- (0.6962207640735243,-0.2599784781065685) -- (0.9028081629785755,-0.3420797249052037) -- (1.1141401028940483,-0.4210420401772043) -- (1.328822635508101,-0.4963891020453533) -- (1.5454412337563759,-0.5676767821717227) -- (1.7625739796301152,-0.634495185020169) -- (1.9788042254149667,-0.6964702977871071) -- (2.192732733617739,-0.7532653455890801) -- (2.4029893141335523,-0.8045819351816009) -- (2.6082439905009758,-0.8501610591697027) -- (2.8072177403879133,-0.8897840213566001) -- (2.9986928687461636,-0.9232733325628161) -- (3.1815230853677363,-0.9504936149340993) -- (3.3546433718711786,-0.9713525414424113) -- (3.5170797364413238,-0.9858018259702244) -- (3.6679589679410354,-0.9938382680543356) -- (3.8065185143086913,-0.99550484505136) -- (3.9321166234502964,-0.9908918331730265) -- (4.044242898130315,-0.980137927525364) -- (4.142529429660415,-0.9634313199718216) -- (4.226762688480551,-0.9410106823263348) -- (4.296896363021909,-0.9131659910682989) -- (4.353065351536458,-0.8802391184573861) -- (4.3956011248729645,-0.8426241036120903) -- (4.425048691474533,-0.8007670058018531) -- (4.442185409167861,-0.7551652308885842) -- (4.448041901609591,-0.7063662105393426) -- (4.443925350550291,-0.6549653025179198) -- (4.431445448371748,-0.6016027690490136) -- (4.412543308648446,-0.546959678934648) -- (4.389523645779245,-0.4917525677884594) -- (4.365090548030439,-0.4367266794394191) -- (4.342387181626558,-0.38264760024253647) -- (4.325039776820402,-0.3302910867200358) -- (4.317206260168998,-0.28043087564246993) -- (4.323629910537315,-0.23382425434518844) -- (4.349698429646727,-0.19119515776154308) -- (4.401508831280398,-0.1532145473401708) -- (4.4859385665529015,-0.12047781569965862) -- (4.401508831280398,-0.1532145473401708) -- (4.349698429646727,-0.19119515776154308) -- (4.323629910537315,-0.23382425434518844) -- (4.317206260168998,-0.28043087564246993) -- (4.325039776820402,-0.3302910867200358) -- (4.342387181626558,-0.38264760024253647) -- (4.365090548030439,-0.4367266794394191) -- (4.389523645779245,-0.4917525677884594) -- (4.412543308648446,-0.546959678934648) -- (4.431445448371748,-0.6016027690490136) -- (4.443925350550291,-0.6549653025179198) -- (4.448041901609591,-0.7063662105393426) -- (4.442185409167861,-0.7551652308885842) -- (4.425048691474533,-0.8007670058018531) -- (4.3956011248729645,-0.8426241036120903) -- (4.353065351536458,-0.8802391184573861) -- (4.296896363021909,-0.9131659910682989) -- (4.226762688480551,-0.9410106823263348) -- (4.142529429660415,-0.9634313199718216) -- (4.044242898130315,-0.980137927525364) -- (3.9321166234502964,-0.9908918331730265) -- (3.8065185143086913,-0.99550484505136) -- (3.6679589679410354,-0.9938382680543356) -- (3.5170797364413238,-0.9858018259702244) -- (3.3546433718711786,-0.9713525414424113) -- (3.1815230853677363,-0.9504936149340993) -- (2.9986928687461636,-0.9232733325628161) -- (2.8072177403879133,-0.8897840213566001) -- (2.6082439905009758,-0.8501610591697027) -- (2.4029893141335523,-0.8045819351816009) -- (2.192732733617739,-0.7532653455890801) -- (1.9788042254149667,-0.6964702977871071) -- (1.7625739796301152,-0.634495185020169) -- (1.5454412337563759,-0.5676767821717227) -- (1.328822635508101,-0.4963891020453533) -- (1.1141401028940483,-0.4210420401772043) -- (0.9028081629785755,-0.3420797249052037) -- (0.6962207640735243,-0.2599784781065685) -- (0.4957375693986799,-0.17524428070103282) -- (0.3026697535438605,-0.08840962570320574) -- (0.1182653363608567,-2.9630293423624495E-5) -- (0.029680935879016344,0.044561362747361104) -- (-0.05630589779139872,0.08932273293757283) -- (-0.21996783422988864,0.1790624367288524) -- (-0.37175374104508796,0.2685993943471454) -- (-0.5108221733041891,0.35734682302912696) -- (-0.6364729672559455,0.44473116076895874) -- (-0.7481632101379336,0.5302037328659585) -- (-0.845523057891877,0.6132543759613054) -- (-0.9283712597975099,0.693427238606728) -- (-0.9967302357402931,0.7703389887221563) -- (-1.050840538533138,0.8436996696133721) -- (-1.0911745204171222,0.9133364575347116) -- (-1.1184490095710256,0.979220585095931) -- (-1.1336367891643497,1.0414977061263728) -- (-1.1379766581933195,1.100521988923617) -- (-1.1329818400442009,1.1568942361278387) -- (-1.1204464914331136,1.2115043407771255) -- (-1.1024500510763406,1.2655783994130652) -- (-1.0813591541499958,1.3207308144199295) -- (-1.0598268253027163,1.3790217290948383) -- (-1.0407886496899157,1.4430201502603222) -- (-1.027455608202946,1.5158731245447317) -- (-1.023303249771372,1.6013813457699932) -- (-1.0320568603213844,1.704081582200247);
\draw[line width=0.8pt] (4.4859385665529015,-0.12047781569965862) -- (4.593462739004595,-0.08846602306949634) -- (4.6961673665382655,-0.05591439851548427) -- (4.794422960110018,-0.022331575879169435) -- (4.888558732578229,0.012718563255889378) -- (4.97886498729503,0.049618515919512406) -- (5.065595455216655,0.08869882834504834) -- (5.148969580532645,0.13023981079428254) -- (5.229174754813906,0.17447329661710423) -- (5.306368499679614,0.22158444554593204) -- (5.380680597982991,0.2717135912248987) -- (5.452215173515918,0.3249581329737943) -- (5.521052719232416,0.38137447178676853) -- (5.587252073990978,0.4409799905657911) -- (5.650852347815749,0.5037550785888719) -- (5.711874795676574,0.5696452002130389) -- (5.770324639787892,0.6385630078120755) -- (5.826192840426485,0.7103904989490165) -- (5.879457815268084,0.784981217783403) -- (5.93008710724283,0.8621625007132947) -- (5.978039000909595,0.9417377662520439) -- (6.0232640873491405,1.0234888491398244) -- (6.065706777576157,1.1071783786899232) -- (6.105306764470129,1.1925522013697882) -- (6.142000433225088,1.279341847616834) -- (6.175722220318182,1.3672670428890112) -- (6.206405920997139,1.4560382629501278) -- (6.233985945286558,1.5453593333899343) -- (6.25839852251306,1.6349300733789671) -- (6.279582854349309,1.7244489836581476) -- (6.297482216376866,1.8136159787631427) -- (6.312045008167915,1.9021351634834838) -- (6.323225751885828,1.9897176535564445) -- (6.330986039404613,2.076084440595675) -- (6.335295427947173,2.1609693012545996) -- (6.336132284242469,2.244121750624568) -- (6.333484577201496,2.3253100398677695) -- (6.327350619112141,2.4043241980849026) -- (6.317739755352888,2.480979118417606) -- (6.3046730026253694,2.5551176883856477) -- (6.288183635705781,2.626613964458872) -- (6.268317722715162,2.695376390863905) -- (6.2451346089085,2.7613510626256206) -- (6.218707348982722,2.8245250328433635) -- (6.189123087903526,2.884929664201935) -- (6.156483390251062,2.9426440247173264) -- (6.120904518084484,2.99779832771723) -- (6.082517657325344,3.0505774160562886) -- (6.04146909265984,3.1012242905661185) -- (5.997920330959934,3.1500436827400833) -- (5.952048173223305,3.1974056716528296) -- (5.904044735032167,3.243749345114582) -- (5.854117415530954,3.2895865050601953) -- (5.802488814922827,3.3355054171729654) -- (5.749396600485074,3.3821746047431986) -- (5.695093321103335,3.4303466867615438) -- (5.639846170324701,3.4808622602470765) -- (5.5839366979296585,3.534653826810148) -- (5.527660470022889,3.592749763449989) -- (5.471326677642933,3.656278337587074) -- (5.4152576938906885,3.7264717663302473) -- (5.359788579576792,3.8046703199785967) -- (5.305266537387828,3.8923264697581033) -- (5.252050314571415,3.9910090797930335) -- (5.200509554140128,4.102407643312101) -- (5.252050314571415,3.9910090797930335) -- (5.305266537387828,3.8923264697581033) -- (5.359788579576792,3.8046703199785967) -- (5.4152576938906885,3.7264717663302473) -- (5.471326677642933,3.656278337587074) -- (5.527660470022889,3.592749763449989) -- (5.5839366979296585,3.534653826810148) -- (5.639846170324701,3.4808622602470765) -- (5.695093321103335,3.4303466867615438) -- (5.749396600485074,3.3821746047431986) -- (5.802488814922827,3.3355054171729654) -- (5.854117415530954,3.2895865050601953) -- (5.904044735032167,3.243749345114582) -- (5.952048173223305,3.1974056716528296) -- (5.997920330959934,3.1500436827400833) -- (6.04146909265984,3.1012242905661185) -- (6.082517657325344,3.0505774160562886) -- (6.120904518084484,2.99779832771723) -- (6.156483390251062,2.9426440247173264) -- (6.189123087903526,2.884929664201935) -- (6.218707348982722,2.8245250328433635) -- (6.2451346089085,2.7613510626256206) -- (6.268317722715162,2.695376390863905) -- (6.288183635705781,2.626613964458872) -- (6.3046730026253694,2.5551176883856477) -- (6.317739755352888,2.480979118417606) -- (6.327350619112141,2.4043241980849026) -- (6.333484577201496,2.3253100398677695) -- (6.336132284242469,2.244121750624568) -- (6.335295427947173,2.1609693012545996) -- (6.330986039404613,2.076084440595675) -- (6.323225751885828,1.9897176535564445) -- (6.312045008167915,1.9021351634834838) -- (6.297482216376866,1.8136159787631427) -- (6.279582854349309,1.7244489836581476) -- (6.25839852251306,1.6349300733789671) -- (6.233985945286558,1.5453593333899343) -- (6.206405920997139,1.4560382629501278) -- (6.175722220318182,1.3672670428890112) -- (6.142000433225088,1.279341847616834) -- (6.105306764470129,1.1925522013697882) -- (6.065706777576157,1.1071783786899232) -- (6.0232640873491405,1.0234888491398244) -- (5.978039000909595,0.9417377662520439) -- (5.93008710724283,0.8621625007132947) -- (5.879457815268084,0.784981217783403) -- (5.826192840426485,0.7103904989490165) -- (5.770324639787892,0.6385630078120755) -- (5.711874795676574,0.5696452002130389) -- (5.650852347815749,0.5037550785888719) -- (5.587252073990978,0.4409799905657911) -- (5.521052719232416,0.38137447178676853) -- (5.452215173515918,0.3249581329737943) -- (5.380680597982991,0.2717135912248987) -- (5.306368499679614,0.22158444554593204) -- (5.229174754813906,0.17447329661710423) -- (5.148969580532645,0.13023981079428254) -- (5.065595455216655,0.08869882834504834) -- (4.97886498729503,0.049618515919512406) -- (4.888558732578229,0.012718563255889378) -- (4.794422960110018,-0.022331575879169435) -- (4.6961673665382655,-0.05591439851548427) -- (4.593462739004595,-0.08846602306949634) -- (4.4859385665529015,-0.12047781569965862);
\draw[line width=0.8pt] (5.200509554140128,4.102407643312101) -- (5.161926928793554,4.198738829806718) -- (5.131297153541593,4.286465342307487) -- (5.106886207583741,4.3666729708636876) -- (5.087099963119688,4.440337270623559) -- (5.070480424970541,4.508329797571373) -- (5.05570189848486,4.571424199911474) -- (5.041567085729467,4.630302165099243) -- (5.027003109965041,4.685559222519019) -- (5.011057468406505,4.737710401808951) -- (4.9928939132681815,4.7871957468328175) -- (4.9717882610937485,4.834385685298755) -- (4.947124130370969,4.879586254024961) -- (4.918388607431208,4.923044179852322) -- (4.8851678406337395,4.964951816203998) -- (4.847142562834826,5.005451935291935) -- (4.804083542141592,5.044642375970342) -- (4.755846960950681,5.0825805472360885) -- (4.702369723271686,5.1192877873760665) -- (4.643664690335387,5.15475357876148) -- (4.5798158444867445,5.188939618289091) -- (4.5109733813626995,5.221783743469397) -- (4.437348730354744,5.253203714161766) -- (4.359209503356286,5.283100849956502) -- (4.276874371794796,5.311363523203856) -- (4.190707871948732,5.3378705076899955) -- (4.101115138549249,5.3624941829598995) -- (4.008536566666704,5.385103594287201) -- (3.9134424018819387,5.405567368290986) -- (3.816327258742337,5.423756484199517) -- (3.7177045675026834,5.4395469007609165) -- (3.6181009491507927,5.452822038800781) -- (3.518050518717933,5.463475119426751) -- (3.4180891168740266,5.471411357880018) -- (3.3187484698076397,5.476550013033767) -- (3.220550277390751,5.478826292538589) -- (3.1240002296283116,5.4781931136147985) -- (3.0295819513925815,5.474622719491734) -- (2.937750875442258,5.468108151493973) -- (2.8489280437263855,5.458664576774506) -- (2.7634938369730433,5.446330471694849) -- (2.6817816325628296,5.431168660852099) -- (2.6040713906871185,5.413267211752937) -- (2.530583168791111,5.392740185134572) -- (2.461470564301659,5.369728240932628) -- (2.3968140856398894,5.344399099895974) -- (2.3366144515185945,5.316947860848503) -- (2.2807858185244196,5.287597173597846) -- (2.2291489369848296,5.256597267491043) -- (2.181424235119863,5.224225835617139) -- (2.1372248314786644,5.190787774656739) -- (2.0960494756608097,5.156614780378504) -- (2.057275417322407,5.122064798782586) -- (2.020151203466988,5.0875213328910025) -- (1.983789404021181,5.053392605184975) -- (1.9471592656951677,5.020110575689182) -- (1.9090792941279275,4.988129815702985) -- (1.868209764317261,4.957926237178571) -- (1.8230451593346033,4.929995677746061) -- (1.7719065373246161,4.90485234138555) -- (1.7129338267895702,4.883027094746092) -- (1.6440780501585062,4.865065619111631) -- (1.563093475641184,4.8515264180138775) -- (1.4675296973668142,4.842978680492123) -- (1.3547236438075743,4.84) -- (1.4675296973668142,4.842978680492123) -- (1.563093475641184,4.8515264180138775) -- (1.6440780501585062,4.865065619111631) -- (1.7129338267895702,4.883027094746092) -- (1.7719065373246161,4.90485234138555) -- (1.8230451593346033,4.929995677746061) -- (1.868209764317261,4.957926237178571) -- (1.9090792941279275,4.988129815702985) -- (1.9471592656951677,5.020110575689182) -- (1.983789404021181,5.053392605184975) -- (2.020151203466988,5.0875213328910025) -- (2.057275417322407,5.122064798782586) -- (2.0960494756608097,5.156614780378504) -- (2.1372248314786644,5.190787774656739) -- (2.181424235119863,5.224225835617139) -- (2.2291489369848296,5.256597267491043) -- (2.2807858185244196,5.287597173597846) -- (2.3366144515185945,5.316947860848503) -- (2.3968140856398894,5.344399099895974) -- (2.461470564301659,5.369728240932628) -- (2.530583168791111,5.392740185134572) -- (2.6040713906871185,5.413267211752937) -- (2.6817816325628296,5.431168660852099) -- (2.7634938369730433,5.446330471694849) -- (2.8489280437263855,5.458664576774506) -- (2.937750875442258,5.468108151493973) -- (3.0295819513925815,5.474622719491734) -- (3.1240002296283116,5.4781931136147985) -- (3.220550277390751,5.478826292538589) -- (3.3187484698076397,5.476550013033767) -- (3.4180891168740266,5.471411357880018) -- (3.518050518717933,5.463475119426751) -- (3.6181009491507927,5.452822038800781) -- (3.7177045675026834,5.4395469007609165) -- (3.816327258742337,5.423756484199517) -- (3.9134424018819387,5.405567368290986) -- (4.008536566666704,5.385103594287201) -- (4.101115138549249,5.3624941829598995) -- (4.190707871948732,5.3378705076899955) -- (4.276874371794796,5.311363523203856) -- (4.359209503356286,5.283100849956502) -- (4.437348730354744,5.253203714161766) -- (4.5109733813626995,5.221783743469397) -- (4.5798158444867445,5.188939618289091) -- (4.643664690335387,5.15475357876148) -- (4.702369723271686,5.1192877873760665) -- (4.755846960950681,5.0825805472360885) -- (4.804083542141592,5.044642375970342) -- (4.847142562834826,5.005451935291935) -- (4.8851678406337395,4.964951816203998) -- (4.918388607431208,4.923044179852322) -- (4.947124130370969,4.879586254024961) -- (4.9717882610937485,4.834385685298755) -- (4.9928939132681815,4.7871957468328175) -- (5.011057468406505,4.737710401808951) -- (5.027003109965041,4.685559222519019) -- (5.041567085729467,4.630302165099243) -- (5.05570189848486,4.571424199911474) -- (5.070480424970541,4.508329797571373) -- (5.087099963119688,4.440337270623559) -- (5.106886207583741,4.3666729708636876) -- (5.131297153541593,4.286465342307487) -- (5.161926928793554,4.198738829806718) -- (5.200509554140128,4.102407643312101);
\draw[line width=0.8pt] (1.3547236438075743,4.84) -- (1.179204198069859,4.844066117837243) -- (1.033088495306832,4.855342404250523) -- (0.9116832072515378,4.872532120327284) -- (0.8108111331113721,4.894447360575503) -- (0.7267723099800341,4.9200015383993865) -- (0.65630705312273,4.9482023150539725) -- (0.596560876231622,4.978144957575446) -- (0.545051241748505,5.00900611118392) -- (0.4996360913516966,5.040037971655507) -- (0.45848410670413214,5.070562843160401) -- (0.42004665055964563,5.099968067063826) -- (0.38303133832442754,5.127701307186559) -- (0.34637719017064383,5.15326617702187) -- (0.30923131379920366,5.176218194405619) -- (0.2709270679486618,5.196161049136308) -- (0.23096365674724384,5.2127431690418815) -- (0.18898710500497923,5.225654569990029) -- (0.14477256454292914,5.2346239753387955) -- (0.09820790165649694,5.239416190324272) -- (0.049278515809806744,5.239829716882144) -- (-0.0019466603418632734,5.235694594399889) -- (-0.05532802950460513,5.226870451896401) -- (-0.11066711677637764,5.213244757125812) -- (-0.16771732505848705,5.194731248102317) -- (-0.2261937087510668,5.171268532542761) -- (-0.2857818156355687,5.142818840723778) -- (-0.3461456468472785,5.109366917250267) -- (-0.40693478484087053,5.070919037231988) -- (-0.46779073925201403,5.0275021323650435) -- (-0.5283525605580435,4.9791630124150466) -- (-0.5882617714407082,4.925967667598756) -- (-0.6471666657540142,4.868000637360939) -- (-0.7047260250001689,4.805364431043271) -- (-0.760612302216647,4.7381789859420325) -- (-0.8145143231773863,4.666581148251399) -- (-0.8661395548111285,4.590724162389095) -- (-0.9152159907399201,4.510777154201205) -- (-0.9614937038407835,4.426924593542911) -- (-1.0047461157335729,4.339365721731948) -- (-1.0447710330980307,4.248313929371545) -- (-1.0813915007230541,4.153996070039665) -- (-1.1144565211911879,4.056651695341279) -- (-1.1438416911013531,3.9565321968205005) -- (-1.1694498037328327,3.8538998402293276) -- (-1.1912114680535164,3.749026677649797) -- (-1.2090857939754285,3.6421933229663237) -- (-1.2230611937605465,3.533687576185006) -- (-1.2331563494799231,3.423802882096676) -- (-1.2394213964291296,3.312836608780489) -- (-1.2419393724030297,3.2010881314448176) -- (-1.2408279827328974,3.088856707102243) -- (-1.2362417309888973,2.97643912557542) -- (-1.228374465250932,2.8641271223306037) -- (-1.2174623898508796,2.752204538635609) -- (-1.2037875924892254,2.6409442145389934) -- (-1.1876821366291095,2.530604600167245) -- (-1.169532769070797,2.4214260708367474) -- (-1.1497862926095888,2.3136269314773172) -- (-1.1289556536801828,2.2073990958640834) -- (-1.107626794890502,2.102903426154499) -- (-1.0864663223480002,2.000264718227262) -- (-1.0662300376814602,1.8995663183199278) -- (-1.0477723846612979,1.8008443564619943) -- (-1.0320568603213844,1.704081582200247) -- (-1.0477723846612979,1.8008443564619943) -- (-1.0662300376814602,1.8995663183199278) -- (-1.0864663223480002,2.000264718227262) -- (-1.107626794890502,2.102903426154499) -- (-1.1289556536801828,2.2073990958640834) -- (-1.1497862926095888,2.3136269314773172) -- (-1.169532769070797,2.4214260708367474) -- (-1.1876821366291095,2.530604600167245) -- (-1.2037875924892254,2.6409442145389934) -- (-1.2174623898508796,2.752204538635609) -- (-1.228374465250932,2.8641271223306037) -- (-1.2362417309888973,2.97643912557542) -- (-1.2408279827328974,3.088856707102243) -- (-1.2419393724030297,3.2010881314448176) -- (-1.2394213964291296,3.312836608780489) -- (-1.2331563494799231,3.423802882096676) -- (-1.2230611937605465,3.533687576185006) -- (-1.2090857939754285,3.6421933229663237) -- (-1.1912114680535164,3.749026677649797) -- (-1.1694498037328327,3.8538998402293276) -- (-1.1438416911013531,3.9565321968205005) -- (-1.1144565211911879,4.056651695341279) -- (-1.0813915007230541,4.153996070039665) -- (-1.0447710330980307,4.248313929371545) -- (-1.0047461157335729,4.339365721731948) -- (-0.9614937038407835,4.426924593542911) -- (-0.9152159907399201,4.510777154201205) -- (-0.8661395548111285,4.590724162389095) -- (-0.8145143231773863,4.666581148251399) -- (-0.760612302216647,4.7381789859420325) -- (-0.7047260250001689,4.805364431043271) -- (-0.6471666657540142,4.868000637360939) -- (-0.5882617714407082,4.925967667598756) -- (-0.5283525605580435,4.9791630124150466) -- (-0.46779073925201403,5.0275021323650435) -- (-0.40693478484087053,5.070919037231988) -- (-0.3461456468472785,5.109366917250267) -- (-0.2857818156355687,5.142818840723778) -- (-0.2261937087510668,5.171268532542761) -- (-0.16771732505848705,5.194731248102317) -- (-0.11066711677637764,5.213244757125812) -- (-0.05532802950460513,5.226870451896401) -- (-0.0019466603418632734,5.235694594399889) -- (0.049278515809806744,5.239829716882144) -- (0.09820790165649694,5.239416190324272) -- (0.14477256454292914,5.2346239753387955) -- (0.18898710500497923,5.225654569990029) -- (0.23096365674724384,5.2127431690418815) -- (0.2709270679486618,5.196161049136308) -- (0.30923131379920366,5.176218194405619) -- (0.34637719017064383,5.15326617702187) -- (0.38303133832442754,5.127701307186559) -- (0.42004665055964563,5.099968067063826) -- (0.45848410670413214,5.070562843160401) -- (0.4996360913516966,5.040037971655507) -- (0.545051241748505,5.00900611118392) -- (0.596560876231622,4.978144957575446) -- (0.65630705312273,4.9482023150539725) -- (0.7267723099800341,4.9200015383993865) -- (0.8108111331113721,4.894447360575503) -- (0.9116832072515378,4.872532120327284) -- (1.033088495306832,4.855342404250523) -- (1.179204198069859,4.844066117837243) -- (1.3547236438075743,4.84);
\draw[line width=0.8pt] (-0.7132163358567141,-3.626039421814307) -- (-0.7309290596531317,-3.75843231604029) -- (-0.7519671154358886,-3.888372196137955) -- (-0.7757851112263201,-4.016140266650912) -- (-0.8018504700423497,-4.141984903099788) -- (-0.8296444779736074,-4.266123028072959) -- (-0.8586632685081851,-4.388741478054184) -- (-0.8884187431110241,-4.509998360987172) -- (-0.9184394280539417,-4.630024404577057) -- (-0.9482712674972892,-4.748924295328799) -- (-0.977478352823248,-4.866778008322502) -- (-1.0056435882207595,-4.983642127725649) -- (-1.0323692925220904,-5.099551158042258) -- (-1.0572777372910338,-5.214518826098953) -- (-1.080011621162745,-5.32853937376796) -- (-1.1002344804352122,-5.441588841427019) -- (-1.1176310359123638,-5.553626342156211) -- (-1.131907475998809,-5.664595326671712) -- (-1.1427916760462147,-5.77442483899646) -- (-1.150033353951318,-5.883030762867738) -- (-1.153404162005574,-5.990317058881683) -- (-1.1526977149964366,-6.096176992374714) -- (-1.1477295545602777,-6.20049435204187) -- (-1.1383370497869394,-6.303144659292081) -- (-1.1243792340759222,-6.403996368340335) -- (-1.1057365782442092,-6.502912057036794) -- (-1.0823106998857228,-6.599749608432806) -- (-1.0540240089824213,-6.694363383083837) -- (-1.020819289767025,-6.786605382089342) -- (-0.9826592188373819,-6.87632640086953) -- (-0.9395258195224679,-6.963377173679059) -- (-0.8914198525000205,-7.047609508857658) -- (-0.8383601426658087,-7.128877414817647) -- (-0.7803828422545386,-7.207038216768401) -- (-0.7175406302123932,-7.281953664177708) -- (-0.6499018478212089,-7.353491028970071) -- (-0.5775495705742857,-7.421524194461898) -- (-0.5005806163038332,-7.485934735033652) -- (-0.4191044895600521,-7.5466129865388725) -- (-0.33324226224185105,-7.603459107450159) -- (-0.24312539047919746,-7.656384130742042) -- (-0.14889446776710497,-7.70531100651079) -- (-0.05069791435125612,-7.750175635331133) -- (0.05130939713474142,-7.790927892349897) -- (0.15696757978046105,-7.827532642116565) -- (0.2661132342581304,-7.859970744150757) -- (0.37858101443046743,-7.88824004924662) -- (0.49420525670688265,-7.91235638651415) -- (0.6128216731480457,-7.932354541157426) -- (0.7342691083188162,-7.948289222989755) -- (0.8583913598895401,-7.960236025685752) -- (0.9850390629857103,-7.968292376770331) -- (1.1140716382859925,-7.972578478344609) -- (1.245359303868615,-7.973238238548737) -- (1.3787851508061253,-7.970440193761647) -- (1.5142472825085078,-7.964378421537723) -- (1.6516610178146705,-7.955273444280377) -- (1.7909611578322933,-7.9433731236525595) -- (1.9321043165260434,-7.928953545724183) -- (2.075071315054153,-7.9123198968564585) -- (2.2198696398533624,-7.893807330323158) -- (2.3665359644722335,-7.873781823668799) -- (2.5151387351528154,-7.852641026803731) -- (2.66578082016069,-7.830815100836168) -- (2.81860222286337,-7.8087675476411125) -- (2.66578082016069,-7.830815100836168) -- (2.5151387351528154,-7.852641026803731) -- (2.3665359644722335,-7.873781823668799) -- (2.2198696398533624,-7.893807330323158) -- (2.075071315054153,-7.9123198968564585) -- (1.9321043165260434,-7.928953545724183) -- (1.7909611578322933,-7.9433731236525595) -- (1.6516610178146705,-7.955273444280377) -- (1.5142472825085078,-7.964378421537723) -- (1.3787851508061253,-7.970440193761647) -- (1.245359303868615,-7.973238238548737) -- (1.1140716382859925,-7.972578478344609) -- (0.9850390629857103,-7.968292376770331) -- (0.8583913598895401,-7.960236025685752) -- (0.7342691083188162,-7.948289222989755) -- (0.6128216731480457,-7.932354541157426) -- (0.49420525670688265,-7.91235638651415) -- (0.37858101443046743,-7.88824004924662) -- (0.2661132342581304,-7.859970744150757) -- (0.15696757978046105,-7.827532642116565) -- (0.05130939713474142,-7.790927892349897) -- (-0.05069791435125612,-7.750175635331133) -- (-0.14889446776710497,-7.70531100651079) -- (-0.24312539047919746,-7.656384130742042) -- (-0.33324226224185105,-7.603459107450159) -- (-0.4191044895600521,-7.5466129865388725) -- (-0.5005806163038332,-7.485934735033652) -- (-0.5775495705742857,-7.421524194461898) -- (-0.6499018478212089,-7.353491028970071) -- (-0.7175406302123932,-7.281953664177708) -- (-0.7803828422545386,-7.207038216768401) -- (-0.8383601426658087,-7.128877414817647) -- (-0.8914198525000205,-7.047609508857658) -- (-0.9395258195224679,-6.963377173679059) -- (-0.9826592188373819,-6.87632640086953) -- (-1.020819289767025,-6.786605382089342) -- (-1.0540240089824213,-6.694363383083837) -- (-1.0823106998857228,-6.599749608432806) -- (-1.1057365782442092,-6.502912057036794) -- (-1.1243792340759222,-6.403996368340335) -- (-1.1383370497869394,-6.303144659292081) -- (-1.1477295545602777,-6.20049435204187) -- (-1.1526977149964366,-6.096176992374714) -- (-1.153404162005574,-5.990317058881683) -- (-1.150033353951318,-5.883030762867738) -- (-1.1427916760462147,-5.77442483899646) -- (-1.131907475998809,-5.664595326671712) -- (-1.1176310359123638,-5.553626342156211) -- (-1.1002344804352122,-5.441588841427019) -- (-1.080011621162745,-5.32853937376796) -- (-1.0572777372910338,-5.214518826098953) -- (-1.0323692925220904,-5.099551158042258) -- (-1.0056435882207595,-4.983642127725649) -- (-0.977478352823248,-4.866778008322502) -- (-0.9482712674972892,-4.748924295328799) -- (-0.9184394280539417,-4.630024404577057) -- (-0.8884187431110241,-4.509998360987172) -- (-0.8586632685081851,-4.388741478054184) -- (-0.8296444779736074,-4.266123028072959) -- (-0.8018504700423497,-4.141984903099788) -- (-0.7757851112263201,-4.016140266650912) -- (-0.7519671154358886,-3.888372196137955) -- (-0.7309290596531317,-3.75843231604029) -- (-0.7132163358567141,-3.626039421814307);
\draw[line width=0.8pt] (2.81860222286337,-7.8087675476411125) -- (2.939882869229073,-7.787057114678665) -- (3.0515493794488133,-7.758726689471583) -- (3.154588096111555,-7.724760547678717) -- (3.249906943493037,-7.686102501946247) -- (3.3383395504726727,-7.643655012935654) -- (3.4206492687853767,-7.598278380016196) -- (3.497533086608362,-7.550790011621895) -- (3.5696254374828595,-7.501963775273037) -- (3.6375019045707986,-7.45252942726217) -- (3.7016828202464227,-7.403172122004622) -- (3.762636761022858,-7.35453200105352) -- (3.82078393781362,-7.307203861779323) -- (3.8764994815290703,-7.261736905713868) -- (3.9301166240078147,-7.21863456655891) -- (3.98192977428305,-7.178354417859181) -- (4.032197490183856,-7.14130816033996) -- (4.081145345271429,-7.107861688909151) -- (4.128968691110257,-7.078335239323858) -- (4.1758353148742575,-7.053003614521492) -- (4.2218879922878365,-7.03209649061536) -- (4.267246935901911,-7.015798802554787) -- (4.312012138704869,-7.00425120944973) -- (4.356265613068474,-6.997550639559906) -- (4.400073525028714,-6.995750914948432) -- (4.443488223901605,-6.998863455799979) -- (4.486550167233929,-7.006858064403413) -- (4.529289741088914,-7.019663788798967) -- (4.571728975666876,-7.037169866089922) -- (4.613883156260791,-7.059226745418774) -- (4.655762329546815,-7.085647190607937) -- (4.697372705209754,-7.1162074624649385) -- (4.7387179529034755,-7.150648580752129) -- (4.779800394546264,-7.188677665820896) -- (4.820622091951128,-7.229969359910405) -- (4.861185829791039,-7.2741673281108135) -- (4.901495993899134,-7.320885838991032) -- (4.941559344903847,-7.369711424890971) -- (4.9813856871989906,-7.420204621878304) -- (5.020988433248791,-7.47190178936974) -- (5.060385063227853,-7.524317009416804) -- (5.099597479996082,-7.57694406565612) -- (5.13865225940855,-7.629258501924216) -- (5.177580795960297,-7.680719760536828) -- (5.2164193437660895,-7.730773400232716) -- (5.255208952875118,-7.778853393781985) -- (5.293995300920642,-7.824384505258925) -- (5.332828420104576,-7.866784746979347) -- (5.371762319517024,-7.905467916102437) -- (5.4108545027907615,-7.939846210897116) -- (5.450165381090657,-7.9693329266729105) -- (5.48975758143804,-7.993345231375323) -- (5.52969515037002,-8.01130702084573) -- (5.57004265293374,-8.022651853745767) -- (5.6108641670155865,-8.026825966146237) -- (5.652222173005336,-8.023291365780532) -- (5.694176338795252,-8.011529005962535) -- (5.736782200114122,-7.9910420391690735) -- (5.780089736196248,-7.961359150286846) -- (5.824141840785368,-7.922037969523872) -- (5.868972688473537,-7.872668564985455) -- (5.914605996374951,-7.812877014914646) -- (5.961053181134704,-7.742329059597217) -- (6.008311411272507,-7.6607338329311485) -- (6.056361554861333,-7.567847673660623) -- (6.008311411272507,-7.6607338329311485) -- (5.961053181134704,-7.742329059597217) -- (5.914605996374951,-7.812877014914646) -- (5.868972688473537,-7.872668564985455) -- (5.824141840785368,-7.922037969523872) -- (5.780089736196248,-7.961359150286846) -- (5.736782200114122,-7.9910420391690735) -- (5.694176338795252,-8.011529005962535) -- (5.652222173005336,-8.023291365780532) -- (5.6108641670155865,-8.026825966146237) -- (5.57004265293374,-8.022651853745767) -- (5.52969515037002,-8.01130702084573) -- (5.48975758143804,-7.993345231375323) -- (5.450165381090657,-7.9693329266729105) -- (5.4108545027907615,-7.939846210897116) -- (5.371762319517024,-7.905467916102437) -- (5.332828420104576,-7.866784746979347) -- (5.293995300920642,-7.824384505258925) -- (5.255208952875118,-7.778853393781985) -- (5.2164193437660895,-7.730773400232716) -- (5.177580795960297,-7.680719760536828) -- (5.13865225940855,-7.629258501924216) -- (5.099597479996082,-7.57694406565612) -- (5.060385063227853,-7.524317009416804) -- (5.020988433248791,-7.47190178936974) -- (4.9813856871989906,-7.420204621878304) -- (4.941559344903847,-7.369711424890971) -- (4.901495993899134,-7.320885838991032) -- (4.861185829791039,-7.2741673281108135) -- (4.820622091951128,-7.229969359910405) -- (4.779800394546264,-7.188677665820896) -- (4.7387179529034755,-7.150648580752129) -- (4.697372705209754,-7.1162074624649385) -- (4.655762329546815,-7.085647190607937) -- (4.613883156260791,-7.059226745418774) -- (4.571728975666876,-7.037169866089922) -- (4.529289741088914,-7.019663788798967) -- (4.486550167233929,-7.006858064403413) -- (4.443488223901605,-6.998863455799979) -- (4.400073525028714,-6.995750914948432) -- (4.356265613068474,-6.997550639559906) -- (4.312012138704869,-7.00425120944973) -- (4.267246935901911,-7.015798802554787) -- (4.2218879922878365,-7.03209649061536) -- (4.1758353148742575,-7.053003614521492) -- (4.128968691110257,-7.078335239323858) -- (4.081145345271429,-7.107861688909151) -- (4.032197490183856,-7.14130816033996) -- (3.98192977428305,-7.178354417859181) -- (3.9301166240078147,-7.21863456655891) -- (3.8764994815290703,-7.261736905713868) -- (3.82078393781362,-7.307203861779323) -- (3.762636761022858,-7.35453200105352) -- (3.7016828202464227,-7.403172122004622) -- (3.6375019045707986,-7.45252942726217) -- (3.5696254374828595,-7.501963775273037) -- (3.497533086608362,-7.550790011621895) -- (3.4206492687853767,-7.598278380016196) -- (3.3383395504726727,-7.643655012935654) -- (3.249906943493037,-7.686102501946247) -- (3.154588096111555,-7.724760547678717) -- (3.0515493794488133,-7.758726689471583) -- (2.939882869229073,-7.787057114678665) -- (2.81860222286337,-7.8087675476411125);
\draw[line width=0.8pt] (6.056361554861333,-7.567847673660623) -- (6.098162424431478,-7.475983507271256) -- (6.139289113461748,-7.372402314326271) -- (6.179352843132004,-7.258544909217158) -- (6.217999479274099,-7.135752173921837) -- (6.2549090203129625,-7.00527079853413) -- (6.289795018958379,-6.868258808928516) -- (6.32240393764747,-6.725790881560142) -- (6.352514437737893,-6.578863445400117) -- (6.379936602451742,-6.42839957100605) -- (6.404511093570138,-6.275253646727898) -- (6.42610824187856,-6.120215842049037) -- (6.444627071362846,-5.96401635806264) -- (6.459994257155924,-5.807329465083296) -- (6.472163017235233,-5.650777327393919) -- (6.481111937870867,-5.494933615127919) -- (6.486843732824407,-5.34032690328663) -- (6.48938393629847,-5.187443857892029) -- (6.488779529636966,-5.036732209274714) -- (6.485097501776048,-4.888603512497135) -- (6.478423343445778,-4.743435694912127) -- (6.468859475122507,-4.601575390856683) -- (6.456523608731934,-4.463340063481018) -- (6.441547043102903,-4.329019913712878) -- (6.424072893171883,-4.198879576357146) -- (6.404254252938162,-4.07315960333069) -- (6.38225229216975,-3.952077734032504) -- (6.358234286859982,-3.8358299528491027) -- (6.332371583434836,-3.7245913337951895) -- (6.304837496710934,-3.6185166722895934) -- (6.2758051416042875,-3.5177409040664775) -- (6.245445198589709,-3.422379311221816) -- (6.2139236129109605,-3.3325275153951357) -- (6.1813992275415846,-3.2482612580865307) -- (6.148021349896454,-3.1696359681089494) -- (6.113927252294032,-3.0966861161757415) -- (6.079239606169323,-3.029424356623485) -- (6.044063850037541,-2.96784045627007) -- (6.008485491208483,-2.911900010408066) -- (5.972567341251598,-2.861542945933344) -- (5.93634668521178,-2.8166818116089813) -- (5.89983238457585,-2.7771998554644206) -- (5.863001913989753,-2.742948889329914) -- (5.825798331726459,-2.713746940506223) -- (5.788127183904567,-2.6893756905695962) -- (5.749853342457624,-2.6695777013120114) -- (5.710797776854134,-2.6540534278166885) -- (5.670734259568291,-2.6424580186688726) -- (5.629386005301406,-2.634397903301883) -- (5.586422243954039,-2.629427166478435) -- (5.541454727348851,-2.6270437099072304) -- (5.494034169704145,-2.626685200994811) -- (5.4436466218581225,-2.6277248087326934) -- (5.38970977924384,-2.62946672671976) -- (5.33156922361488,-2.6311414833199254) -- (5.268494598521722,-2.631901038955075) -- (5.199675718538819,-2.630813670533267) -- (5.12421861224238,-2.6268586430122043) -- (5.041141498938862,-2.618920668097978) -- (4.949370699144167,-2.60578415007908) -- (4.847736478813539,-2.5861272187956827) -- (4.734968827322178,-2.5585155497441883) -- (4.609693169196546,-2.521395971317048) -- (4.470426009596393,-2.473089859177851) -- (4.315570513547478,-2.4117863177716807) -- (4.470426009596393,-2.473089859177851) -- (4.609693169196546,-2.521395971317048) -- (4.734968827322178,-2.5585155497441883) -- (4.847736478813539,-2.5861272187956827) -- (4.949370699144167,-2.60578415007908) -- (5.041141498938862,-2.618920668097978) -- (5.12421861224238,-2.6268586430122043) -- (5.199675718538819,-2.630813670533267) -- (5.268494598521722,-2.631901038955075) -- (5.33156922361488,-2.6311414833199254) -- (5.38970977924384,-2.62946672671976) -- (5.4436466218581225,-2.6277248087326934) -- (5.494034169704145,-2.626685200994811) -- (5.541454727348851,-2.6270437099072304) -- (5.586422243954039,-2.629427166478435) -- (5.629386005301406,-2.634397903301883) -- (5.670734259568291,-2.6424580186688726) -- (5.710797776854134,-2.6540534278166885) -- (5.749853342457624,-2.6695777013120114) -- (5.788127183904567,-2.6893756905695962) -- (5.825798331726459,-2.713746940506223) -- (5.863001913989753,-2.742948889329914) -- (5.89983238457585,-2.7771998554644206) -- (5.93634668521178,-2.8166818116089813) -- (5.972567341251598,-2.861542945933344) -- (6.008485491208483,-2.911900010408066) -- (6.044063850037541,-2.96784045627007) -- (6.079239606169323,-3.029424356623485) -- (6.113927252294032,-3.0966861161757415) -- (6.148021349896454,-3.1696359681089494) -- (6.1813992275415846,-3.2482612580865307) -- (6.2139236129109605,-3.3325275153951357) -- (6.245445198589709,-3.422379311221816) -- (6.2758051416042875,-3.5177409040664775) -- (6.304837496710934,-3.6185166722895934) -- (6.332371583434836,-3.7245913337951895) -- (6.358234286859982,-3.8358299528491027) -- (6.38225229216975,-3.952077734032504) -- (6.404254252938162,-4.07315960333069) -- (6.424072893171883,-4.198879576357146) -- (6.441547043102903,-4.329019913712878) -- (6.456523608731934,-4.463340063481018) -- (6.468859475122507,-4.601575390856683) -- (6.478423343445778,-4.743435694912127) -- (6.485097501776048,-4.888603512497135) -- (6.488779529636966,-5.036732209274714) -- (6.48938393629847,-5.187443857892029) -- (6.486843732824407,-5.34032690328663) -- (6.481111937870867,-5.494933615127919) -- (6.472163017235233,-5.650777327393919) -- (6.459994257155924,-5.807329465083296) -- (6.444627071362846,-5.96401635806264) -- (6.42610824187856,-6.120215842049037) -- (6.404511093570138,-6.275253646727898) -- (6.379936602451742,-6.42839957100605) -- (6.352514437737893,-6.578863445400117) -- (6.32240393764747,-6.725790881560142) -- (6.289795018958379,-6.868258808928516) -- (6.2549090203129625,-7.00527079853413) -- (6.217999479274099,-7.135752173921837) -- (6.179352843132004,-7.258544909217158) -- (6.139289113461748,-7.372402314326271) -- (6.098162424431478,-7.475983507271256) -- (6.056361554861333,-7.567847673660623);
\draw[line width=0.8pt] (4.315570513547478,-2.4117863177716807) -- (4.188480446139957,-2.3621750747296275) -- (4.0676959462683895,-2.322352955696638) -- (3.952662685972053,-2.2917480204938876) -- (3.8428571382991707,-2.269789033015134) -- (3.737785355164979,-2.2559067707982625) -- (3.6369817883777102,-2.2495353086095307) -- (3.5400081538325074,-2.250113276040542) -- (3.4464523388732373,-2.2570850891179197) -- (3.355927352822246,-2.2699021559256947) -- (3.2680703206780115,-2.288024056240409) -- (3.1825415199807283,-2.3109196951789275) -- (3.0990234608458094,-2.3380684308589657) -- (3.0172200091653028,-2.368961176072327) -- (2.9368555529772316,-2.403101473970853) -- (2.857674212002845,-2.4400065477650847) -- (2.779439090351801,-2.4792083244356387) -- (2.7019315723952584,-2.5202544324572917) -- (2.6249506618068867,-2.5627091735357825) -- (2.548312363771801,-2.606154468357318) -- (2.4718491103634146,-2.650190776350802) -- (2.3954092290882048,-2.6944379894627657) -- (2.318856454598406,-2.738536299945019) -- (2.2420694835726125,-2.7821470421550036) -- (2.164941572764311,-2.8249535083688753) -- (2.087380180218317,-2.8666617386072786) -- (2.00930664965515,-2.9070012844738464) -- (1.9306559380233068,-2.9457259470064074) -- (1.8513763862194696,-2.9826144885409054) -- (1.7714295329766236,-3.017471318588038) -- (1.690789971920099,-3.0501271537225927) -- (1.6094452517915254,-3.080439651485508) -- (1.527395819840715,-3.108294018298642) -- (1.4446550083854537,-3.1336035913922564) -- (1.3612490645392188,-3.1563103947452023) -- (1.2772172231068128,-3.176385669037835) -- (1.1926118226479157,-3.1938303756176243) -- (1.1074984647085562,-3.2086756744774867) -- (1.0219562162205027,-3.2209833762468305) -- (0.9360778550685742,-3.230846368195305) -- (0.8499701588258661,-3.2383890142492695) -- (0.763754236656899,-3.243767529020972) -- (0.6775659043886841,-3.2471703258504365) -- (0.5915561027497085,-3.248818338860071) -- (0.5058913587768378,-3.24896531902198) -- (0.4207542903901407,-3.2478981042379886) -- (0.33634415413562874,-3.245936863432387) -- (0.2528774360959178,-3.2434353146573773) -- (0.17058848596880616,-3.240780917211242) -- (0.0897301943137738,-3.238395037769214) -- (0.010574712966398558,-3.236733090527069) -- (-0.0665857813793079,-3.2362846513574235) -- (-0.14143828042064505,-3.237573545978748) -- (-0.2136480943280505,-3.241157912137089) -- (-0.2828579586590072,-3.2476302358005094) -- (-0.3486870981040324,-3.2576173613662345) -- (-0.41073024706474714,-3.2717804758805134) -- (-0.4685566270640278,-3.2908150672711916) -- (-0.5217088809882374,-3.315450856592997) -- (-0.5697019641615401,-3.3464517042855384) -- (-0.612021992252294,-3.384615490444011) -- (-0.6481250460115283,-3.430773969102622) -- (-0.6774359328434987,-3.4857925965307226) -- (-0.6993469052083259,-3.550570333541654) -- (-0.7132163358567141,-3.626039421814307) -- (-0.6993469052083259,-3.550570333541654) -- (-0.6774359328434987,-3.4857925965307226) -- (-0.6481250460115283,-3.430773969102622) -- (-0.612021992252294,-3.384615490444011) -- (-0.5697019641615401,-3.3464517042855384) -- (-0.5217088809882374,-3.315450856592997) -- (-0.4685566270640278,-3.2908150672711916) -- (-0.41073024706474714,-3.2717804758805134) -- (-0.3486870981040324,-3.2576173613662345) -- (-0.2828579586590072,-3.2476302358005094) -- (-0.2136480943280505,-3.241157912137089) -- (-0.14143828042064505,-3.237573545978748) -- (-0.0665857813793079,-3.2362846513574235) -- (0.010574712966398558,-3.236733090527069) -- (0.0897301943137738,-3.238395037769214) -- (0.17058848596880616,-3.240780917211242) -- (0.2528774360959178,-3.2434353146573773) -- (0.33634415413562874,-3.245936863432387) -- (0.4207542903901407,-3.2478981042379886) -- (0.5058913587768378,-3.24896531902198) -- (0.5915561027497085,-3.248818338860071) -- (0.6775659043886841,-3.2471703258504365) -- (0.763754236656899,-3.243767529020972) -- (0.8499701588258661,-3.2383890142492695) -- (0.9360778550685742,-3.230846368195305) -- (1.0219562162205027,-3.2209833762468305) -- (1.1074984647085562,-3.2086756744774867) -- (1.1926118226479157,-3.1938303756176243) -- (1.2772172231068128,-3.176385669037835) -- (1.3612490645392188,-3.1563103947452023) -- (1.4446550083854537,-3.1336035913922564) -- (1.527395819840715,-3.108294018298642) -- (1.6094452517915254,-3.080439651485508) -- (1.690789971920099,-3.0501271537225927) -- (1.7714295329766236,-3.017471318588038) -- (1.8513763862194696,-2.9826144885409054) -- (1.9306559380233068,-2.9457259470064074) -- (2.00930664965515,-2.9070012844738464) -- (2.087380180218317,-2.8666617386072786) -- (2.164941572764311,-2.8249535083688753) -- (2.2420694835726125,-2.7821470421550036) -- (2.318856454598406,-2.738536299945019) -- (2.3954092290882048,-2.6944379894627657) -- (2.4718491103634146,-2.650190776350802) -- (2.548312363771801,-2.606154468357318) -- (2.6249506618068867,-2.5627091735357825) -- (2.7019315723952584,-2.5202544324572917) -- (2.779439090351801,-2.4792083244356387) -- (2.857674212002845,-2.4400065477650847) -- (2.9368555529772316,-2.403101473970853) -- (3.0172200091653028,-2.368961176072327) -- (3.0990234608458094,-2.3380684308589657) -- (3.1825415199807283,-2.3109196951789275) -- (3.2680703206780115,-2.288024056240409) -- (3.355927352822246,-2.2699021559256947) -- (3.4464523388732373,-2.2570850891179197) -- (3.5400081538325074,-2.250113276040542) -- (3.6369817883777102,-2.2495353086095307) -- (3.737785355164979,-2.2559067707982625) -- (3.8428571382991707,-2.269789033015134) -- (3.952662685972053,-2.2917480204938876) -- (4.0676959462683895,-2.322352955696638) -- (4.188480446139957,-2.3621750747296275) -- (4.315570513547478,-2.4117863177716807);
\draw [line width=0.8pt] (-5.751684463967369,-1.357630463419289) circle (2.5);
\draw [line width=0.8pt] (10.738216482070609,-1.250300765911508) circle (2.5);
\draw [->,line width=0.8pt] (0.27284947541322513,-0.8596626280033284) -- (0.27284947541322513,-2.888048606106838);
\draw [->,line width=0.8pt] (4.845937862410231,-2.334852430260426) -- (4.809058117353804,-0.6752639027211913);
\draw [->,line width=0.8pt] (-3.6732832456245146,0.910565134705189) -- (-1.681777012577431,1.8694385061723024);
\draw [->,line width=0.8pt] (8.68143134827869,-3.5150042720661046) -- (6.8005643504008875,-4.5845168787025);
\draw [shift={(8.284903017234004,4.099164115573706)},line width=0.8pt]  plot[domain=-2.6156804469145953:2.723683447666531,variable=\t]({1.*1.5030529384896822*cos(\t r)+0.*1.5030529384896822*sin(\t r)},{0.*1.5030529384896822*cos(\t r)+1.*1.5030529384896822*sin(\t r)});
\draw [->,line width=0.8pt] (6.911203585570172,4.709178875517215) -- (6.616165625118753,4.155982699670803);
\draw (8.017595937262996,4.693577600799352) node[anchor=north west] {$f$};
\draw (-3.7470427357373675,2.4439131523572783) node[anchor=north west] {$\varphi_1$};
\draw (7.943836447150142,-3.678124527009676) node[anchor=north west] {$\varphi_2$};
\draw (-0.7,-1.0596626280033283) node[anchor=north west] {$F$};
\draw (4.956577097579516,-0.9859031378904734) node[anchor=north west] {$h$};
\draw (1.9324380029524635,3.0071093282036898) node[anchor=north west] {$\Omega$};
\draw (1.895558257896036,-4.23699789847679) node[anchor=north west] {$F(\Omega)$};
\end{tikzpicture}
\caption{Stoilow decomposition and associated Riemann mappings when $\Omega' = \Omega$.}\label{figStoilow}
\end{figure}

For Theorem  \ref{theoSobolevStabilityOfDomains} one also needs to develope new composition  results for functions spaces (see e.g. Lemma \ref{lemBesovAdmissibleSpace} or Corollary \ref{lemSobolevAdmissibleSpace}) and, in addition, show  $h  \in W^{s+1,p}\bigl(F(\Omega)\bigr)$. That, on the other hand,
will be  a consequence of Theorem \ref{theoSobolevRiemann}.


 Note that 
 in Theorem \ref{theoSobolevRiemann}, if  $\Omega$ is a Lipschitz-domain with 
 $\varphi \in W^{s+1,p}(\D)$, the condition ``$\varphi^{-1} \in W^{s+1,p}(\Omega)$'' is equivalent to the condition ``$\varphi$ is bi-Lipschitz'', see \rf{eqInvarianceUnderInversionSOBO} in Corollary \ref{lemSobolevAdmissibleSpace} below. 
{To underline the optimality of Theorem \ref{theoSobolevStabilityOfDomains}} it is useful to combine the above results in the following form.
 
 \begin{theorem}\label{theoSobolevRiemannKaksi}
Let  $\, s > 0  $ and $1 < p < \infty$  with $sp>2$. If $\Omega$ is a simply  connected, bounded domain and $g:\D \to \Omega$  is a $\mu$-quasiconformal mapping, then the following are equivalent:
\begin{enumerate}
\item $\Omega$ is a $B^{s+1-\frac1p}_{p,p}$-domain  and $\mu\in W^{s,p}(\Omega)$.
\item $g$ is bi-Lipschitz and $g \in W^{s+1,p}(\D)$.
\end{enumerate}
\end{theorem}
\noindent   In particular, {the above  bi-Lipschitz condition} is required to exclude boundary cusps. 
\medskip

Theorem \ref{theoSobolevStabilityOfDomains} has  natural counterparts in the context of 
  ``supercritical'' Triebel-Lizorkin functions, at least for fractional smoothness parameters,  $s \notin \N$:
\begin{theorem}\label{theoTriebelStabilityOfDomains}
Let $ s \in \R_+ \setminus \N$, let  $1 < p < \infty$ with $sp>2$, and let  $1<q<\infty$. Suppose $\Omega, \Omega'$ are simply connected, 
 bounded $B^{s+1-\frac1p}_{p,p}$-domains and $f:\Omega \to\Omega'$ is a quasiconformal mapping, with $\mu_f \in F^s_{p,q}(\Omega)$. Then $f\in F^{s+1}_{p,q}(\Omega)$.   
\end{theorem}

The argument in the Triebel-Lizorkin spaces  follows a similar strategy as in the  Sobolev setting. 
However, note that our approach fails for $s \in \N$, see Remark \ref{blablabla}.
For details and proof of Theorem \ref{theoTriebelStabilityOfDomains} see Section \ref{secSobolev}.
On the other hand, for conformal mappings Theorem \ref{theoSobolevRiemann} 
 has a version  in the Triebel-Lizorkin setting valid even for integer $s$, see Theorem \ref{theoTriebelRiemann} below. Similarly, our results work also for finitely connected Besov domains, see Sections \ref{secRiemann} and \ref{secSobolev} below.


\medskip


Let us then turn to the other aspect of global regularity and the study of principal mappings. Here our main goal is to show that dilatation $\mu \in  F^s_{p,q}(\Omega)$
implies for the principal mapping the regularity $f\in F^{s+1}_{p,q}(\Omega)$ in the whole domain $\Omega$. However, the situation here is more subtle and accordingly the proof is much more involved.
\begin{theorem}\label{theoInvertBeltrami}
Let $0<s<1$, let $2<sp <\infty$ (see Figure \ref{figIndicesMuN12Conclusions}), let $1<q<\infty$ and let $\Omega$ be a bounded $B^{s+1-1/p}_{p,p}$-domain. Suppose  $\mu \in F^s_{p,q}(\Omega)$ with $\|\mu\|_{L^\infty}<1$ and $\supp(\mu)\subset\overline{\Omega}$. Then, the principal solution  $f$ to \eqref{eqBeltrami} is in the space $F^{s+1}_{p,q}(\Omega)$.
\end{theorem}

We expect also this theorem to be sharp, in particular in view of the results in \cite{TolsaSharp} for the case of smoothness one. The counterpart of this result in the Sobolev scale can be found in \cite{PratsQuasiconformal}. The proof follows the scheme of Iwaniec for $VMO$ Beltrami coefficients adapted by Cruz, Mateu and Orobitg for the domain-restricted setting.
The key idea is to reduce it to three steps using a Fredholm theory argument. First, one needs to show that the Beurling transform (see \ref{eqBeurling} for its definition) restricted to $\Omega$, that is $\Beurling_\Omega=\chi_\Omega \Beurling(\chi_\Omega \cdot)$ is bounded in $F^s_{p,q}(\Omega)$ assuming the Besov regularity of the boundary. For the cases $q\in \{2,p\}$ this can be found in \cite{CruzTolsa}. 
{Next} we need to show the compactness of the commutator $[\mu,\Beurling_\Omega]$, which was studied in \cite{CruzMateuOrobitg}  for more regular domains, but the adaptation to our context is rather straight-forward, see Lemma \ref{lemCompactnessCommutator}.

The third step is to check the compactness  of the {\em Beurling reflection} $\mathcal{R}:=\chi_\Omega \Beurling\left(\chi_{\Omega^c} \Beurling (\chi_\Omega \,\cdot)\right)$.  In Proposition \ref{lemSmallNormCloseToBoundary} we show that not only is $\mathcal{R}$  compact in $F^{s}_{p,q}(\Omega)$, but it is in fact smoothing in the following sense:
$$\norm{\mathcal{R} f }_{\dot F^s_{p,q}(\Omega)} \lesssim_h \norm{f}_{\mathcal{C}^h(\Omega)}$$
for every $h>0$. To verify that this embedding holds we make use of several techniques,  including  the approximation of the boundary of the domain by straight lines, as Cruz and Tolsa introduced in \cite{CruzTolsa}, {which in turn uses the fact that the kernel of the Beurling transform is even.}  That allows us to replace the transform of the characteristic function of the domain at a given point by a sum of beta coefficients introduced by Dorronsoro in \cite{Dorronsoro}. We also use a recent expression of the kernel of the reflection obtained in \cite{PratsQuasiconformal} (see Section \ref{secReflection}) and the techniques on chains of Whitney cubes introduced in \cite{PratsTolsa, PratsSaksman}.

\medskip

The first results on the global regularity of $f{|_{\Omega}}$, for a principal mapping $f$ with Beltrami coefficient supported on $\Omega$,
was obtained in \cite{MateuOrobitgVerdera}. In this work Mateu, Orobitg and Verdera  showed  that,  surprisingly,  given a $C^{1+\varepsilon}$-domain $\Omega$ , with $\varepsilon\in(0,1), $ and 
a Beltrami coefficient $\mu$ for which $ \mu |_{\Omega} \in C^\varepsilon (\Omega)$ and  $\mu |_{\C \setminus \Omega} = 0$, 
 the corresponding principal solution $f$ to \rf{eqBeltrami}
 is bilipschitz in all of $\C$,  in  spite of the possible discontinuity of $\mu$ at  $\partial \Omega$. In addition, $f |_{\Omega} \in C^{1+\varepsilon}(\Omega)$. {A key ingredient in their proof is again the even character of the Beurling transform kernel}.
 
 Later in \cite{CruzMateuOrobitg} it was shown that something can be said  about the Sobolev and Besov regularity as well for such domains. Namely, when $0<s<\varepsilon<1$ and $2 <ps<\infty$, if $\Omega$ is a $C^{1+\varepsilon}$-domain, $\norm{\mu}_{L^\infty}<1$, $\supp(\mu)\subset\overline{\Omega}$ and $f$ is a $\mu$-principal mapping, then
\begin{equation*}
\mu \in B^s_{p,p}(\Omega) \implies f \in B^{s+1}_{p,p}(\Omega),
\end{equation*}
and the same happens in the scale $F^s_{p,2}$. Note that for any interval $I$, and thus for any parametrization of the boundary,
$$C^{1+\varepsilon}(I) \subset C^{s+1-\frac{1}{p}}(I) \subset B^{s+1-\frac{1}{p}}_{p,p}(I)\subset C^{s+1-\frac{2}{p}}(I),$$ 
(we used the embeddings in  \cite[Section 2.7]{TriebelTheory} for the last step)
that is, the assumptions in Theorem \ref{theoInvertBeltrami} are strictly weaker than the conditions in \cite{CruzMateuOrobitg}. 

One may note that in all our results we assume $sp>2$. This comes from that fact that in the (subcritical) range $sp<2$ derivatives $\overline{\partial} f$ and $\partial f$ do not even locally need to gain the same  regularity  as the dilatation has. The underlying reason for this lies in the fact   the Neumann series for the solution contains product terms whose regularity deteriorates because the Sobolev (or Triebel-Lizorkin) space in question is {no longer an algebra.} We refer the reader to \cite{ClopFaracoMateuOrobitgZhong,PratsBeltrami} (see especially the examples in  \cite[p. 205]{ClopFaracoMateuOrobitgZhong})  for basic regularity results in the subcritical  case.  One also expects a small loss for the critical case according to the results in \cite{ClopFaracoMateuOrobitgZhong,BaisonClopOrobitg}.  {Note also that one may consider regularity results also for  $\R$-linear Beltrami equations, see e.g. \cite[Chapter 15]{AstalaIwaniecMartin}. However, in that case one may have smooth solutions without the coefficients being smooth, so the exact equivalence  as in our  theorems need not  hold even locally.}

In turn, Theorem \ref{theoInvertBeltrami} is a fractional counterpart to \cite[Theorem 1.1]{PratsQuasiconformal}, where the Sobolev spaces $W^{s,p}$ with $s\in\N$ and $p>2$ where dealt with. Taking a look at Figure \ref{figIndicesMuN12Conclusions}, it seems natural to conjecture the following:

\begin{conjecture}[{see \cite[Conclusions]{PratsTesi}}]\label{conjInvertBeltrami}
For $s\in \R$, we write $s=n_s+ \{s\} $, with $n_s \in \Z$ and $0<\{s\}\leq 1$. Let  $\Omega$ be a bounded $B^{s+1-1/p}_{p,p}$-domain for some $2 <p\{s\}<\infty$ (see Figure \ref{figIndicesMuN12Conclusions}) and let $\mu\in W^{s,p}(\Omega)$ with $\norm{\mu}_{L^\infty}<1$ and $\supp(\mu)\subset\overline{\Omega}$. Then, the principal solution  $f$ to \rf{eqBeltrami} is in the space $W^{s+1,p}(\Omega)$.
\end{conjecture}
\noindent Note that $n_s=s-1$ when $s$ is an integer, while  for  $s\notin \Z$ it is the integer part. In particular, $n_s$ is always smaller than $s$.

By the Sobolev embedding (combine \cite[Section 2.7]{TriebelTheory} with appropriate extension theorems), restricting ourselves to the indices $p\{s\}>2$ implies that if $f\in W^{s,p}(\Omega)$ then all its weak derivatives up to order $n_s$ are continuous, and therefore, ordinary derivatives. The same holds for the parameterizations of the boundaries. 

The authors wonder whether Theorem \ref{theoInvertBeltrami} will be valid only in the range covered by Conjecture \ref{conjInvertBeltrami}, where  the  techniques developed so far might apply, or whether  the statement holds true in  the whole supercritical region $sp>2$ like in Theorem \ref{theoSobolevStabilityOfDomains}.


\begin{figure}[ht]
 \centering
\begin{subfigure}{0.4 \textwidth}
 \centering{\includegraphics[width=\textwidth]{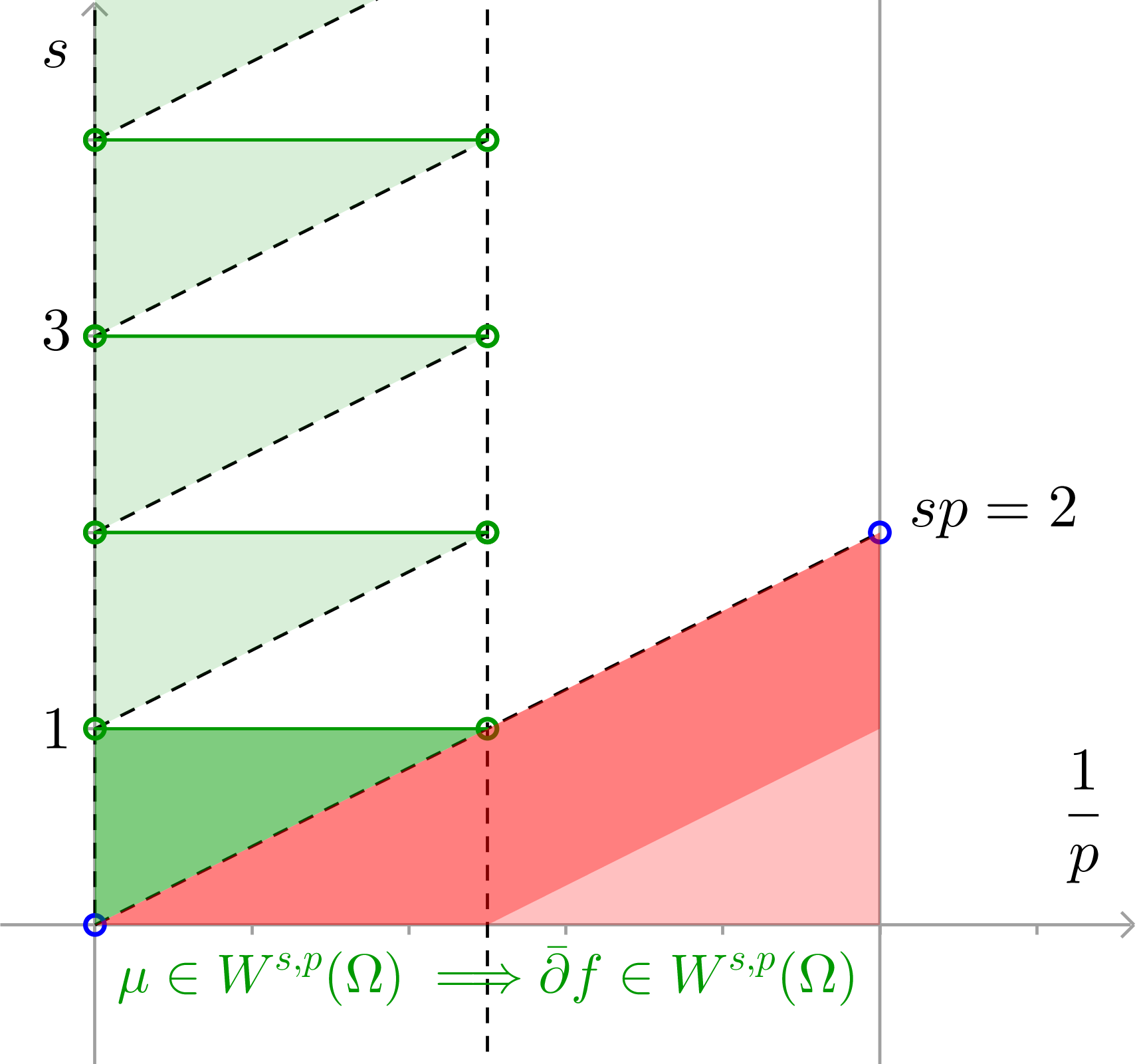}}
\end{subfigure}
\caption{Light green represents the  results conjectured in \ref{conjInvertBeltrami}, dark  green stands for the known
 results, red is for the range of indices where the result is known to be false, and light red where it is conjectured to fail: In case $s\in\N$ (see \cite{PratsQuasiconformal}) or $s<1$ (see Theorem \ref{theoInvertBeltrami} above),  Conjecture \ref{conjInvertBeltrami} holds. In case $\frac2p-1<s<\frac2p$ there are radial stretchings which are counterexamples, see \cite{ClopFaracoMateuOrobitgZhong}.}\label{figIndicesMuN12Conclusions}
\end{figure}

{
\subsection{Structure of the paper}\label{subse:structure}
In Section \ref{secPreliminaries} we provide a {summary} of notation, the definitions of domains, and function spaces used throughout the paper, and finally we collect some {known and also a couple of new auxiliary results on} functions spaces, paying special attention to their behavior under multiplication, composition and inversion. 

We devote Section \ref{secRiemann} to the study of the Riemann mapping. {We especially  establish Theorem \ref{theoTriebelRiemann} which implies  Theorem \ref{theoSobolevRiemann}, but also covers the whole Triebel-Lizorkin scale, and relaxes the condition that  $\varphi^{-1}$ is in the function space by the assumption that $\varphi$ is bi-Lipschitz. After that, we extend the result to finitely connected domains in Corollary \ref{finiteconn} and use suitable properties of the function spaces to deduce Theorem \ref{theoSobolevRiemann}. Here we need to deal with a subtlety: The functions $f\in W^{s,p}(\Omega)$ are defined as  restrictions of Triebel-Lizorkin functions in the complex plane, see   Definition  \ref{defRestrictionNorm}. However, in the case $s\in \N$, it is quite natural to work with intrinsic definitions for Sobolev spaces instead, see Definition \ref{defSobolevNonHomogeneous}. These function spaces are known to coincide at least if the domain is Lipschitz, but we don't have this information a-priori. For that purpose we establish  Proposition \ref{intrinsic} which verifies that a conformal map $\varphi$ with finite norms $\| D^{k} \varphi \|_{L^p(\D)}$ and $\| D^{k} \varphi^{-1} \|_{L^p(\Omega)} $ must be bi-Lipschitz, granting therefore that our domain is Lipschitz.}

In Section \ref{secSobolev} we prove Theorems  \ref{theoSobolevStabilityOfDomains}, \ref{theoSobolevRiemannKaksi} and \ref{theoTriebelStabilityOfDomains}. To complete {this  we} just need to check that the traces of Sobolev spaces $W^{s,p}(\Omega)$ and Triebel-Lizorkin spaces $F^{s}_{p,q}(\Omega)$ in Besov domains are precisely the Besov spaces {defined in} Section \ref{BesovDomains}. Once this is settled, the aforementioned theorems are deduced from our results in Section  \ref{secRiemann} by means of Stoilow factorization.

{Finally we prove Theorem \ref{theoInvertBeltrami} in Section \ref{secPrincipal}.  The proof is done by using Fredholm theory to invert the operator $(I-\mu \Beurling)(\chi_\Omega \cdot)$ in $F^s_{p,q}(\Omega)$, and this is reduced in  Section \ref{secOutlineBeltrami}
 to establishing several key auxiliary facts: the polynomial growth of the operator norm of the truncated iterates of the Beurling transform $\Beurling^m(\chi_\Omega \cdot)$ on $F^s_{p,q}(\Omega)$,  and  the compactness of the commutator $[\mu, \Beurling(\chi_\Omega\cdot)]$ together with that of the ``reflection'' $\mathcal{R}_m g:=\chi_\Omega \Beurling\left(\chi_{\Omega^c} \Beurling^{m}(\chi_\Omega \, \cdot)\right)$ on $F^s_{p,q}(\Omega)$. The aim of the latter subsections is to verify these key facts. } For that purpose we first  define  in Section \ref{secBetas} suitable beta coefficients that serve to measure the Besov smoothness of functions {a la}  Dorronsoro. {After that} we recall some results from the literature that explain how these betas can appear when working with truncated Beurling transforms and their relation to the Besov character of the boundary of the domain, see lemmata \ref{lemNormalVectorBetas}, \ref{lemBoundForBeurling} and \ref{lemSummingTheBetas}. These results {will} be crucial in the subsequent sections. 

In Section \ref{secIterates} we turn our attention to the boundedness of the truncated iterates of the Beurling transform. Here we already make use of the results of the previous section to obtain a polynomial growth (quadratic, in fact) on the number of iterates. In Section \ref{secCommutator} we deduce the compactness of the commutator, which is quite straightforward after the Cruz and Tolsa results. 

The most technical step turns out to be the compactness of the Beurling reflection which spans through sections \ref{secReflection} and \ref{secMeyer}. This boils down to the estimate $\norm{\mathcal{R}_m f }_{\dot F^s_{p,q}(\Omega)} \lesssim_h \norm{f}_{\mathcal{C}^h(\Omega)}
$. The proof follows a discretization argument, expressing the norm as a sum on Whitney cubes, and then substituting the function by its mean on each cube. The difference gives rise to a local part and a nonlocal part which are treated differently. The local part is dealt with by using an explicit expression of the kernel of the operator, and then applying lemmata \ref{lemBoundForBeurling} and \ref{lemSummingTheBetas}. For the nonlocal part we need to use an equivalent expression for the kernel obtained in \cite{PratsQuasiconformal}. Again we will apply the same lemmata to estimate part of the terms arising from this expression. However, some other terms need {considerable extra work}. In particular, they are handled by using discretization techniques typical {in connection with} uniform domains, such as chains of Whitney cubes which play the role of Harnack chains, and the {use} of polynomial approximations, namely Meyer's approximating polynomials, which allow us to {apply} the Poincar\'e inequality iteratively. All this techniques are quite specific for these nonlocal terms, {whence} they are introduced in a separate section, namely Section \ref{secMeyer}, where the proof of the compactness of the Beurling reflection is {finally} completed.}

\section{Preliminaries}\label{secPreliminaries}
In this section we provide a handbook of notation and then we collect some results from the function spaces.

\subsection{Notation}
Throughout this paper we  write $C$ for constants which may change from one occurrence to the next. If we want to make clear in which parameters $C$ depends, we will add them as a subindex. In the same spirit, when comparing two quantities $x_1$ and $x_2$, we may write $x_1\lesssim x_2$ instead of $x_1\leq C x_2$, and $x_1\lesssim_{p_{1},\dots, p_{j}} x_2$ for $x_1\leq C_{p_{1},\dots, p_{j}} x_2$, meaning that the constant depends on all these parameters.

For $1\leq p\leq \infty$ we denote by $p'$  the  H\"older conjugate, that is $\frac{1}{p}+\frac{1}{p'}=1$.
Given $x\in \R^d$ and $r>0$, we write $B(x,r)$ or $B_r(x)$ for the open ball centered at $x$ with radius $r$ and $Q(x,r)$ for the open cube centered at $x$ with sides parallel to the axis and side-length $2r$. For any cube $Q$, we write $\ell(Q)$ for its side-length, and $rQ$ will stand for the cube with the same center but enlarged by a factor $r$. We will use the same notation for one dimensional  balls and cubes, that is, intervals. 

\begin{definition}\label{defLipschitzDomain}
Let $\delta,R>0$, $d\geq 2$. We say that a domain $\Omega\subset \R^d$ is a $(\delta,R)$-Lipschitz domain (or just a Lipschitz domain when the constants are not important) if for every point $z\in\partial\Omega$, there exists a cube $\mathcal{Q}=Q(0,R)$ and a Lipschitz function $A_z:\R^{d-1}\to\R$ supported in $[-4R,4R]^{d-1}$ such that $\norm{DA_z}_{L^\infty}\leq \delta$ and,  possibly after a translation that sends $z$ to the origin and a rotation, we have that
$$\mathcal{Q}\cap \Omega = \{(x,y)\in \mathcal{Q}: y>A_z(x)\}.$$

If $d=1$ we say that $\Omega\subset \R$ is a Lipschitz domain if $\Omega$ is an open interval.
\end{definition}

When dealing with line integrals in the complex plane\footnote{We identify $\R^2$ and $\C$ when appropriate.}, we will write $dz$ for the form $dx+i\,dy$ and analogously $d\bar{z}=dx-i\,dy$, where $z=x+i\,y$. When integrating a function with respect to  the Lebesgue measure of a complex variable $z$ we will always use $dm(z)$ to avoid confusion,  or simply $dm$. We adopt the traditional Wirtinger notation for derivatives, that is, given any $\phi\in C^\infty_c(\C)$, then 
$\partial \phi (z):=\frac{\partial \phi}{\partial z}(z)=\frac12(\partial_x\phi-i\,\partial_y\phi) (z)$
 and 
 $\bar \partial \phi (z):=\frac{\partial\phi}{\partial \bar z}(z)=\frac12(\partial_x\phi+i\,\partial_y\phi) (z).$

For any measurable set $A$ and any measurable function $f$,  the mean  over the set $A$ is denoted by $f_A= \fint_A f \, dm$.
We let $\N=\{1,2,\ldots\}$ be the set of natural numbers, and set  $\N_0=\N \cup \{0\}$.


The principal solution to \rf{eqBeltrami} can be found using the Beurling transform,  defined for $\varphi \in C^\infty
_c(\C)$ by
\begin{equation}\label{eqBeurling}
\Beurling \varphi (z) = -\lim_{\varepsilon \to 0} \frac{1}{\pi} \int_{|w-z|>\varepsilon} \frac{\varphi(w)}{(z-w)^2} \, dm(w),
\end{equation}
see \cite[Section 5]{AstalaIwaniecMartin}).
The Beurling transform  extends to a bounded operator in $L^p$ for every $1<p<\infty$ with $\norm{\Beurling}_{L^2\to L^2}=1$. Thus, $I-\mu\Beurling$ is invertible in $L^2$, and one obtains that $\bar\partial f:=(I-\mu\Beurling)^{-1}(\mu)$ is a well-defined compactly supported $L^2$ function, whenever $\mu \in L^{\infty}$ is compactly supported with $\| \mu \|_{\infty} < 1$. Finally, $f=z+ \Cauchy [(I-\mu\Beurling)^{-1}(\mu)](z)$, where the Cauchy transform $\Cauchy$ is defined by 
$$\Cauchy \varphi (z) =  \frac{1}{\pi} \int_{\C} \frac{\varphi(w)}{z-w} \, dm(w)$$
for every $\varphi \in \mathcal{S}$. The Cauchy transform of the compactly supported function  $(I-\mu\Beurling)^{-1}(\mu)$ is well defined even pointwise. For this and further information see e.g. \cite{AstalaIwaniecMartin}.

We will denote by $\T:=\partial \D$ the boundary of the  unit circle.

\subsection{Definitions of  function spaces}
We start by recalling   the homogeneous H\"older-Zygmund seminorm:
\begin{definition}
Given an open set $U\subset\R^d$, and $0<s<1$, we say that 
$f\in\dot{\mathcal{C}}^s(U)$ if
$$\norm{f}_{\dot{\mathcal{C}}^s(U)}:=\sup_{x, y\in U} \frac{|f(x)-f(y)|}{|x-y|^s}<\infty.$$
For $s=1$ we replace $|f(x)-f(y)|$ by $|f(x)-2f(\frac{x+y}{2}) +f(y)|$ and take the supremum over  those  $x,y\in U$ for which  also the midpoint $(x+y)/2$ lies in $U$.

 If   $s\in(k, k+1]$ where $k\in\N, $ we say that $f\in\dot{\mathcal{C}}^s(U)$ if
 $$\norm{f}_{\dot{\mathcal{C}}^s(U)}:=\norm{D^k f}_{\dot{\mathcal{C}}^{s-k}(U)}<\infty.$$
\end{definition}
One can define Banach spaces of functions modulo polynomials using the previous seminorms. However, the standard  non-homogeneous H\"older-Zygmund spaces are more suitable for our purposes:
\begin{definition}
For $0<s<\infty$, we say that  $f\in \mathcal{C}^{s}(U)$ if $f\in L^\infty \cap \dot{\mathcal{C}}^s(U)$. We define the norm
$$\norm{f}_{\mathcal{C}^{s}(U)}:=\norm{f}_{L^\infty(U)}+\norm{f}_{\dot{\mathcal{C}}^s(U)}.$$
\end{definition}
\noindent Note, in particular that  for Lipschitz domains  $C^1(U)\subset Lip(U)\subset \dot{\mathcal{C}}^1(U)$, where $Lip$ stands for the standard Lipschitz continuous functions. 
Moreover, if $s=k+\alpha$ with $k\in\N_0$ and $\alpha\in (0,1)$, then we have the coincidence
$\mathcal{C}^s(U)=C^{k,\alpha}(U)$ with the classical H\"older spaces.

The classical Sobolev spaces are defined analogously:
 \begin{definition}\label{defSobolevNonHomogeneous}
Given $s\in\N$ and $1\leq p \leq \infty$, we say that a locally integrable  function $f$ belongs to the space 
$ {W}^{s,p}(U)$ if 
$$
\norm{f}_{{W}^{s,p}(U)}:= \sum_{|\alpha|\leq s}\norm{D^\alpha f}_{L^p(U)}<\infty,
$$
where the derivatives are understood in the distributional sense. 
We say that $f\in W^{s,p}_{loc}(U)$ if $f_{|V}\in W^{s,p}(V)$ for every open set $V$ contained in a compact subset of $U$.
\end{definition}

For some function spaces  such intrinsic definitions are not always easily formulated. Thus, one introduces the following general definition:
\begin{definition}\label{defRestrictionNorm}
Let $X$ be a Banach space of complex-valued functions in $\R^d$. Given an open set $U\subset \R^d$, we say that a measurable function $f:U \to \C$ belongs to $X(U)$ if
$$\norm{f}_{X(U)}:=\inf_{F|_U\equiv f} \norm{F}_X<\infty.$$
\end{definition}
\noindent In regular situations, such as for the Lipschitz domains studied in the present paper,  intrinsic definitions will coincide with the above definition via  restrictions since one may construct suitable extension operators. For the classical Sobolev spaces this is done e.g. in  \cite{Jones}. We refer to  \cite{Shvartsman,KoskelaRajalaZhang} for  extension theorems on even worse domains. In any case, our function spaces (Sobolev, Triebel, Besov)  in subdomains of $\R^d$ are defined via  Definition \ref{defRestrictionNorm} unless otherwise stated. Furthermore, our main interest lies is $B^{s}_{p,p}$-domains with $s>1+1/p$ (see Definition \ref{defBesovDomain} below), and they are automatically Lipschitz domains.

 Test functions are included in the classical Sobolev spaces, and from the Leibniz' rule (see \cite[Section 5.2.3]{Evans}) the space $W^{s,p}(\R^d)$ is closed under multiplication by $C^\infty_c$ functions, i.e., for $\varphi \in C^\infty_c$  and $f\in W^{s,p}$,
\begin{equation*}
\norm{\varphi f}_{W^{s,p}}\leq C_{\varphi} \norm{f}_{W^{s,p}}.
\end{equation*}
The same property holds   true also for all Besov and Triebel spaces defined below.




\begin{definition}
Let $0<s<\infty$, and let $\Omega$ be a bounded planar domain. We say that $\Omega$ is a $\mathcal{C}^{1+s}$-domain if $\partial\Omega$ is the finite disjoint union of Jordan curves  $\{\Gamma_j\}_{j=1}^M$ and  
for each $j\in\{1,\ldots, M\}$ there exists a bi-Lipschitz parametrization $\gamma_j: \T \to \Gamma_j$
 with $\gamma_j \in \mathcal{C}^{1+s}(\T)$.
\end{definition}

\begin{remark}
It is not difficult to see that every $\mathcal{C}^{1+s}$-domain, or even a  $C^1$-domain,  is a Lipschitz domain in the sense of Definition \ref{defLipschitzDomain}. Indeed, the requirement of $\gamma \in C^1(\T)$ being bi-Lipschitz forbids the spiralling of the curve.
\end{remark}

To end this introduction we give the definition of Besov and Triebel-Lizorkin spaces. For a complete treatment we refer the reader to \cite{TriebelTheory}.

Consider a family $\{\psi_j\}_{j=0}^\infty\subset C^\infty_c(\R^d)$ satisfying that $\sum_{j=0}^\infty \psi_j\equiv 1$ with $\supp \,\psi_0 \subset B(0,2)$, $\supp \,\psi_j \subset B(0,2^{j+1})\setminus B(0,2^{j-1})$ for $j\geq 1$,	and such that for all  $\vec{i}\in\N_0^d$ there exists a constant $c_{\vec{i}}$ with 
$$\norm{D^{\vec{i}} \psi_j}_\infty \leq \frac{c_{\vec{i}}}{2^{j |\vec{i}| }} \mbox{\,\,\, for every } j\geq 0,$$
where we use the standard multi-index notation.
 \begin{definition}

Let $F$ denote the Fourier transform.
Let $s \in \R$,  $1< p< \infty$, $1\leq q\leq\infty$. 
For any tempered distribution $f\in \mathcal{S}'(\R^d)$ 
we define the non-homogeneous Triebel-Lizorkin norm
\begin{equation*}
\norm{f}_{F^s_{p,q}}=\norm{\left\{2^{sj}F^{-1}\psi_j F f\right\}}_{L^p(\ell^q)}=\norm{\norm{\left\{2^{sj}F^{-1}\psi_j F f(\cdot)\right\}}_{\ell^q}}_{L^p},
\end{equation*}
and call  $F^s_{p,q}:=F^s_{p,q}(\R^d) \subset \mathcal{S}'(\R^d)$  the set of tempered distributions for which this norm is finite.
\end{definition}
These norms are equivalent for different choices of $\{\psi_j\}_j$.  
\begin{definition}\label{defFamous}
Let $s>0$ and $1<p<\infty$. Then the Bessel potential space $W^{s,p}$ is defined as  $W^{s,p}:=W^{s,p}(\R^d):=F^s_{p,2}(\R^d)$. For the Besov space $B^{s}_{p,p}$ we use the definition $B^{s}_{p,p}:=F^{s}_{p,p}$. In case $p=\infty$, one sets $B^{s}_{\infty,\infty}:=\mathcal{C}^s$. 
\end{definition}
The above definition of Besov spaces agrees with the fact that the \emph{diagonal} spaces $F^s_{p,p}$ coincide with the Besov spaces with the same indices, see \cite{TriebelTheory}.
 The thorough reader will note the extreme futility of using  diagonal Besov spaces here instead of diagonal Triebel-Lizorkin spaces when they are exactly the same,  but we couldn't stand the burden of being unfaithful to the tradition of the field.
Finally, we recall that the Besov or Triebel spaces on subdomains are obtained via Definition \ref{defRestrictionNorm}.

By   Sobolev's  Theorem (see \cite[Section 5.6]{Evans}),  there is a continuous embedding
\begin{equation}\label{eqEmbeddingIntoDiniSOBO}
\norm{f}_{L^\infty}\leq C\norm{f}_{W^{s,p}} \mbox{\quad\quad whenever } sp>d,
\end{equation}
and the same holds true if $W^{s,p}$ is replaced by $B^s_{p,p}$ or $F^s_{p,q}$ ($q\in(1,\infty)$ may be arbitrary).

When $f$ is a complex valued function, we note that it belongs to a Triebel (resp. Besov) space if both real and imaginary parts of $f$ belong to the Triebel (resp. Besov) space in question. We also say that a function (or distribution)
$f$ defined on the domain $\Omega\subset\R^d$ belongs to $F^s_{p,q, loc}(\Omega)$ if
$\psi f\in F^{s}_{p,q}(\R^d)$ (continued as zero outside $\Omega$)  for all $\psi\in C_c^\infty(\Omega)$,
 and $B^s_{p,p,loc}(\Omega)$ is defined analogously. \quad

\subsection{Besov domains and auxiliary results on Besov spaces}\label{BesovDomains}

In order to obtain optimal results for (quasi-)conformal mappings with regard to the smoothness of the boundary, it is useful to notice that both for Sobolev and Triebel spaces defined on a smooth and bounded domain (or in the upper half space), the boundary values belong  exactly  to the corresponding diagonal Besov space, with smoothness decreased by $1/p$. Let us state this more precisely in the case particularly important for us, i.e. for the unit disc and the space $W^{s,p}(\D)$ with $s>1/p$ and $p\in (1,\infty)$. The trace $f_{|\T}$ (originally defined for $f\in C^\infty (\overline{D})$)  extends to a linear and bounded operator 
$$
 W^{s,p}(\D)\ni f \;\mapsto\; f_{|\T}\in  B^{s-1/p}_{p,p}(\T).
 $$
 Moreover, this result is the best possible one since the trace mapping is onto. Exactly the same result holds true if $W^{s,p}(\D)$ is replaced by $F^{s}_{p,q}(\D)$ for any $q\in (1,\infty)$. Above, the definition of $B^s_{p,p}(\T)$ on the torus can be done equivalently in various ways (any reasonable definition leads to the same space),  for example 
\begin{equation*}
 f\in B^s_{p,p}(\T) \quad \textrm{if and only if}\quad f(e^{i\cdot})\in  B^s_{p,p,loc}(\R)
 \end{equation*}
 and one may define the norm either by setting $\|  f\|_{B^s_{pp}(\T)}:= \|  \psi_0f(e^{i\cdot})\|_{B^s_{pp}(\R)},$ where $\psi_0\in C_0^\infty(\R)$ is a fixed test function that satisfies $\psi(x)=1$ for (say) $x\in [-1,6]$, or by the formulae  \eqref{eqBesovInnerNorm} and \eqref{eqExtensionDomainXBESO} below.
 
Assume next that  $f:\D\to \Omega$ is a conformal map, where $\Omega$ is a bounded Jordan domain, and we know that $f\in W^{s,p}(\D)$. It then follows that $\partial\Omega=f(\T)$, where $f_{|\T}\in B^{s-1/p}_{p,p}(\T)$. Hence,  in order to obtain optimal smoothness results 
in the context of interior Sobolev regularity $f\in W^{s,p}(\D)$,  the natural assumption  is that the boundary is parametrized by a $B^{s-1/p}_{p,p}$-function.  However, this condition alone
does not prevent possible cusps of the boundary, as  seen by considering the image of the unit disc under the conformal map $z\mapsto (z+1)^2$. The right condition to prevent this phenomenon is simply to assume that the derivative of the parametrization does not vanish. Hence we assume in all our results that the boundary of the bounded domain $\Omega$ admits a bi-Lipschitz parametrization, which is also Besov-regular. This will be soon formalized in Definition \ref{defBesovDomain} below. 

Before going to the definition of Besov-domains, we prepare ourselves with some observations on compositions of Besov functions that will be useful in verifying the independence of the definition of the used parametrization, and later on needed while proving some of our main results.
 
By  \eqref{eqEmbeddingIntoDiniSOBO} the inequality $s p>1$ grants 
$\norm{f}_{L^\infty}\leq \norm{f}_{B^{s}_{p,p}}$ for every $f\in B^{s}_{p,p}(\R)$.  In fact, 
we also have the following counterpart to the Sobolev embeddings in the setting of Besov spaces:
\begin{proposition}\label{theoTriebelEmbeddings}
Given $1\leq p_0\leq p_1\leq \infty$ and $-\infty < s_1\leq s_0<\infty$. Then for spaces in one dimension (like $B^s_{p,p}(\R),$ 
$B^s_{p,p}(\T)$, or $B^s_{p,p}(I)$ where $I\subset \R$ is an interval) we have
\begin{equation*}
B^{s_0}_{p_0,p_0}\subset B^{s_1}_{p_1,p_1}\mbox{\,\,\,\, if }s_0-\frac{1}{p_0}\geq s_1-\frac{1}{p_1}.
\end{equation*}
\end{proposition}
\noindent { This result can be  found e.g. in  \cite[Section { 2.7.1}]{TriebelTheory} in case of $B^s_{p,p}(\R).$ 
 In the reference the inclusion is shown for $s_0-\frac{1}{p_0}= s_1-\frac{1}{p_1}$, but
 the general case follows simply by noting that $B^{s}_{p,p}\subset B^{s'}_{p,p}$ if $s\geq s'.$
Moreover, one has $B^s_{\infty,\infty}=\mathcal{C}^s$ (see  \cite[Section { 2.5.7}]{TriebelTheory}) and hence if $s >1+1/p$, then taking $p_1=\infty$ in the above embedding leads to 
\begin{equation}\label{inclusion}  B^{s}_{p,p}\subset \mathcal{C}^{s-\frac1{p}}=\mathcal{C}^{1+\varepsilon},
\end{equation}
 which is useful to keep in mind when we define Besov domains shortly below.

When considering functions on  the torus or an interval, for $k\in \N_0$, $k < s < k+1$ one may also use the seminorm
\begin{equation}\label{eqBesovInnerNorm}
\norm{f}_{\dot B^s_{p,p}(\T)}:=\left(\int_0^{2\pi}\int_0^{2\pi} \frac{|f^{(k)}(e^{it}) - f^{(k)}(e^{i\tau})|^p}{|e^{it}-e^{i\tau}|^{(s-k) p +1}} dtd\tau\, \right)^\frac1p.
\end{equation}
By \cite[Theorem 4.4.2]{TriebelInterpolation} the norm $\norm{f}_{L^p(\T)}+\norm{f}_{\dot B^s_{p,p}(\T)} $ is equivalent to the restriction norm for the Besov spaces (see Definitions \ref{defRestrictionNorm} and  \ref{defFamous}) whenever $1\leq p\leq \infty$ (see \cite[Section 3.4.2]{TriebelTheory} for the endpoints). Moreover, by the lifting property \cite[Theorem 3.3.5]{TriebelTheory} it holds that
\begin{equation}\label{eqExtensionDomainXBESO}
\norm{f}_{B^{s+1}_{p,p}(\T)} \approx \norm{f}_{B^{s}_{p,p}(\T)}+\norm{f'}_{B^{s}_{p,p}(\T)} \approx \norm{f}_{L^p(\T)}+\norm{f'}_{B^{s}_{p,p}(\T)} .\end{equation}
  Both statements
    \eqref{eqBesovInnerNorm} and \eqref{eqExtensionDomainXBESO} remain valid  if the torus $\T$ is replaced 
   by a bounded interval $I$. 

We next consider composition properties of the Besov spaces.
\begin{lemma}\label{lemBesovAdmissibleSpace} Let $I_j\subset \R$ for $j\in\{1,2\}$  be bounded intervals and assume that $f:I_1\to I_2$ is  a bi-Lipschitz homeomorphism.

\noindent (i)\quad Assume that $s\in (0,1)$ and $p\in (1,\infty)$. Then the composition $g\mapsto g\circ f$ defines a linear isomorphism between the spaces $B^s_{p,p}(I_2)$ and $B^s_{p,p}(I_1)$.

\smallskip

\noindent (ii)\quad Let $1<p <\infty$ and $s>1+1/p$. Then 
\begin{equation}\label{eqInvarianceUnderInversionBESO}
 f \in B^{s}_{p,p}(I_1) \implies f^{-1} \in B^{s}_{p,p}(I_2),
\end{equation}
and 
\begin{equation}\label{eqInvarianceUnderbiLipschitzBESO}
 f,g\in B^{s}_{p,p}(I_2) \implies g \circ f \in B^{s}_{p,p}(I_1).
\end{equation}



\smallskip

\noindent (iii) \quad  Both statements (i) and (ii) remain valid if  
 $I_1$ and $I_2$ are replaced by $\T$.

\end{lemma}

 \begin{proof}  The first statement (i) is well-known and follows almost immediately from \eqref{eqBesovInnerNorm} by a change of variables. Also the last statement (iii) follows readily 
 after (i) and (ii) are established.
Concerning  (ii)  we note that the second result  \eqref{eqInvarianceUnderbiLipschitzBESO} is contained in \cite[Theorem 1]{BourdaudMoussaiSickelTL}, since in the compactly supported case we may ignore the condition $f(0)=0$. In addition, we prove \eqref{eqInvarianceUnderInversionBESO} only for  $s\in \N$, as this case is not covered by Theorem \ref{lemTriebelAdmissibleBanach}  below and \cite[Lemma 2.2]{PratsTLComposition}. 
We may assume without loss of generality that $I_1=I_2=\T$ and that $f$ is a bi-Lipschitz homeomorphism of the unit circle. Denote by $F$ (resp. $G$) the Poisson extension of $f$ (resp. $g$) to the unit disc. Then by the classical Rado-Kneser-Choquet theorem (see \cite[Section 3.1]{Duren}) $F$ is a  homeomorphism of the unit disc. 
 Moreover, the extensions $F,G$ belong to $F^{s+\frac1p}_{p,2}(\D)$ by \cite[Theorem 4.3.3]{TriebelTheory}.

Since $s+\frac1p\notin \N$, we can apply \cite{PratsTLComposition} to the extensions as soon as $F$ is bi-Lipschitz. Here note that $(s-1)p>1$ implies that $B^s_{p,p}(\T)\subset C^{1+\varepsilon}(\T)$, in particular  the Hilbert transform $H(f')\in L^\infty$. We can hence  use Pavlovic's theorem \cite{Pavlovic} which says that a harmonic homeomorphism of the disk onto itself is bi-Lipschitz if and only if its boundary function is bi-Lipschitz and the Hilbert transform of the derivative of the boundary function is in $L^\infty$. In particular, our $F$ is a bi-Lipschitz map of  $\D$ and therefore, by \cite{PratsTLComposition} we see that $F^{-1}\in F^{s+\frac1p}_{p,2}(\D)$ and $G\circ F\in F^{s+\frac1p}_{p,2}(\D)$. 
 Finally, taking the traces we see that $g\circ f=(G \circ F)|_{\T}\in B^s_{p,p}(\T)$ and $f^{-1}= (F)^{-1}|_{\T}\in B^s_{p,p}(\T)$, proving \rf{eqInvarianceUnderbiLipschitzBESO} and \rf{eqInvarianceUnderInversionBESO}.  

 \end{proof}
 
We record an observation of compositions of Besov-functions, which will play a crucial role  later on in the proof of  Theorem \ref{theoSobolevRiemann}.
 
 \begin{lemma}\label{le:tasma}
 Let $1<p<\infty$ together with  $s_1>1+1/p$ and $s_2\in (0,s_1]$. Assume that $f\in B^{s_1}_{p,p}(\T)$ is bi-Lipschitz. Then for all functions $g\in B^{s_2}_{p,p}(\T)$ we
 have  
 $$
 g\circ f\in B^{s_2}_{p,p}(\T).
 $$
 \end{lemma}
 \begin{proof}  For $s_2\in (0,1)$ or $s_2=s_1$ the claim follows from Lemma
 \ref{lemBesovAdmissibleSpace} (i) or (ii), respectively. This entails that the linear operator
 $$
T: g\mapsto g\circ f
 $$ 
is bounded $T:B^{s_2}_{p,p}\to B^{s_2}_{p,p}$ for $s_2\in (0,1)$ and $s_2=s_1$, whence the boundedness follows for all $s_2\in (0,s_1]$ by complex interpolation, see \cite[Section 2.4.7]{TriebelTheory}
\end{proof}

\begin{lemma}\label{le:peetre} Assume that  $1<p<\infty$ and $s>1/p$.

\smallskip

\noindent (i)\quad Assume that  $f\in B^{s}_{p,p}(\T)$ is real-valued. Then $\psi(f)\in B^{s}_{p,p}(\T)$ for all $\psi\in C^\infty (\R)$.
 
 \smallskip
 
 \noindent (ii)\quad Assume that $f=u+iv\in B^{s}_{p,p}(\T)$.  Then $e^f\in  B^{s}_{p,p}(\T)$. Moreover, if $f(w)\not=0$ for $w\in \T$, then any local branch of $\arg(f)$  belongs to  $B^{s}_{p,p, loc}(\T)$  as well. 
\end{lemma}

\begin{proof}
{Statement (i)  is due to Peetre - note we do not need the condition $\psi(0)=0$ in view of  compactness of $\T$}. We refer the reader to \cite[Theorem 5.4.1]{Runst} for this and more general statements of the same type.   Towards (ii) we note first that $e^f=e^ue^{iv}$, and the space $B^s_{p,p}(\Omega)$ is an algebra (see e.g. Lemma \ref{lemAlgebra} below), so it is enough to prove the claim for both factors separately, and the claim then follows from part (i). Finally, we may write locally $\arg{f}=\arctan (v/u) +c $ or $c-\arctan (u/v)$, and again the statement follows by part (i) and the algebra property of the Besov space as soon as one localizes $x\mapsto 1/x$ suitably to a globally smooth function.
\end{proof}

We are finally ready to define (bounded) Besov-domains in the plane:

\begin{definition}\label{defBesovDomain}
Let   $1\leq p <\infty$ and $s>1+1/p$, and assume  that $\Omega$ is a  bounded  and finitely connected domain. 
We say that  $\Omega$ is a \emph{$B^{s}_{p,p}$-domain} if $\partial\Omega$ is a finite collection of disjoint Jordan curves $\{\Gamma_j\}_{j=1}^M$ and each boundary component $\Gamma_j$ has a bi-Lipschitz parameterization $\gamma^j \in B^{s}_{p,p}(\T)$.
 \end{definition}
 
 Note that by the  definition above, every $B^{s}_{p,p}$-domain $\Omega$ is automatically a Lipschitz domain, and thus in particular, $\partial \Omega$ has no cusps and allows no spiralling. 
 
 In a similar way we can define Besov spaces  on the boundaries of Besov-domains.
\begin{definition}\label{defTraceSpace}
Assume that   $1\leq p<\infty$ and $s>1+1/p$. Let $\Gamma$ be the boundary of a   simply connected $B^{s}_{p,p}$-domain,  with a bi-Lipschitz parameterization $\gamma: \T \to \Gamma$ such that  $\gamma \in B^{s}_{p,p}(\T)$. We say that a measurable function $f:\Gamma \to \C$ belongs to $B^{s}_{p,p}(\Gamma)$ if 
$$\norm{f}_{B^{s}_{p,p}(\Gamma)}:= \norm{f\circ \gamma}_{B^{s}_{p,p}(\T)} < \infty.$$
This definition extends naturally to the boundaries of finitely connected Besov-domains. \end{definition}

It is important to note that the above Definition 
does not depend on the particular choice of the parametrization. We state this fact as a separate lemma.

\begin{lemma}\label{le:independence} Assume that  $s>1+1/p$.   For a $B^{s}_{p,p}$-domain  the arc-length parametrization (actually, a suitable multiple of it) yields an admissible $B^{s}_{p,p}$-parametrization.  Moreover, in Definition \ref{defTraceSpace} different parametrizations lead to  equivalent  norms for  $B^{s}_{p,p}(\Gamma)$. In particular, for two different admissible parametrizations $w_1$ and $w_2$ one has
 \begin{equation}\label{eq:vaihto}
w_2^{-1}\circ w_1\in B^{s}_{p,p} (\T).
\end{equation}
\end{lemma}

\begin{proof}
 In order to prove the first statement, we may assume that $M=1$ in Definition \ref{defBesovDomain} and we denote $\gamma=(\gamma_1,\gamma_2):=\gamma^1$. It is naturally enough to prove that the arc-length parametrization of $\Gamma$ is in $B^{s}_{p,p}(\T)$, since this then also proves the stated independence of the parametrization. To simplify notation, since we only deal with local properties we may assume that the parametrizations are defined as periodic functions on  the real axis. Thus  the components $\gamma_1,\gamma_2$ are in (real-valued) $B^s_{p,p}(\R)$, and by the bi-Lipschitz property we have  $0<c\leq  |\gamma'(t)|\leq C$ for all $t$. Let us denote by $\ell(t)$ the length of the curve over the parameter interval $[0,t]$:
 $$
 \ell(t):=\int_0^t|\gamma'(u)|du= \int_0^t\big((\gamma_1'(u))^2+(\gamma_2'(u))^2\big)^{1/2}du.
 $$
 By the algebra property of the space $B^{s-1}_{p,p}(\R)$ we see that $|\gamma'|^2\in B^{s-1}_{p,p}(\R)$, and then the same is true for $u\mapsto |\gamma'(u)|$, as we may continue $x^{1/2}$ on  $[c,C]$ to a compactly supported $C^\infty$-function,  and the Besov-property is preserved locally under composition with smooth functions. The latter fact  is true also for Triebel spaces, and {is attributed to Peetre} \cite{Peetre}, we {again} refer to \cite{Runst} for a more extensive discussion of this kind of results. In any case, we obtain that $\ell$ is bi-Lipschitz and  $\ell\in B^{s}_{p,p,loc}(\R)$. Hence,  by writing $S:= \int_0^{2\pi}|\gamma'(u)|du={\mathcal H}^1(\partial \Omega)$, Lemma \ref{lemBesovAdmissibleSpace} verifies that  (a multiple of) the arclength parametrization $
 \gamma\circ \ell^{-1}(2\pi \cdot /S)$ indeed yields a $B^{s}_{p,p}(\R)$-parametrization. 
 
 In order to prove \eqref{eq:vaihto} we write $w_j(t)=u_j(t)+iv_j(t)$ , $j=1,2$.  Given an arbitrary parameter $t_1$, we have  $u'_1(t_1)\not=0$ or  $v'_1(t_1)\not=0$, and may assume by symmetry the first alternative. Considering the tangent line of $\partial \Omega$ at $w_1(t_1)$ we see that also $u'_2(t_2)\not=0$, where $w_2(t_2)=w_1(t_1)$. Especially, the functions $u_j$  are locally invertible in a neighborhood of $t_j$ ($j=1,2$), and hence in a neighborhood of $t_1$ we have
 $$
 w_2^{-1}\circ w_1= u_2^{-1}\circ u_1,
 $$
 which proves \eqref{eq:vaihto} in view of Lemma \ref{lemBesovAdmissibleSpace}.

  Finally, the second statement  concerning Definition \ref{defTraceSpace} follows immediately from  Lemma \ref{lemBesovAdmissibleSpace} and \eqref{eq:vaihto}.
 \end{proof}

{According to what we have  just proved,} in the parametrization of the boundary of a component of a Besov-domain we may assume without loss of generality that the parametrization  of $\gamma$ is  a multiple of the arc-length. This means that $\gamma_j' (z) = c \nu(\gamma(z))^{\bot}$ for $z\in\T$, where $\nu$  is the unit normal vector on $\Gamma$. Thus we see that the domains appearing in \cite{PratsPlanarDomains} are exactly the ones in Definition \ref{defBesovDomain}:
\begin{corollary}
Let  $1 \leq p \leq \infty$  and with  $s>1+1/p$. A bounded Lipschitz domain $\Omega$ is a ${B}^{s}_{p,p}$-domain if and only if given its outward unit normal vector
$\nu$ and  (scaled) arc-length parameterization $\gamma$ of any component of its boundary, we have that $\nu\circ \gamma \in {B}^{s-1}_{p,p}(\T)$.  
\end{corollary}


\subsection{Properties of Triebel-Lizorkin spaces}


Let us fix the following notation: given a domain $\Omega\subset\R^d$ and $x\in \R^d$, set
$$\delta_\Omega(x):=\dist(x,\partial\Omega).$$ 
In Section \ref{secPrincipal}  and the proof of Theorem \ref{theoInvertBeltrami} we will use the following characterization for the Triebel-Lizorkin space on $\Omega$:
\begin{theorem}[see {\cite[Theorem 1.2, Corollary 1.4]{PratsTL}} and {\cite[Corollary 2]{Seeger}}]\label{theoNormOmegapgtrq}
Let $\Omega\subset \R^d$ be a bounded Lipschitz domain, let  $0<\sigma< 1$, $k\in\N_0$, $1\leq  p <\infty$, $1\leq q\leq \infty$  with $\sigma>\frac dp-\frac dq$, and $0<\rho<1$. 
For $s=k+\sigma$ we write
\begin{equation*}
\norm{f}_{F^s_{p,q}(\Omega)}:=\norm{f}_{W^{k,p}(\Omega)} + \sum_{|\alpha|=k} \left(\int_\Omega \left(\int_{B\left(x,\rho \delta_\Omega(x)\right)}\frac{|D^\alpha f(x)-D^\alpha f(y)|^q}{|x-y|^{\sigma q+d}} \,dy\right)^{\frac{p}{q}}dx\right)^{\frac{1}{p}}<\infty,
\end{equation*}
with the usual modification for $q=\infty$. Then, the norm defined above is equivalent to the restriction norm  from Definition \ref{defRestrictionNorm}.
\end{theorem}
%

The previous result is based on an extension operator fit to the intrinsic norms defined above. 
\begin{theorem}[see {\cite[Theorem 1.5]{PratsTL}}]\label{theoExtension}
Let $\Omega$ be a Lipschitz domain and $k\in \N_0$. There exists a linear operator $\Lambda_k: W^{k,1}_{\rm loc}(\Omega) \to L^1_{\rm loc}(\R^d)$ such that for every $1\leq p <\infty $, $1\leq q\leq \infty$ and $0<\sigma<1$ with $\sigma>\frac dp-\frac dq$, then 
$$\Lambda_k:F^s_{p,q}(\Omega)\to F^s_{p,q} (\R^d)$$
(with $s=\sigma+k$) is a bounded extension operator.

\end{theorem}

In fact, e.g. by  using universal extension operators (see [Rychkov]), the embeddings described in Proposition \ref{theoTriebelEmbeddings} have the following counterpart in the present setting:
\begin{proposition}[See {\cite[Sections 2.3 and 2.7]{TriebelTheory}}]\label{theoTriebelEmbeddingsOmega}
The following properties hold whenever $\Omega$ is a bounded Lipschitz domain:

\smallskip

\noindent(i)\quad Let $1 \leq q_0 < q_1\leq \infty$ and $1\leq p <  \infty$ and $s\in\R$. Then 
$$F^{s}_{p,q_0}(\Omega)\subset F^{s}_{p,q_1}(\Omega).$$

\smallskip

\noindent(ii)\quad 
 Let $1\leq q_0, q_1\leq \infty$ and $1\leq p<  \infty$, $s\in\R$ and $\varepsilon>0$. Then 
$$F^{s+\varepsilon}_{p,q_0}(\Omega)\subset F^{s}_{p,q_1}(\Omega).$$

\smallskip

\noindent(iii)\quad  Given $\varepsilon>0$, $1\leq p_0< p_1< \infty$, $1\leq q_0,q_1\leq \infty$ and $-\infty < s_1 < s_0<\infty$ with $s_0\in\N$ and $s_0-\frac{d}{p_0}=s_1-\frac{d}{p_1}$, then
$$F^{s_0}_{p_0,q_0}(\Omega)\subset F^{s_1}_{p_1,q_1}(\Omega).$$
\end{proposition}

We will also need a {minor} extension to Lemma \ref{le:peetre}.

\begin{lemma}\label{le:runst} Assume that  $1<p<\infty$ and $s>1/p$, and $q\in (1,\infty)$. Let $B\subset\C$ be a disc and assume that
$f\in F^{s}_{p,q}(B)$ satisfies $|f(z)|\geq c>0$ on $B$.    Then $1/f\in F^{s}_{p,q}(B)$.
\end{lemma} 
\begin{proof}
By writing $f=u+iv$ and noting that $1/f=(u^2+v^2)^{-1}(u-iv),$ the claim follows from the algebra property of Triebel spaces in  the above range of $s$, $p$ and $q$, and the  fact that again in this range, the Triebel spaces on $\R^2$ are invariant under compositions $g\mapsto\psi(g)$, where $\psi\in C^\infty(\R)$ satisfies $\psi(0)=0$, see \cite[Theorem 5.4.1]{Runst}. The claim then follows by choosing $\psi$ which is a suitable localization of $1/x$.
\end{proof}


Here again we have the lifting property, stability under composition and an inverse function theorem, which are essential in the proof of Theorem \ref{theoTriebelStabilityOfDomains} based on  Stoilow factorization.
\begin{theorem}[see {\cite[Theorem 1.1]{PratsTLComposition}}]\label{lemTriebelAdmissibleBanach}
Let $ s \in \R_+ \setminus \N$, $1\leq p<\infty$, $1\leq q \leq\infty$  and  $d\in\N$. 
Then the following holds:

\smallskip

\noindent(i)\quad 
 Given a bounded Lipschitz domain $\Omega\subset \R^d$, every function $f \in F^{s+1}_{p,q}(\Omega)$ satisfies that
\begin{equation}\label{eqExtensionDomainXTRIEBEL}
\norm{f}_{F^{s+1}_{p,q}(\Omega)} \approx \norm{f}_{F^{s}_{p,q}(\Omega)} +\norm{\nabla f}_{F^{s}_{p,q}(\Omega)}.\end{equation}

\smallskip

\noindent(ii)\quad 
 If $s >1+d/p$, given bounded Lipschitz domains $\Omega_j\subset \R^d$ and functions $f_j\in F^{s}_{p,q}(\Omega_j)$ with $f_1(\Omega_1)\subset \Omega_2$ and $f_1$ bi-Lipzchitz, then 
\begin{equation*}
f_2 \circ f_1 \in F^{s}_{p,q}(\Omega_1)
\end{equation*}
and  if $f_1(\Omega_1)=\Omega_2$, then
\begin{equation*}
f_1^{-1} \in F^{s}_{p,q}(\Omega_2).
\end{equation*}
 \end{theorem}
%

In the case of Sobolev spaces the corresponding result is true also for integer values of the smoothness:
\begin{theorem}[see {\cite[Lemma 2.10]{PratsTLComposition}}]\label{lemSobolevAdmissibleSpace}
Let $s>1+d/p$ and $1< p < \infty$. Given bounded Lipschitz domains $\Omega_j\subset \R^d$ and functions $f_j\in W^{s,p}(\Omega_j)$  with $f_1(\Omega_1)\subset \Omega_2$ and $f_1$ bi-Lipschitz, then 
\begin{align}\label{eqInvarianceUnderbiLipschitzSOBO}
f_2 \circ f_1 & \in W^{s,p}(\Omega_1),
\end{align}
and if $f_1(\Omega_1)=\Omega_2$, then
\begin{equation}\label{eqInvarianceUnderInversionSOBO}
f_1^{-1} \in W^{s,p}(\Omega_2).
\end{equation}
\end{theorem}

 \begin{remark} \label{blablabla} {\rm Somewhat surprising, it is not clear if Theorem \ref{lemTriebelAdmissibleBanach} holds true for the integer smoothness $s \in \N$. The "technical" reason for this is that in the characterization of Triebel-Lizorkin spaces in Theorem \ref{theoNormOmegapgtrq},  when $s$ is an integer the differences of derivatives 
$D^\alpha f$ need to be replaced by double differences, see \cite{PratsTLComposition}. 
As pointed out in Theorem \ref{lemSobolevAdmissibleSpace}, 
this problem does not arise in the Sobolev scale and the classical Sobolev spaces of integer smoothness.}
\end{remark}


To close this Section, let us check the algebra structure of the supercritical Triebel-Lizorkin spaces, which we will use in Section \ref{secPrincipal}.
%
%
\begin{lemma}\label{lemAlgebra}
Let $d\in \N$, $s>0$, $1\leq p<\infty$ and $1\leq q\leq \infty$. If $\Omega\subset \R^d$ is a Lipschitz domain, then for every pair $f,g\in F^s_{p,q}(\Omega)$ we have that
\begin{equation*}
\norm{f\, g}_{F^s_{p,q}(\Omega)}\leq C_{d,s,p,q,\Omega} \norm{f}_{F^s_{p,q}\cap L^\infty(\Omega)}\norm{ g}_{F^s_{p,q}\cap L^\infty(\Omega)}.
\end{equation*}
In particular, if $sp>d$, then
\begin{equation}\label{eqAlgebra}
\norm{f\, g}_{F^s_{p,q}(\Omega)}\leq C_{d,s,p,q,\Omega} \norm{f}_{F^s_{p,q}(\Omega)}\norm{ g}_{F^s_{p,q}(\Omega)}.
\end{equation}
Moreover, for $m\in \N$ we have that
\begin{equation}\label{eqAlgebraPower}
\norm{f^m}_{F^s_{p,q}(\Omega)}\leq C_{d,s,p,q,\Omega} m^{N} \, \norm{f}_{L^\infty(\Omega)}^{m-1}  \norm{f}_{F^s_{p,q}(\Omega)},
\end{equation}
with $N$ depending on $d$, $s$, $p$ and $q$.
\end{lemma}
\begin{proof}
In \cite[Theorem 4.6.4/2]{RunstSickel} it is shown that $F^s_{p,q}\cap L^\infty(\R^d)$ is a multiplicative algebra. By Corollary \ref{coroRychkov}, there is a common extension operator $\mathcal{E}$ for $F^s_{p,q}(\Omega)$ and $L^\infty(\Omega)$ Thus,
\begin{align*}
\norm{f\, g}_{F^s_{p,q}(\Omega)}
	& \leq \norm{\mathcal{E} f \mathcal{E} g}_{F^s_{p,q}(\R^d)}	
		 \leq C \left(\norm{\mathcal{E} f}_{F^s_{p,q}(\R^d)}\norm{\mathcal{E} g}_{L^\infty(\R^d)} + \norm{\mathcal{E} f}_{L^\infty(\R^d)}\norm{\mathcal{E} g}_{F^s_{p,q}(\R^d)}\right)\\
	&	 \leq C_{d,s,p,q,\Omega} \norm{f}_{F^s_{p,q}\cap L^\infty(\Omega)}\norm{ g}_{F^s_{p,q}\cap L^\infty(\Omega)}.
\end{align*}

Regarding inequality \rf{eqAlgebraPower}, we proceed analogously, using the algebra structure \cite[Theorem 4.6.4/2]{RunstSickel} to get
\begin{align*}
\norm{\mathcal{E}(f)^{2j}}_{F^s_{p,q}}
	& \leq 2c \norm{\mathcal{E}(f)}_{L^\infty}^{j}\norm{\mathcal{E}(f)^{j}}_{F^s_{p,q}}, \quad\quad\mbox{and}\\
\norm{\mathcal{E}(f)^{2j+1}}_{F^s_{p,q}}
	& \leq c \left(\norm{\mathcal{E}(f)}_{L^\infty}\norm{\mathcal{E}(f)^{2j}}_{F^s_{p,q}}+\norm{\mathcal{E}(f)}_{L^\infty}^{2j}\norm{\mathcal{E}(f)}_{F^s_{p,q}}\right)
\end{align*}
and combining both estimates in an iterated way, in about $\lfloor\log_2(j)\rfloor$ steps.   
\end{proof}

\section{Regularity of the Riemann mapping}\label{secRiemann}






In this section we study the global  $W^{s,p}$- and Triebel-Lizorkin -regularity of conformal paramet-rizations. The natural framework here is the class of finitely connected domains, since the classical theorem of Koebe allows them a parametrization by a circle domain. On the other hand, for the general Triebel-Lizorkin spaces it is useful to slightly modify the regularity assumptions of Theorem \ref{theoSobolevRiemann}.
Namely, instead of requiring  regularity for the inverse, we assume  that the conformal maps are bi-Lipschitz, in addition the map  having the appropriate Triebel-Lizorkin -regularity.
Then later in this section we  will return to Theorem \ref{theoSobolevRiemann} and show how it follows from the results obtained.

Moreover, for clarity of presentation we start with the simply connected domains, but once they are well understood the case of general finitely connected domains will follow easily, see Corollary \ref{finiteconn} below.

The key result of this section is the following. 

\begin{theorem}\label{theoTriebelRiemann}
Suppose $s>0$,  $1<p<\infty$ with $sp>2$  and   $1 < q  <\infty$,  and let $\Omega$ be a bounded simply connected domain with Riemann map $\varphi :\D \to \Omega$. 

\noindent Then $\Omega$ is a $B^{s+1-\frac1p}_{p,p}$-domain, if and only if $\varphi \in F^{s+1}_{p,q}(\D)$ and $\varphi$ is bi-Lipschitz. 
\end{theorem}

The attentive reader notes that {the above condition} on the regularity of the boundary is independent of the value of $q$, i.e. the Riemann mapping is in all Triebel-Lizorkin spaces $F^s_{p,q}(\D)$, regardless of the value of $q$. This is related to the fact that for $s$ and $p$ fixed, all these spaces have the same trace space. Hence as a side result, we see that if a bi-Lipschitz Riemann mapping is in $F^s_{p,q_0}$ then it is   in $F^s_{p,q}$  for every $1< q< \infty$. \quad 

\medskip

For the proof of Theorem \ref{theoTriebelRiemann} we will generalize the approach used by Pommerenke in \cite[Theorems 3.5, 3.6]{PommerenkeConformal}, and work with the interplay between three elements: First, information on the Riemann mapping will be carried by $\Phi(z): = \log \varphi '(z)$. Second, the boundary values of the Riemann mapping will be encoded in the function 
\begin{equation}\label{gamma}
\gamma := (\arg \varphi ')_{|\T} ={\rm Im\,}\Phi_{|  \T}
\end{equation}
and finally we need to analyze the relation between $\gamma$ and the given  Besov-regular parameterization of the boundary $\partial \Omega$, see \eqref{eqDerivativeBoundaryGood} below.  

Since $\Phi$ can be expressed  as the Herglotz extension of $\gamma$, 
{\begin{equation*}
\Phi(z) = \log|\varphi'(0)| + \frac{i}{2\pi} \int_\T \frac{\zeta+z}{\zeta - z} \, \gamma(\zeta)\, |d\zeta| ,
\end{equation*}
see  \cite[Theorem 3.2]{PommerenkeConformal},}
in the end this will allow us to deduce the regularity of $\Phi$ by means of classical extension results. However, as the definition of $\gamma$ uses $\Phi$, this scheme needs to be applied inductively to reach arbitrary values $s > 2/p$.

\begin{lemma}\label{propoRiemannTriebel}  Let   $1<p<\infty.$  Then

\noindent  (i) \quad The Poisson extension maps $B^{s-\frac1p}_{p,p}(\T)$ to $F^s_{p,q}(\D)$, for any $1\leq q \leq \infty$ and $s>1/p$.

\smallskip

\noindent (ii) \;  The Hilbert transform acts  boundedly on $B^{s}_{p,p}(\T)$, for any $s>0.$
\smallskip

\noindent (iii) \;  {The Herglotz extension maps  $B^{s-\frac1p}_{p,p}(\T)$ to $F^s_{p,q}(\D)$, for any $1\leq q \leq \infty$ and $s>1/p$.}
\end{lemma}
\begin{proof}

The first statement  follows from \cite[Theorem 4.3.3]{TriebelTheory} which proves that the harmonic extensions of  $B^{s-\frac1p}_{p,p}(\T)$-functions belong to $F^s_{p,q}(\D)$. The reference, however, does not cover the spaces with $q=\infty$,
 but that case may be deduced from Proposition \ref{theoTriebelEmbeddingsOmega}, via the boundedness shown for $q<\infty$.  
 
 In turn, claim (ii) can  be deduced from the  equivalent definition of Besov spaces based on the Fourier series on  the torus (see e.g. \cite[6.6.1.1]{Sawano}), and the fact that the Hilbert transform commutes with Fourier multipliers and is bounded on $L^p(\T)$ for $1<p<\infty$. 

{Note that the Herglotz extension of an integrable function is holomorphic and, in particular, its real and imaginary parts are conjugate harmonic. Therefore its real part coincides with the Poisson extension, and its imaginary part is the Poisson extension of its Hilbert transform modulo additive constant, see \cite[Chapter III.1]{Garnett} for instance. Thus, the third statement follows from the previous two statements.}
\end{proof}
\smallskip


\begin{proof}[Proof of Theorem \ref{theoTriebelRiemann}]
Let us first assume that the Riemann mapping  $\varphi \in F^{s+1}_{p,q}(\D)$ and that the map is bi-Lipschitz. Then according to \cite[Theorem 3.3.3]{TriebelTheory}, the boundary value of $\varphi$ lies in the corresponding trace space,  $\varphi_{|\T} \in B^{s+1-\frac1p}_{p,p}(\T)$. In other words,  the trace of $\varphi$  is a  bi-Lipschitz Besov parameterisation of the boundary, so that the requirements of a $B^{s+1-\frac1p}_{p,p}$-domain in  Definition \ref{defBesovDomain} are satisfied. Therefore,  we only need to prove the converse direction
of Theorem \ref{theoTriebelRiemann}.

Hence assume that $\Omega$ is a $B^{s +1 -\frac1p}_{p,p}$-domain. The embedding  in Proposition \ref{theoTriebelEmbeddings} and the inclusion \eqref{inclusion} with $\varepsilon := s -\frac2p>0$ guarantee that $\Omega$ is a  ${\mathcal{C}}^{1+\varepsilon}$-domain. In particular, \cite[Theorem 3.5]{PommerenkeConformal} implies that $\varphi$ is bi-Lipschitz up to the boundary. It thus remains to show that $\varphi\in F^{s+1}_{p,q}(\D)$.
\smallskip

For this, recall  the notation $\Phi(z) := \log \varphi '(z)$, and note that defining a suitable continuous branch of the logarithm on $\overline{\D}$ poses no problem since $\D$ is simply connected. In particular, $\gamma(\zeta)$, the argument of $\varphi'$ on the boundary as in \eqref{gamma},
is a well defined continuous function.   
{Recall also that the functions are related via the
Herglotz formula,} 
\begin{equation*}
\Phi(z) = \log|\varphi'(0)| + \frac{i}{2\pi} \int_\T \frac{\zeta+z}{\zeta - z} \, \gamma(\zeta)\, |d\zeta| =:\log|\varphi'(0)| + \frac{i}{2\pi} A(\gamma)(z).
\end{equation*}

\smallskip

The main point of the proof is  to show that 
\begin{equation}\label{eq:sufficient}
\gamma \in B^{s-1/p}_{p,p}(\T).
\end{equation}
Namely, let us assume \eqref{eq:sufficient} holds. Then, since  $s-1/p>1/p$,  Lemma \ref{propoRiemannTriebel} shows that $(\log \varphi ')_{|\T} \in  B^{s-1/p}_{p,p}(\T)$, and  Lemma \ref{le:peetre}(ii) on  post-composition with smooth  functions  proves that 
$(\varphi_{|\T})' = ( \varphi')_{|\T}\in B^{s -1/p}_{p,p}(\T)$. In turn, this implies $\varphi_{|\T}\in B^{s+1-1/p}_{p,p}(\T)$.  Finally, we use Proposition \ref{propoRiemannTriebel} (i) to conclude that  $\varphi\in F^{s+1}_{p,q}(\D)$, as desired. 



For the proof of \eqref{eq:sufficient}, assume  that $\Omega$ is a $B^{s +1 -\frac1p}_{p,p}$-domain. Then its boundary has a 
 bi-Lipschitz parametrization $w:\T \to \partial \Omega$ with $w \in B^{s+1-\frac1p}_{p,p}(\T)$.  Via \eqref{inclusion} it  follows that  $w\in C^{1+\varepsilon}$, and by \cite[Theorem 3.2]{PommerenkeConformal} (or by a direct verification) we have 
 \newcommand{\Arg}{\mathop{\rm Arg}}
\begin{equation}\label{eqDerivativeBoundaryGood}
\gamma(e^{it})=-t-\frac{\pi}{2} + [\Arg w']\circ \left[w^{-1}\circ \varphi \right](e^{it}), 
\end{equation}
where we understand $\Arg w'(e^{iu}):=\arg\big(\frac{d}{du} \big[w(e^{iu})\big]\big)$ which is locally well-defined.
Note that any choice of $\Arg w'$ increases by $2\pi$ as $e^{it}$ spins around $\T$, and the linear part of $\gamma$ 
in \eqref{eqDerivativeBoundaryGood} decreases by the same amount, so $\gamma$ is indeed continuous on $\T$ after picking any consistent  choice for the multivalued function $\Arg w'$.
Moreover,  by Lemma \ref{le:peetre}(ii)  we have
 \begin{equation}\label{eqArgW}
 \Arg{w'} 
  \in B^{s-\frac1p}_{p,p}(\T)
 \end{equation} 
 as well. 
\smallskip

After these preparations we proceed to establish \eqref{eq:sufficient}, by using the knowledge \eqref{eqArgW}.  This requires  a bootstrapping argument which employs  identity \eqref{eqDerivativeBoundaryGood}: Namely,
 by assuming the truth of the claim for a smoothness $s'\in (s-1,s)$ 
 with $s'p >2$,  we will deduce the statement for  $s$.

  \smallskip
  
  As the first step  assume that   $2/p <s< 1+1/p$. This base case for our  induction can be dealt with easily. Indeed, since  our assumption on $\Omega$ implies that $\varphi$  is bi-Lipschitz, and since $w$ is  bi-Lipschitz by definition, as a parametrization of a Besov boundary, also  $w^{-1}\circ \varphi$ is a bi-Lipschitz  homeomorphism of $\T$.  Thus for any  $2/p <s< 1+1/p$,  Lemma \ref{lemBesovAdmissibleSpace}(i) and the identities \eqref{eqDerivativeBoundaryGood} - \eqref{eqArgW}
show that $\gamma  \in B^{s-\frac1p}_{p,p} (\T)$. As discussed after \eqref{eq:sufficient} above, this gives the theorem for such indices.

We then describe the induction step. Assume that $s >1$ and $\Omega $ is a $B^{s+1-1/p}_{p,p}$-domain with $p\in (1,\infty)$ and  $ps>2$.  
Then obviously $\Omega$ is also a  $B^{s'+1-1/p}_{p,p}$-domain
for every $2/p <s'<s$. Let us now assume that we already know that  the theorem is true for   such a pair $(s',p)$. We then obtain  $\varphi \in B^{s'+1-1/p}_{p,p}(\T)$, and that $\varphi$ is
  bi-Lipschitz. Moreover, also $w$ is  bi-Lipschitz and obviously
 $w\in B^{s'+1-1/p}_{p,p}$. Since  both $w$ and $\varphi$ are admissible $B^{s'+1-1/p}_{p,p}$-parametrizations of $\partial \Omega$ we deduce from  Lemma \ref{le:independence}   that 
\begin{equation*}
w^{-1}\circ \varphi \in B^{ s'+1-1/p}_{p,p}(\T), \quad\textrm{and}\;  \;\; w^{-1}\circ \varphi \;\;\textrm{is bi-Lipschitz}.
\end{equation*}
Furthermore, if in addition to the above assumptions we have $s < s'+1$, then  
 Lemma \ref{le:tasma}  with  \eqref{eqArgW}  implies that  $\, \Arg w'\circ w^{-1}\circ \varphi\in B^{s-1/p}_{p,p}(\T)$.
 In particular, from \eqref{eqDerivativeBoundaryGood} we see that now \eqref{eq:sufficient} holds.

 Thus, knowing the theorem for the pair $(s',p)$ implies it for all $(s,p)$ with  $s'<s<s'+1$. This implies the theorem for all $s>2/p$ since in the first step of the proof we noted that the claim is true for $s$ in the range $2/p<s<1+1/p$, and this interval is non-empty for all $p>1$.
\end{proof}

\smallskip

{It is likely that one may allow  some cases of $p,q\in\{1,\infty\}$ in the above theorem, but we have not pursued this since the result in the reflexive range is enough for our purposes.}

Let us then turn to the finitely connected domains. According to the  classical theorem of Koebe \cite{Koebe}, any finitely connected planar domain $\Omega$ with non-degenerate boundary components can be uniformized by a {\it circle domain}, a domain of the type
$$ U = \D(z_0,r) \setminus \bigcup_{j=1}^m {\overline{ \D(z_j,r_j)}},
$$
where the  closed subdiscs ${\overline{ \D(z_j,r_j)}} \subset \D(z_0,r) $ are  disjoint. Theorem \ref{theoTriebelRiemann}
generalises quickly to such situations.

\begin{corollary} \label{finiteconn} Suppose and $s>0$,  $1<p<\infty$ with $sp>2$  and   $1 < q  <\infty$. 
\smallskip

 \noindent Let  $U$ be a circle domain and  $\varphi :U \to \Omega$ a conformal uniformisation of   a bounded finitely connected domain $\Omega$. Then $\Omega$ is a $B^{s+1-\frac1p}_{p,p}$-domain, if and only if $\varphi \in F^{s+1}_{p,q}(\D)$ and $\varphi$ is bi-Lipschitz. 
\end{corollary}
\begin{proof} 
 If $\varphi \in F^{s+1}_{p,q}(\D)$ and $\varphi$ is bi-Lipschitz, then clearly  $\Omega$ is a $B^{s+1-\frac1p}_{p,p}$-domain. For the converse, assume that $\Omega$ is a $B^{s+1-\frac1p}_{p,p}$-domain. As such, it is also a $C^{1+\varepsilon}$-domain for some $\varepsilon >0$, so that by  \cite[Theorem 3.5]{PommerenkeConformal} the mapping is bi-Lipschitz. Hence the Corollary follows as soon as we show that  every component circle  $C_j$ of 
 $\partial U$ has a collar neighbourhood  where $\varphi \in F^{s+1}_{p,q,loc}$.
  
 Thus fix a component $\Gamma_j = \varphi(C_j)$ of $\partial \Omega$ - here $C_j$ is one of the  boundary circles  of $U$ - and we may well assume that   $C_j=\T=\partial \D$. Consider now first the case where $\Gamma_j = \varphi(\T)$ is the  outer component of $\partial \Omega$. We may then assume that similarly  $\T$ is the outer component of $\partial U$. Namely if not,  compose $\varphi$ with $\rho(z) = 1/z$, the analytic reflection in $\T$. Since  diffeomorphic coordinate changes preserve the Triebel spaces,
 we have $\varphi \in F^{s+1}_{p,q}(U)$ if and only if  $\varphi \circ \rho \in F^{s+1}_{p,q}(\rho U)$.
 
  Next, let $\widetilde \Omega \supset \Omega$ denote the simply connected and bounded Besov-domain for which $\partial\widetilde\Omega = \varphi(\T)$. Choose also a conformal map $\widetilde\varphi:\D\to\widetilde\Omega$.  By the reflection principle, the map $$\psi:=(\widetilde\varphi)^{-1}\circ\varphi$$  extends analytically across the unit circle, to a full neighbourhood $N$ of  $\T$.  Furthermore,
Theorem \ref{theoTriebelRiemann} shows that  $\widetilde\varphi \in F^{s+1}_{p,q}( \D)$. Hence one may again apply a diffeomorphic coordinate change (which preserves the Triebel spaces) to  see that  the composition $\varphi=\widetilde \varphi\circ\psi\in F^{s+1}_{p,q}(N\cap U)$.
\smallskip


 Finally, consider the case where $\Gamma_j = \varphi(\T)$ is an inner component of the boundary $\partial \Omega$. Similarly as above we may assume that, however, $\T$ is the outer component of $\partial U$.  In this setting,  let us denote by $\widetilde\Omega \subset \overline{\C}$
   the unbounded domain whose boundary (with respect to $\overline{\C}$) is equal to $\varphi(\T)$. Thus again $\Omega\subset\widetilde\Omega$. We may assume that $0$ lies inside the bounded component of $\C\setminus \partial \widetilde\Omega$, and  again  apply  Theorem \ref{theoTriebelRiemann} to an auxiliary mapping $\widetilde \varphi$ on $\D$,  now a conformal map onto the image of $\widetilde\Omega $ under the inversion $z\mapsto 1/z$. 
   
  In this case the analytic change of coordinates is given by  $\psi(z) = \widetilde \varphi \,^{-1}\big(1/\varphi(z)\big)$, but as above one can  use  Theorem \ref{theoTriebelRiemann} and  Lemma \ref{le:runst} to show that $\varphi = 1/\bigl(\widetilde \varphi\circ\psi\bigr) \in F^{s+1}_{p,q}(N\cap U) $.
  \end{proof}

 \smallskip

 
Last, the above analysis generalises from circle domains to finitely connected $B^{s+1-1/p}_{p,p}$-domains. However, 
 in taking compositions and inverses in  Triebel-Lizorkin spaces,  according to Theorem \ref{lemTriebelAdmissibleBanach} we need to restrict to spaces with non-integer smoothness.


\begin{corollary} \label{confB}
Suppose $s>0$ and $1<p<\infty$ with $sp>2$. If $\Omega$ and $\Omega'$ are bounded and finitely connected $B^{s+1-1/p}_{p,p}$-domains, and $h:\Omega \to \Omega'$  is a conformal homeomorphism, then 
\smallskip

a) \; $h \in W^{s+1,p}(\Omega)$.  

b) \; If $ s $ is non-integer 
and   $1 \leq q \leq \infty$, then $h \in F^{s+1}_{p,q}(\Omega)$.

c) \;  $h$ is bilipschitz.

\end{corollary}
\begin{proof} For a), if $\varphi_1: U_1 \to \Omega$ and $\varphi_2:=h\circ\varphi_1: U_1 \to \Omega'$ are conformal maps from the circle domain $U_1$,  then by Corollary \ref{finiteconn} both $\varphi_j$'s are bi-Lipschitz with $\varphi_j \in W^{s+1,p}(U_1)$. Thus  Theorem \ref{lemSobolevAdmissibleSpace} shows that
$h = \varphi_2 \circ \varphi_1^{-1} \in W^{s+1,p}(\Omega)$.

When $s\in \R_+ \setminus \N$  the same argument, now with  Theorem \ref{lemTriebelAdmissibleBanach}, proves b). Further, in both cases the decomposition $h = \varphi_2 \circ \varphi_1^{-1}$ shows that $h$
 is bi-Lipschitz.
 \end{proof}

\medskip

\noindent {\it \bf Proof of Theorem \ref{theoSobolevRiemann}}. \,  If $\varphi :\D \to \Omega$ is a Riemann map of a $B^{s+1-1/p}_{p,p}$-domain, then by Corollary \ref{confB} both $\varphi  \in W^{s+1,p}(\D)$ and  $\varphi^{-1}  \in W^{s+1,p}(\Omega)$. 

Conversely, assume that we have a Riemann map with $\varphi  \in W^{s+1,p}(\D)$ and $\varphi^{-1} \in W^{s+1,p}(\Omega)$. The first condition implies that $\varphi$ is Lipschitz continuous up to the boundary $\partial \D$. To see that it is even bi-Lipschitz, note that
 by Definition \ref{defRestrictionNorm},  $\varphi^{-1}$ is the restriction to $\Omega$ of a function belonging to $W^{s+1,p}(\C)$, in particular Lipschitz continuous in all of $\Omega$.

  Thus $\varphi$ provides the  parametrization of $\partial \Omega$, as required by Definition \ref{defBesovDomain}, for $\Omega$ to be a  $B^{s+1-1/p}_{p,p}$-domain. \hfill $\Box$
  
  \bigskip


In the above proof of Theorem \ref{theoSobolevRiemann} we used   Definition  \ref{defRestrictionNorm} for the fractional Sobolev space $W^{s+1,p}(\Omega)$, as the space of restrictions to $\Omega$ of functions in $W^{s+1,p}(\C)$. However,  there is a subtlety here, in that in case $s$ is an integer it is equally natural to consider $W^{s+1,p}(\Omega)$ as defined by  the 'intrinsic'  Definition \ref{defSobolevNonHomogeneous}, that is, to consider the space of functions $f$ in $\Omega$ with  the norm
 \begin{equation}\label{wnormi}
  \| f\|_{W^{s+1,p}(\Omega)} = \| f \|_{L^p(\Omega)} +  \sum_{k=1}^{s+1} \| D^{k} f \|_{L^p(\Omega)} < \infty.
 \end{equation}
In general domains $\Omega$ these two definitions of $W^{s+1,p}(\Omega)$ give different function spaces.
However, our Theorem \ref{theoSobolevRiemann} holds true independently of which definition is chosen. Indeed, in view of the above proof of Theorem \ref{theoSobolevRiemann}, to see this it suffices to show
\begin{proposition} \label{intrinsic} Suppose $1 \leq s \in \N$ and $1 < p < \infty$ with $sp > 2$, and that $\varphi : \D \to \Omega$ is a conformal map onto a bounded Jordan domain $\Omega$. If we have
$$ \sum_{k=1}^{s+1} \bigl( \| D^{k} \varphi \|_{L^p(\D)} + \| D^{k} \varphi^{-1} \|_{L^p(\Omega)} \bigr) < \infty 
$$
in terms of the norm \eqref{wnormi},  then  $\varphi $ is bi-Lipschitz.
\end{proposition}
\begin{proof}
First note that $\varphi' \in L^{\infty}(\D)$ by the Sobolev embedding \eqref{eqEmbeddingIntoDiniSOBO}, and thus $\varphi$ is Lipschitz continuous.

To show that $\varphi$ is bi-Lipschitz, suppose next $s=1$ which requires $p >2$. With chain rule,  differentiating  the basic identity $\bigl( \varphi^{-1}\bigr)' = (1/\varphi') \circ \varphi^{-1}$ one obtains 
 \begin{equation}\label{inverse}
 \bigl( \varphi^{-1}\bigr)'' = - \frac{\varphi''}{(\varphi')^3} \circ \varphi^{-1} = \frac{1}{2} \left( \frac{d}{dz}\frac{1}{(\varphi')^2}\right) \circ \varphi^{-1} . 
 \end{equation}
In particular,
$$ 
 \| (\log  \varphi')'\|^p_{L^p(\D)} \leq \| \varphi' \|_{\infty}^{2p-2} \int_\D \left|  \frac{\varphi''}{(\varphi')^3}  \right|^p |\varphi'|^2 \,dm = 
  \| \varphi' \|_{\infty}^{2p-2}  \int_\Omega \left|   \bigl( \varphi^{-1}\bigr)'' \right|^p  \,dm < \infty
$$
by our assumptions. When $p>2$ Morrey-Sobolev inequality, see  \cite[Section 5.6.2]{Evans}, shows that  $\log  \varphi'$ is H\"older continuous. Thus 
\cite[Theorem 3.5]{PommerenkeConformal} applies and proves that $\varphi$ is bi-Lipschitz. 

 In case  $2 \leq s \in \N$, any $1 < p < \infty$ is allowed. Now the above analysis does not suffice if $p <2$, and this requires us to estimate the third derivatives. Here it is convenient to   consider a non-linear differential operator, a variant of the classical {\it Schwarzian derivative}
$S_f = \left( \frac{f''}{f'} \right)' - \frac{1}{2} \left( \frac{f''}{f'} \right)^2$ of an analytic function  $f$. For our purposes we define, for $f$ analytic, 
$$ {\widehat S}_f :=  \left( \frac{f''}{f'} \right)' - 2 \left( \frac{f''}{f'} \right)^2.$$
As an amusing side note,  it is not difficult to see, e.g. from \eqref{schwarz} below, that ${\widehat S}_f \equiv 0 $ if and only if either $ f(z) = a\sqrt{z+b} +c$  or $f(z)=az+b$, for some $a,b,c \in \C$. In comparison, the classical Schwarzian derivative $S_f \equiv 0 $ if and only if $f$ is a M\"obius transform, i.e.  $f(z) = a \frac{1}{z+b} + c$ or $f(z)=az+b$, for some $a,b,c \in \C$. 

Next, via  an explicit differentiation and \eqref{inverse},
 \begin{equation}\label{schwarz}
  \frac{d^2}{dz^2}\frac{1}{(\varphi')^2} =  -2 \,{\widehat S}_\varphi \, \frac{1}{(\varphi')^2}, \quad {\rm and} \quad \bigl( \varphi^{-1}\bigr)''' = - \left(\frac{1}{(\varphi')^3} \, {\widehat S}_\varphi  \right) \circ \varphi^{-1}.
 \end{equation} 
 These identities show, first, that by our assumptions
 $$ \| {\widehat S}_\varphi \|_{L^p(\D)} \leq \| \varphi' \|_{\infty}^{3-2/p} \, \| \bigl( \varphi^{-1}\bigr)'''\|_{L^p(\Omega)} < \infty,
 $$
 and second, that the function $w(z) :=  1/(\varphi')^{2} $ is analytic in the unit disc $\D$ with
 \begin{equation}\label{linear}
 w''(z) = A(z) w(z), \quad z \in \D, \qquad {\rm where } \quad A=-2\widehat S_\varphi \in L^p(\D) \;  \; {\rm for \; some\; } 1 < p < \infty. 
  \end{equation} 
  In particular, here note that  by the mean value principle  and H\"older's inequality the coefficient $A(z)$ has the bound 
  $$
 | A(z)|=\frac{2}{\pi(1-|z|)^2}\Big|\int_{B(z,1-|z|)} A(z)dz \Big|\leq  2\pi^{1/p' -1}(1-|z|)^{2/p'-2}\|A\|_{L^{p}(\D)}, \qquad z\in\D.
  $$

  
  It remains to show that \eqref{linear} forces $1/\varphi'$ to be bounded. Once this is done, the bi-Lipschitz property follows from  \cite[Theorem 3.5]{PommerenkeConformal}.   
  For readers convenience we formulate this remaining step in terms of  a separate lemma. 

  
  \begin{lemma}\label{le:comp} Assume that $1<\alpha<2$ and $u:[0,1)\to\C$ is a $C^2$-function that satisfies on $[0,1)$ the differential equation
  \begin{equation*}
  u''(t)= h(t)u(t). \qquad 
  \end{equation*}
  If $h$ has  the bound 
    \begin{equation*}
 |h(t)|\leq c(1-t)^{-\alpha}, \qquad t\in[0,1),
  \end{equation*}
  then the solution $u$ is bounded, with $\|u\|_{L^\infty[0,1)}\leq C(c,\alpha, |u(0)|, |u'(0)|)$.
  \end{lemma}
  \noindent {\it Proof.} We first note a  comparison property: Suppose $v$ is a solution to
    \begin{equation*}
  v''(t)= H(t)v(t)\qquad \textrm{with}\quad \qquad H(t)\geq c(1-t)^{-a},
  \end{equation*}
 where  $v(0)>|u(0)|$ and $v'(0)>|u'(0)|$ (in particular, $v$, $v'$ and $H$ are all  real-valued and positive). Then
  $|u(t)|<v(t)$ for all $t\in (0,1)$.   
  
  To see this, note that by the initial conditions  $|u(t)|  < v(t)$ for  all small enough $t>0$. 
  Hence consider
   $$ t_1 := \sup\{ t \in [0,1) :  |u(s)|<v(s)\;\; \; \forall \; 0  \leq s < t \}.
  $$

  
  
 \noindent  Choose then $0 < \delta <  v'(0) - |u'(0)|$. As $|u|\leq v$ on $[0,t_1]$ we  obtain
  \begin{eqnarray*}
  |u'(s)|&=&| u'(0)+\int_0^{s}h(t)u(t)dt| < v'(0)+ \int_0^{s}H(t)v(t)dt - \delta =v'(s) - \delta, \quad 0 \leq s \leq t_1.
  \end{eqnarray*}
Another integration gives then $|u(t_1)| \leq v(t_1) -\delta t_1 < v(t_1)$. Thus $t_1 = 1$, proving the comparison property. 

The statement of the lemma now follows by applying  the above comparison to the function $Bv(t)$, where $v(t):=\exp\big(-c'(1-t)^{2-\alpha}\big)$ with $c':=c(2-\alpha)^{-1}(\alpha-1)^{-1}$, by choosing the constant $B$  large enough. This works, since $v(0), v'(0)>0$  and we have
  $$
  H(t):=\frac{v''(t)}{v(t)}=c(1-t)^{-\alpha} +\frac{c^2 }{(\alpha-1)^2}(1-t)^{2(1-\alpha)}\geq c(1-t)^{-\alpha}.
  $$
\smallskip
  
 Finally, with Lemma \ref{le:comp}  now proven,  
 Proposition \ref{intrinsic} follows also in the case $1 < p < 2$, by applying the Lemma on each radius of the unit disc to the function $1/(\varphi')^2$,  which  satisfies the complex differential equation \eqref{linear}. We obtain an upper bound independent of the radius, whence  the function $1/(\varphi')^2$ is bounded on the unit disc.
    \end{proof}

\section{Quasiconformal mappings on domains}\label{secSobolev}  

We next turn to proving 
Theorems \ref{theoSobolevStabilityOfDomains}, \ref{theoSobolevRiemannKaksi} and \ref{theoTriebelStabilityOfDomains}, which describe  the global Sobolev and Triebel-Lizorkin regularity  for quasiconformal mappings $f:\Omega \to \Omega'$ between two $B^{s+1-\frac1p}_{p,p}$-domains. We begin with necessary  results on traces. 

In $C^\infty$-domains it is well known that the trace spaces of Sobolev functions (and of Triebel-Lizorkin functions) are the diagonal Besov spaces with a decay of $1/p$ in the smoothness parameter. Using properties of the Riemann mapping  we can now recover 
the classical trace relations from the $C^\infty$-domains to the  domains with only  $B^{s+1-\frac1p}_{p,p}$-regularity.

\begin{lemma}\label{lemSobolevTrace}
Let $s\in\N_0 =\N \cup \{0\}$ and $1 < p<\infty$ with $sp>2$. If $\Omega$ is a finitely connected $B^{s+1-\frac1p}_{p,p}$-domain and $g\in W^{s+1,p}(\Omega)$, then  $g|_{\partial\Omega} \in B^{s+1-\frac1p}_{p,p}(\partial\Omega)$. 
\end{lemma}
\begin{proof}
First assume that $\Omega$ is a simply connected $B^{s+1-1/p}_{p,p}$-domain and let $g\in W^{s+1,p}(\Omega)$. Consider a Riemann mapping $\varphi:\D \to \Omega$. By Theorem \ref{theoTriebelRiemann}, the mapping $\varphi$ lies in $ W^{s+1,p}(\Omega)$ and  is bi-Lipschitz. Thus by \rf{eqInvarianceUnderbiLipschitzSOBO} we have $g \circ \varphi \in W^{s+1,p}(\D)$ as well. 

 It follows from  \cite[Theorem { 3.3.3}]{TriebelTheory} that  the trace space of $W^{s+1,p}(\D)$ equals $B^{s+1-1/p}_{p,p}(\T)$.
Hence the restriction $(g \circ \varphi)|_{\T} \in B^{s+1-1/p}_{p,p}(\T)$. Since both functions are continuous up to the boundary, the trace is defined pointwise, and $(g \circ \varphi)|_{\T}=g|_{\partial\Omega} \circ \varphi|_{\T}$. By Definition \ref{defTraceSpace}, this means that $g\in B^{s+1-1/p}_{p,p}(\partial\Omega)$. 

Assume next that $\Omega$ is a finitely connected $B^{s+1-1/p}_{p,p}$-domain and let $g\in W^{s+1,p}(\Omega)$. Since $\Omega$ is a Sobolev extension domain, there is a compactly supported function $Eg \in W^{s+1,p}(\C)$ which coincides with $g$ in $\Omega$ and, therefore, it has the same trace in $\partial\Omega$. 

Consider then one of the boundary components $\Gamma_j$ of $\partial\Omega$ and let $\Omega_j$ be the bounded simply connected domain with boundary $\Gamma_j$.  
Further, let $g_j:= Eg|_{\Omega_j}$. Then, as we have shown above, $g|_{\Gamma_j} \equiv g_j|_{\partial\Omega_j} \in B^{s+1-1/p}_{p,p}(\partial\Omega_j) \equiv B^{s+1-1/p}_{p,p}(\Gamma_j)$.
Since this happens with all the components of the boundary of $\Omega$, it follows that 
$g|_{\partial\Omega} \in B^{s+1-\frac1p}_{p,p}(\partial\Omega).$
\end{proof}

  A similar argument applies in the spaces  $F^{s+1}_{p,q}(\Omega)$. However, now we use Lemma \ref{lemTriebelAdmissibleBanach} instead of Lemma \ref{lemSobolevAdmissibleSpace}, and thus require a non-integer smoothness parameter $s$.

\begin{lemma}\label{lemTriebelTrace}
Let $ s \in \R_+ \setminus \N$, $1 < p <\infty$ and $1\leq q \leq \infty$  with $s p>2$. 
 If $\Omega\subset\C$ is a finitely connected $B^{s+1-\frac1p}_{p,p}$-domain and $g\in F^{s+1}_{p,q}(\Omega)$, then  $g|_{\partial\Omega} \in B^{s+1-\frac1p}_{p,p}(\partial\Omega)$.
 \end{lemma} 
 
  

\bigskip

Returning then to the proof of Theorem \ref{theoSobolevStabilityOfDomains}, assume  
 that  $f:\Omega \to \Omega'$ is a  quasiconformal mapping between two $B^{s+1-\frac1p}_{p,p}$-domains, with $\mu_f \in W^{s,p}(\Omega)$.

Since by assumption  $\Omega$ is a Lipschitz-domain one can extend $\mu_f$ to a compactly supported Beltrami coefficient $\tilde \mu \in W^{s,p}(\C)$. 
According to  \cite[Theorem 1.1]{CruzMateuOrobitg} one has  $F \in W^{s+1,p}_{loc}(\C)$ for any quasiconformal map with Beltrami coefficient  $\tilde \mu$. In particular,  $F|_\Omega \in  W^{s+1,p}(\Omega)$. 

Next, as $ \tilde \mu \in C^1$, the Jacobian $J(z,F)$ does not vanish \cite[p. 167]{AstalaIwaniecMartin},  and therefore $F$  is also bi-Lipschitz  on compact subsets of the plane. Moreover, from Lemma \ref{lemSobolevTrace} we see that the boundary values $F|_{\partial\Omega} \in B^{s+1-\frac1p}_{p,p}(\partial \Omega)$. A glance at the Definitions \ref{defBesovDomain} and \ref{defTraceSpace} shows that thus also $F(\Omega)$ is a $B^{s+1-\frac1p}_{p,p}$-domain. 

To connect $F$ with the original mapping  $f:\Omega \to \Omega'$ one applies Stoilow's factorization which gives $f = h \circ F$, where $h: F(\Omega) \to \Omega'$ is conformal. Here $h \in W^{s+1,p}\bigl(F(\Omega)\bigr)$ by Corollary \ref{confB}, while 
Corollary \ref{lemSobolevAdmissibleSpace} shows that $f = h \circ F \in W^{s+1,p}(\Omega)$. This completes the proof of Theorem \ref{theoSobolevStabilityOfDomains}. 
\smallskip

For the  proof of Theorem \ref{theoSobolevRiemannKaksi}, note first that if $g: \D \to \Omega$ is a quasiconformal homeomorphism onto a $B^{s+1-\frac1p}_{p,p}$-domain with Beltrami coefficient $\mu \in W^{s,p}(\D)$, then Theorem \ref{theoSobolevStabilityOfDomains} shows that $g \in W^{s+1,p}(\D)$. Also, the factorisation $g = h \circ F$ with Corollary \ref{confB} c) and the proof of 
Theorem \ref{theoSobolevStabilityOfDomains} shows that $g$ is bi-Lipschitz. Conversely, for a quasiconformal $g \in W^{s+1,p}(\D)$ the  Beltrami coefficient $\mu_g \in W^{s,p}(\D)$, and if in addition $g$ is bi-Lipschitz, then Lemma \ref{lemSobolevTrace} shows that $g|_{\T}$ gives a parametrization of $\partial \Omega$, required by  Definition \ref{defBesovDomain} to make $\Omega$ a $B^{s+1-\frac1p}_{p,p}$-domain.

\bigskip

Last,  the proof of  Theorem \ref{theoTriebelStabilityOfDomains}, concerning the case where $f:\Omega \to \Omega'$ is quasiconfomal with $\mu_f \in F^s_{p,q}(\Omega)$ and  $ s \in \R_+ \setminus \N$, is again similar. By assumptions of the Theorem, $\Omega$ and  $\Omega'$ are $B^{s+1-\frac1p}_{p,p}$-domains, thus $\mu_f$ extends to $\tilde \mu \in F^s_{p,q}(\C)$,  and if $F: \C \to \C$ is quasiconformal with Beltrami coefficient  $\tilde \mu$, then   \cite[Theorem 1.1]{CruzMateuOrobitg} implies that $F \in (F^{s+1}_{p,q})_{ loc}(\C)$. 

As in the proof of Theorem  \ref{theoSobolevStabilityOfDomains} we see that $F$ is bi-Lipschitz in $\Omega$, and using then Lemma \ref{lemTriebelTrace} it follows that $F(\Omega)$ is a  $B^{s+1-\frac1p}_{p,p}$-domain. Finally, in the Stoilow factorisation $f = h \circ F$ the conformal factor $h \in F^{s+1}_{p,q}\bigl(F(\Omega)\bigr)$ by Corollary \ref{confB} b), so that  the 
proof of Theorem  \ref{theoTriebelStabilityOfDomains}  is completed via Theorem \ref{lemTriebelAdmissibleBanach}. 

 \bigskip
 
  To compare the required steps the reader may  use the following  dictionary:
\begin{center}
\begin{tabular}{ |c|c|c| } 
 \hline
  Sobolev context  & Triebel-Lizorkin context\\
 \hline
  $W^{s,p}(\Omega)$   & $F^s_{p,q}(\Omega)$  \\ 
  $W^{s+1,p}(\Omega)$  & $F^{s+1}_{p,q}(\Omega)$\\ 
  $B^{s+1-\frac1p}_{p,p}(\partial\Omega)$ & $B^{s+1-\frac1p}_{p,p}(\partial\Omega)$\\ 
 \hline
\end{tabular}
\end{center}

\section{Proof of Theorem \ref{theoInvertBeltrami}}\label{secPrincipal}
In this section we prove Theorem \ref{theoInvertBeltrami}. We will write $\Beurling_{\Omega}=\chi_{\Omega} \Beurling( \chi_\Omega \cdot)$ and $\Beurling_{\Omega,\Omega^c}=\chi_{\overline{\Omega}^c} \Beurling( \chi_\Omega \cdot)$, and similarly for the Cauchy transform or any other operator acting on functions defined in $\C$. In Section \ref{secOutlineBeltrami} we outline its proof, which follows the steps of \cite{PratsQuasiconformal} by means of a classical Fredholm argument, reducing the proof to checking that $I_\Omega-\mu^m (\Beurling^m)_\Omega$ is invertible, and that the commutator $[\mu,\Beurling_\Omega]$ and the Beurling reflection  $\chi_\Omega \Beurling (\chi_{\Omega^c} \Beurling (\chi_\Omega \cdot))$ are compact (together with a family of related operators).

After that we recall Dorronsoro's Betas in Section \ref{secBetas}, which will be our tool to measure the flatness of the boundary in a multi-scale basis. Next we show that the iterates of the truncated Beurling transform are bounded with subexponential  growth (polynomial in fact) in Section \ref{secIterates}, which will allow us to isolate the invertible part of the Fredholm operator. We check the compactness of the commutator  in Section \ref{secCommutator}. Finally we prove the compactness of the Beurling Reflection in Section \ref{secReflection} in what represents the most difficult challenge in this paper, leaving a technical lemma to be shown in Section \ref{secMeyer} where Meyer's polynomials are introduced to control oscillation in Whitney cubes. Some additional details of the structure of this section were already given in Section \ref{subse:structure}.

{
There are several novelties in this section worthy to mention. First of all, the approximation of the boundary smoothness via betas is not new, stemming from the work of Dorronsoro and adapted to the measurement of boundary Besov regularity by Cruz and Tolsa, but in Definition \ref{defBetasBoundary} we introduce an approach which makes it more easy to use than in  previous papers, defining a beta coefficient directly related to the Whitney cube. This does not imply a change in the techniques, but  it makes the proofs more readable, since we don't need to switch every time to a local parameterization of the boundary and back. This usage is summarized by Lemmata \ref{lemNormalVectorBetas}, \ref{lemBoundForBeurling} and \ref{lemSummingTheBetas}, which {will be used} to control many  {key quantites} related to boundary.}


{{All in all,} the main result of this section is 
the compactness of the Beurling reflection. In \cite{PratsQuasiconformal} the proof for the Sobolev scale  with {integer  degrees of smoothness is quite cumbersome, and does not extend to the case we have here. Here, instead we prove directly}  the smoothing estimate 
$$\norm{\chi_\Omega \Beurling (\chi_{\Omega^c} \Beurling^m (\chi_\Omega f)) }_{\dot F^s_{p,q}(\Omega)} \lesssim_h \norm{f}_{\mathcal{C}^h(\Omega)}.$$
The authors believe that this approach may help to simplify the proof in the Sobolev scale with natural smoothness as well. 

To obtain {the above  estimate} we have to use the beta lemmata again for the local part, but we approach the nonlocal part {via} an explicit expression for the kernels of these reflections. 
{Using} this particular expression, we rewrite the reflection as a main term which is the product of the iterated Beurling transform times the derivative of the Beurling transform of the characteristic function plus a finite sum of terms whose kernels involve the Taylor errors of the antiderivatives of the iterated Beurling transform $\bar\partial^{-(m-j)}\Beurling^j\chi_\Omega$. 
{These error terms are controlled using  
a number of techniques developed by the second and third authors in previous works, the main novelty being the proof of Lemma \ref{lemTaylorToBesov}. There we substitute the Taylor approximation by Meyer's polynomials, in order to apply Poincar\'e inequalities.} 
}

\subsection{Proof  of Theorem \ref{theoInvertBeltrami} modulo key auxiliary results}\label{secOutlineBeltrami}

{Our aim is to establish the invertibility} of $(I-\mu \Beurling)(\chi_\Omega \cdot)$ in $F^s_{p,q}(\Omega)$. Here  $\chi_\Omega g$ denotes the extension of  a given function $g\in F^s_{p,q}(\Omega)$ by zero in $\Omega^c$. We will follow the scheme used in \cite{Iwaniec}. That is, we will reduce the proof to the compactness of the commutator. In our context, however, as it happens in \cite{CruzMateuOrobitg} and \cite{PratsQuasiconformal}, we will have to deal with the compactness of operators like $\chi_\Omega \Beurling \left(\chi_{\Omega^c} \Beurling\left(\chi_\Omega\cdot \right)\right)$ as well. 

Consider $m\in \N$. Recall that $(\Beurling^m)_\Omega g =\chi_\Omega \Beurling^m(\chi_\Omega g)$ for $g \in L^p(\Omega)$  and $I_\Omega$ be the identity on $F^s_{p,q}(\Omega)$.  
Let us define 
$P_m:=I_\Omega+\mu \Beurling_\Omega+(\mu \Beurling_\Omega)^2+\cdots + (\mu \Beurling_\Omega)^{m-1}$. We will check that the truncated Beurling transform is bounded on $F^s_{p,q}(\Omega)$ in Theorem \ref{theoIterates} below. Since $F^s_{p,q}(\Omega)$ is a multiplicative algebra (under the conditions of Lemma \ref{lemAlgebra}), we have that $P_m$ is bounded in $F^s_{p,q}(\Omega)$. Note that
\begin{equation}\label{eqPmToIMu}
P_m\circ (I_\Omega-\mu \Beurling_\Omega)=(I_\Omega-\mu \Beurling_\Omega)\circ P_m=I_\Omega-(\mu \Beurling_\Omega)^m,
\end{equation}
and
\begin{align}\label{eqFactorizePm}
I_\Omega-(\mu \Beurling_\Omega)^m
\nonumber	&=(I_\Omega-\mu^m(\Beurling^m)_\Omega)+\mu^m((\Beurling^m)_\Omega-(\Beurling_\Omega)^m)+(\mu^m (\Beurling_\Omega)^m-(\mu \Beurling_\Omega)^m)\\
			& =A^{(1)}_m+\mu^m A^{(2)}_m+ A^{(3)}_m.
\end{align}
Note the difference between $(\Beurling_\Omega)^m g= \chi_\Omega \Beurling(\dots\chi_\Omega \Beurling(\chi_\Omega \Beurling(\chi_\Omega g)))$ and $(\Beurling^m)_\Omega g =\chi_\Omega \Beurling^m(\chi_\Omega g)$.  We want to check that for $m$ large enough, the operator $I_\Omega-(\mu \Beurling_\Omega)^m$ is the sum of an invertible operator and a compact one.

First we will study the compactness of $A^{(3)}_m=\mu^m (\Beurling_\Omega)^m-(\mu \Beurling_\Omega)^m$. To start, writing $[\mu, \Beurling_\Omega](\cdot)$ for the commutator $\mu \Beurling_\Omega(\cdot)-\Beurling_\Omega(\mu \cdot)$ we have the telescopic sum
\begin{align*}
A^{(3)}_m	
			& =\sum_{j=1}^{m-1}\mu^{m-j}[\mu, \Beurling_\Omega]\left(\mu^{j-1}(\Beurling_\Omega)^{m-1}\right)+(\mu \Beurling_\Omega) A^{(3)}_{m-1}.
\end{align*}
Arguing by induction we can see that $A^{(3)}_m $ can be expressed as a sum of operators bounded in $F^s_{p,q}(\Omega)$ which have $[\mu, \Beurling_\Omega]$ as a factor. It is well-known that the compactness of a factor implies the compactness of the operator (see for instance \cite[Section 4.3]{Schechter}). { In \emph{our first key auxiliary result,} Lemma \ref{lemCompactnessCommutator} below}, we verify that the commutator $[\mu, \Beurling_\Omega]$ is compact in $F^s_{p,q}(\Omega)$.

Consider now $A^{(2)}_m=(\Beurling^m)_\Omega-(\Beurling_\Omega)^m$, which is bounded in $F^s_{p,q}(\Omega)$ again by Theorem \ref{theoIterates} below. We define the operator 
\begin{equation}\label{eqDefinitionReflection}
\mathcal{R}_m g:=\chi_\Omega \Beurling\left(\chi_{\Omega^c} \Beurling^{m}(\chi_\Omega \, g)\right)
\end{equation} 
{which is well defined for instance if $g\in L^p(\C)$ with $p>1$}. This operator can be understood as a (regularizing) reflection with respect to the boundary of $\Omega$. Note that $\mathcal{R}_{m-1}=\chi_\Omega [\chi_\Omega,\Beurling]\circ \Beurling^{m-1}(\chi_\Omega\cdot)$, leading to $A^{(2)}_m= \mathcal{R}_{m-1}+ \Beurling_\Omega \circ A^{(2)}_{m-1}$. Thus, the reflection is bounded and the compactness of $\mathcal{R}_m$ shown in { \emph{our second key auxiliary result,} Theorem \ref{theoCompactnessReflection} below, } will prove the compactness of $A^{(2)}_m$. 

Now, the following claim is the remaining ingredient for the proof of Theorem \ref{theoInvertBeltrami}.
\begin{lemma}\label{lemInvertible}
 For $m$ large enough, $A^{(1)}_m $ is invertible.
\end{lemma}
\begin{proof}
Since $sp>2$ we can use the algebra structure \rf{eqAlgebra} and \rf{eqAlgebraPower} to conclude that 
for every $g\in F^s_{p,q}(\Omega)$
\begin{align*}
\norm{\mu^m (\Beurling^m)_\Omega g}_{F^s_{p,q}(\Omega)}
	& \lesssim \norm{\mu^m}_{F^s_{p,q}(\Omega)}\norm{(\Beurling^{m})_\Omega g}_{F^s_{p,q}(\Omega)}\\
	& \lesssim m^{N} \norm{\mu}^{m-1}_{L^\infty}\norm{\mu}_{F^s_{p,q}(\Omega)} \norm{(\Beurling^m)_\Omega}_{F^s_{p,q}(\Omega)\to F^s_{p,q}(\Omega)}\norm{g}_{F^s_{p,q}(\Omega)}.
\end{align*}

By  {  \emph{our third key auxiliary result,}} Theorem \ref{theoIterates} below, there are constants depending on the Lipschitz character of $\Omega$ (and other parameters) but not on $m$, such that 
$$\norm{(\Beurling^m)_\Omega}_{F^s_{p,q}(\Omega)\to F^s_{p,q}(\Omega)} \lesssim m^{2} \norm{\nu}_{B^{s-1/p}_{p,p}(\partial\Omega)}.$$
As a consequence, for $m$ large enough the operator norm $\norm{\mu^m (\Beurling^m)_\Omega}_{F^s_{p,q}(\Omega)\to F^s_{p,q}(\Omega)}<1$ and, thus, $A^{(1)}_m$ in \rf{eqFactorizePm} is invertible.
\end{proof}

Now we can {complete the proof} of Theorem \ref{theoInvertBeltrami} for $0<s<1$ by the usual Fredholm argument as follows.

\begin{proof}[Proof of Theorem \ref{theoInvertBeltrami}]
For $m$ big enough, the restricted Beltrami operator $I_\Omega-(\mu \Beurling_\Omega)^m$ can be expressed using \rf{eqFactorizePm} as the sum of an invertible operator $A^{(1)}_m$ (see Lemma \ref{lemInvertible}) and the compact operator $\mu^m A^{(2)}_m+ A^{(3)}_m$ (its compactness granted in the comments above together with Lemma \ref{lemCompactnessCommutator} and Theorem \ref{theoCompactnessReflection}). By \rf{eqPmToIMu}, we can deduce that $I_\Omega - \mu \Beurling_\Omega$ is a  Fredholm operator (see \cite[Theorem 5.5]{Schechter}). The same argument works with any other operator $I_\Omega-t \mu \Beurling_\Omega$ for $0<t<1/\norm{\mu}_\infty$. It is well known that the Fredholm index is continuous with respect to the operator norm on Fredholm operators (see \cite[Theorem 5.11]{Schechter}), so the index of $I_\Omega-\mu \Beurling_\Omega$ equals that of $I_\Omega$, i.e., $0$.

It only remains to see that our operator is injective in order to obtain its invertibility. Since the Beurling transform is an isometry on $L^2(\C)$ and $\norm{\mu}_\infty<1$, the operator $I-\mu \Beurling$ is injective in $L^2(\C)$. Thus, if $g\in F^s_{p,q}(\Omega)$, and $(I_\Omega - \mu \Beurling_\Omega)g=0$, we define $G(z)= g(z)$ if $z\in\Omega$ and $G(z)=0$ otherwise, and then we have that 
$$(I - \mu \Beurling)G(z)=(I - \mu \chi_\Omega \Beurling)(\chi_\Omega G)(z)=
\begin{cases}
(I_\Omega - \mu \Beurling_\Omega)g(z) =0  & \mbox{when } z\in\Omega \\
0 & \mbox{otherwise}.
\end{cases}$$
 By the injectivity of the first operator, since $G\in L^2$ we get that $G=0$ and, thus, $g=0$ as a function of $F^s_{p,q}(\Omega)$.
 
Now, remember that the principal solution of \rf{eqBeltrami} is $f(z)=\Cauchy h(z)+z$, where
$$h:=(I-\mu \Beurling)^{-1} \mu,$$
that is, $h - \mu \Beurling(h)=\mu$, so $\supp(h)\subset \supp(\mu)\subset \overline{\Omega}$ and, thus, $\chi_\Omega h + \mu \Beurling_\Omega(h)=h + \mu \Beurling(h)=\mu$ modulo null sets, so 
 $$h|_\Omega=(I_\Omega-\mu \Beurling_\Omega)^{-1} \mu,$$
proving that $h|_{\Omega} \in F^s_{p,q}(\Omega)$. By \cite[Theorem 4.3.12]{AstalaIwaniecMartin} we have that $\Cauchy h\in L^p(\C)$. Since the  derivatives of the principal solution, $\overline{\partial}f|_{\Omega}=h|_{\Omega}$ and $\partial f |_{\Omega}= \Beurling_\Omega h+ 1$, are in $F^s_{p,q}(\Omega)$, we have $f\in F^{s+1}_{p,q}(\Omega)$ by \rf{eqExtensionDomainXTRIEBEL}.
\end{proof}

\medskip

{The remaining subsections below are devoted to establishing our {three} key auxiliary results needed,  Lemma \ref{lemCompactnessCommutator} together with  Theorems   \ref{theoCompactnessReflection} and  \ref{theoIterates} below.}

\subsection{Dorronsoro's Betas}\label{secBetas}

Given two sets $A$ and $B$, their {\rm symmetric difference} is $A\Delta B:=(A\cup B) \setminus (A \cap B)$ and their long distance is
\begin{equation}\label{eqLongDistance}
\Dist(A,B):=\diam(A)+\diam(B)+\dist(A,B).
\end{equation}
{Let $\mathcal{D}_d$ stand for a dyadic grid of $\R^d$. }

\begin{definition}\label{defWhitney}
Given a domain $\Omega$, we say that a collection of open dyadic cubes $\mathcal{W}{\subset\mathcal{D}_d}$ is a {\em Whitney covering} of $\Omega$ if the cubes are disjoint,  $\Omega=\bigcup_{Q\in\mathcal{W}}\overline{Q}$, there exists a constant $C_{\mathcal{W}}$ such that 
$$C_\mathcal{W} \ell(Q)\leq \dist(Q, \partial\Omega)\leq 4C_\mathcal{W}\ell(Q),$$
and the family $\{50 Q\}_{Q\in\mathcal{W}}$ has a finite superposition property. Moreover, we will assume that 
\begin{equation*}
S\subset 5Q \implies \ell(S)\geq \frac12 \ell(Q).
\end{equation*}
\end{definition}
The existence of such a covering is granted for any open set different from $\R^d$ and in particular for any domain as long as $C_\mathcal{W}$ is big enough (see \cite[Chapter 1]{SteinPetit} for instance). Note that $C_{\mathcal{W}}$ may be increased if needed for our purposes by dividing each cube into its dyadic sons, for instance.

Jose R. Dorronsoro introduced the following polynomials to study the Besov norms of functions in  \cite{Dorronsoro} (see \cite[Proposition 2.3]{PratsPlanarDomains} for the consistency of this definition):
\begin{definition}
Let $I$ be an interval and let $f\in L^1_{\rm loc}(3I)$. Then,  there exists a unique polynomial $\mathbf{R}^n_I f$ of degree $n$ (or smaller) such that for every $j \in \{0,1,\cdots, n\}$, 
$$\int_I (\mathbf{R}^n_I f -f) x^j\, dm =0,$$
{see Figure \ref{figBeta2}.}
Then, we define
$$\beta_{(n)}(f,I):=\frac{1}{\ell(I)}\int_{3I}\frac{|f(x)-\mathbf{R}^n_I f (x)|}{\ell(I)} \, dm(x).$$
\end{definition}

\begin{figure}[ht]
 \centering
\begin{subfigure}{0.5 \textwidth}
 \centering{\includegraphics[width=\textwidth]{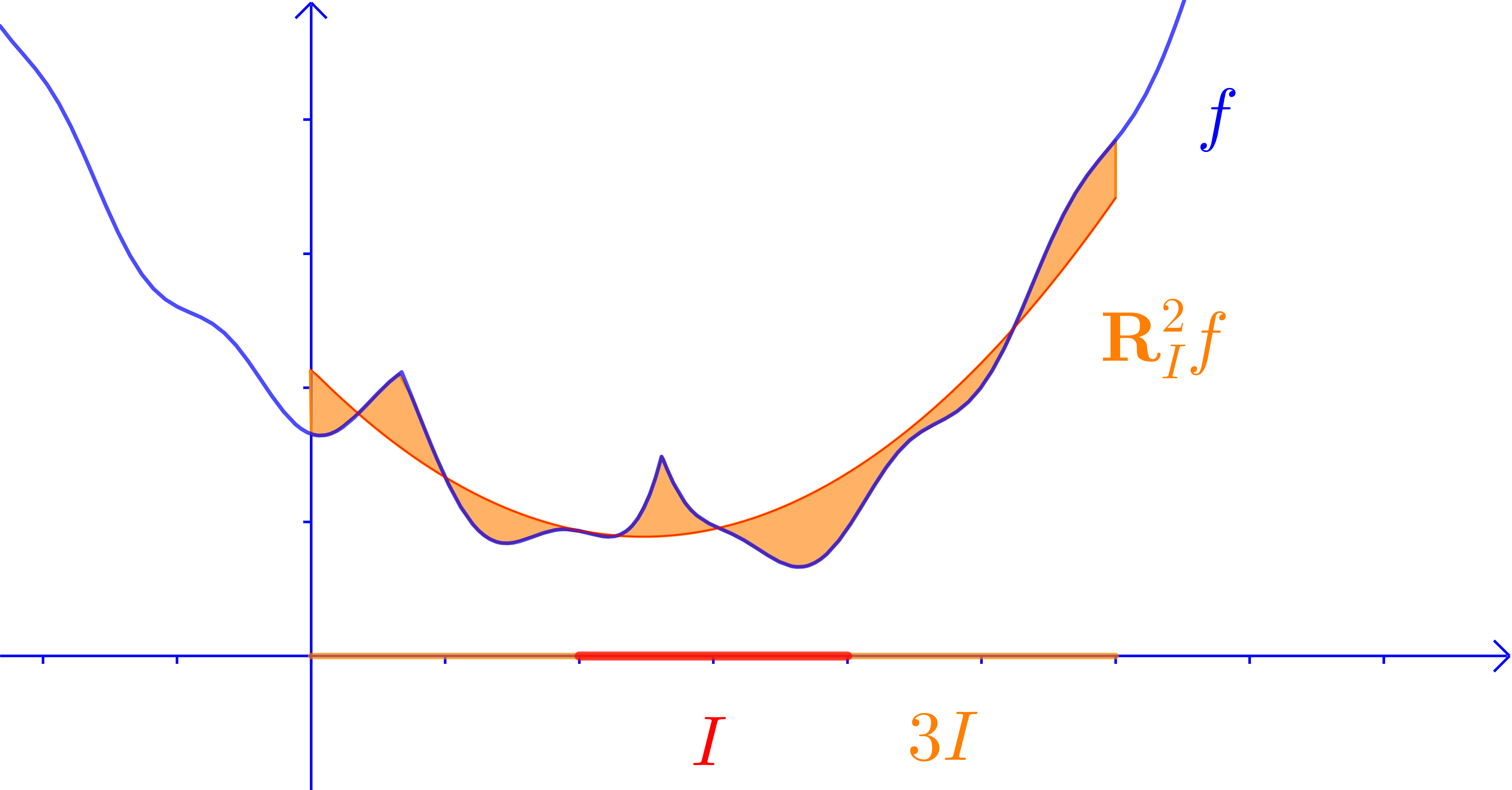}}
\end{subfigure}
\caption{{The minimal normalized area comprised between a degree 2 polynomial and the function $f$ in the interval $3I$ is  $\beta_{(2)}(f,I)$.}}\label{figBeta2}
\end{figure}

The $\beta$-coefficients are closely related to Jones-David-Semmes ones. Namely, if $f$ is Lipschitz and  $n=1$, then $\beta_{(1)} \approx \beta_1$. 
On the other hand, these polynomials satisfy that 
\begin{equation}\label{eqDorronsoroStayAway}
\norm{\mathbf{R}^n_I f}_{L^\infty(I)}\lesssim_n  \ell(I)^{-1} \norm{f}_{L^1(I)}.
\end{equation}

As it was observed in \cite[(2.10)]{PratsPlanarDomains}, one can rewrite \cite[Theorem 1]{Dorronsoro} in terms of these coefficients as follows. 
\begin{lemma}\label{lemDorronsoro}
Let $0<s<n+1$ and $1\leq p< \infty$. Then for every $f\in \dot B^s_{p,p}(\R)$, we have that
$$\norm{f}_{\dot{B}^s_{p,p}}\approx \left(\sum_{I\in\mathcal{D}_1}\left(\frac{\beta_{(n)}(f,I)}{\ell(I)^{s-1}}\right)^p \ell(I)\right)^{\frac1p}.$$
\end{lemma}

We will use the beta coefficients to measure the regularity of a domain. Namely, we measure  in every scale  how far is  each portion of the boundary to be the graph of a polynomial, via a dyadic approach. To make the notation less dense, we will assign the coefficients to the Whitney cubes straight ahead. To do so, we will chose a beta coefficient comparable to the supremum of the betas of the reasonable choices for each cube. In the following definitions and computation we use $\varepsilon_\delta:=\frac{R}{\sqrt{1+\delta^2}}$, which grants that whenever $x<\varepsilon_\delta$, the image $(x,A(x))$ under a parameterization $A$ is a boundary point even if the Lipschitz constant $\delta$ is big. 

Let $\Omega\subset \C$ be a bounded $(\delta,R)$-Lipschitz domain. {We will consider a given} finite collection of boundary points ${\mathcal{X}=}\{x_j \}_{j=1}^M\subset \partial\Omega$ such that $\left\{B\left(x_j,\frac{\varepsilon_\delta}{12}\right)\right\}_{j=1}^M$ is a disjoint family but the double balls cover the boundary. After an appropriate rigid movement  $\tau_j$ (rotation and translation) which maps $x_j$ to the origin, the boundary $\partial (\tau_j\Omega)$ coincides with the graph of a Lipschitz function $A_j$ in the cube $\mathcal{Q}_j=Q(0,R)$, with $A_{j}$ supported in $[-4R,4R]$ and derivative satisfying $\norm{A'}_{L^\infty}\leq\delta$. 

\begin{definition}\label{defBetasBoundary}
Given a cube $Q \in \W_\Omega$ with $\ell(Q)$ small enough, say $\ell(Q)\leq C_\Omega$, we say that a pair $(x_j, J){\in \mathcal{X}\times \mathcal{D}_1}$ is admissible for $Q$ {(see Figure \ref{figBetasSelect})}, writing $(x_j,J) \in \mathcal{J}_Q$ if 
\begin{enumerate}
\item The length  $\ell(J) = \ell(Q)$ and $J\subset [-\frac {\varepsilon_\delta}{3},\frac {\varepsilon_\delta}{3}]$.
\item The image of $J$ under the graph function $\tau_j^{-1}\circ(Id,A_j)$ is a set $U_J\subset \partial\Omega$ with {\emph{long distance}} $\Dist(Q, U_J)\approx \ell(Q)$, {see \rf{eqLongDistance}}.
\end{enumerate}
Then, we assign the number 
$$\beta_{(n)}(Q) := {\min_{(x_j, J)\in \mathcal{J}_Q}} \beta_{(n)}(A_j,J).$$
If the cube is greater than $C_\Omega$, we will assign $\beta_{(n)}(Q):=1$. 
\end{definition}

\begin{figure}[ht]
 \centering
\begin{subfigure}{0.8 \textwidth}
 \centering{\includegraphics[width=\textwidth]{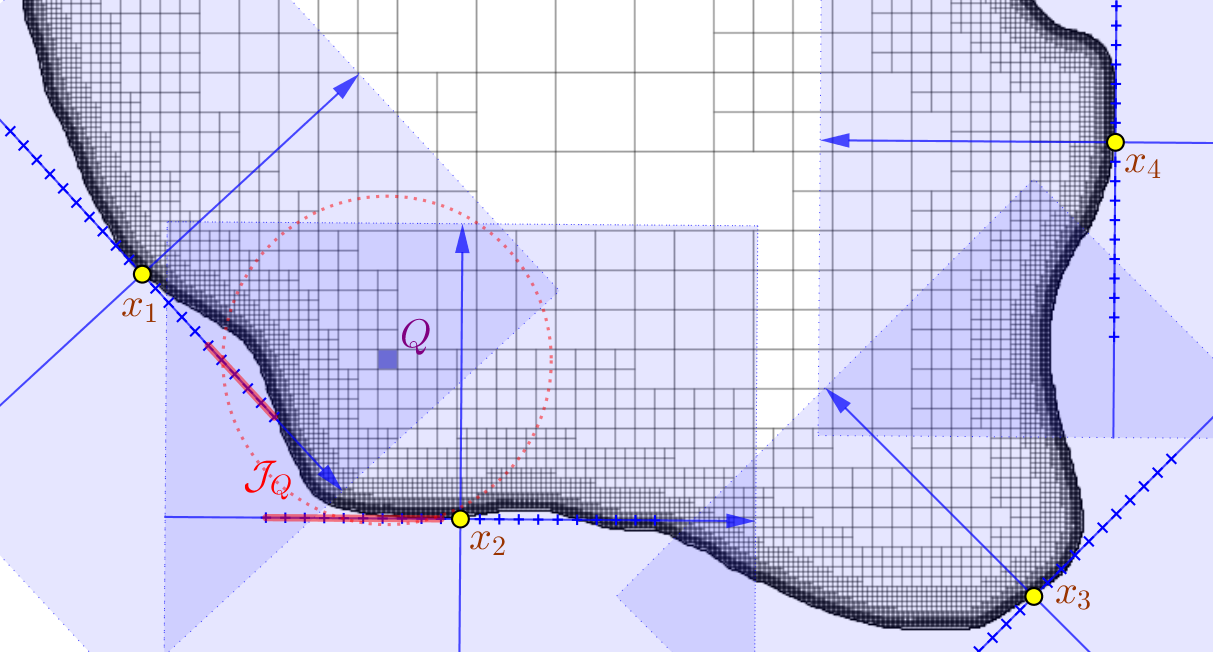}}
\end{subfigure}
\caption{{The candidate intervals in $\mathcal{J}_Q$ can be identified with the bold red segments, see Definition \ref{defBetasBoundary}, which correspond to only a bounded number of different local parameterizations of the boundary. }}\label{figBetasSelect}
\end{figure}

\begin{remark}\label{remUniformBounds}
Note that the number of candidates $J$ above is uniformly bounded in terms of the Lipschitz character and the Whitney constants. At the same time, every interval $J$ can be chosen for a uniformly bounded number of Whitney cubes depending on the same constants.	
\end{remark}

Combining Lemma \ref{lemDorronsoro} with the proof of \cite[(A.1)]{PratsPlanarDomains} and noting that in Definition \ref{defBetasBoundary} one has the uniform bound $\#\mathcal{J}_Q\leq C$, one gets the following lemma:
\begin{lemma}\label{lemNormalVectorBetas}
Let $0<s<1$, $ 1< p<\infty$ with $sp>2$ and let $\Omega$ be a $(\delta,R)$-Lipschitz $B^{s+1-\frac{1}{p}}_{p,p}$-domain. Then
\begin{equation*}
\left(\sum_{Q\in \mathcal{W}: \ell(Q)\leq C_\Omega} \beta_{(1)}(Q)^p \ell(Q)^{2-sp}\right)^\frac1p \lesssim \norm{\nu}_{B^{s-\frac{1}{p}}_{p,p}(\partial\Omega)},
\end{equation*}
with constants depending on the Lipschitz character of $\Omega$ and $\mathcal{H}^1(\partial\Omega)$, where $\nu$ stands for the unit outward normal vector to the boundary of the domain.
\end{lemma}

The $\beta$-coefficients will appear in a natural way along the present section thanks to the following relation introduced in \cite[(7.3)]{CruzTolsa}.  
\begin{lemma}\label{lemBoundForBeurling}
Let $\Omega$ be a bounded $(\delta,R)$-Lipschitz domain and let $\mathcal{W}$ be a  Whitney covering with appropriate constants. Then, for $x,y\in Q\in\mathcal{W}$ with $Q\subset \bigcup_{j=1}^M B\left(x_j,\frac{\varepsilon_\delta}{12}\right)$, there exists a half plane $\Pi_Q$ so that for every $0<\ell_0<R$, the estimate
\begin{equation}\label{eqBoundForBeurling}
\int_{\Omega \Delta \Pi_Q}\frac{1}{|z-x|^{2+\eta}} dm(z) \lesssim \left( \sum_{\substack{P\in \mathcal{W}: \rho P \supset Q\\ \ell(P)\leq \ell_0}} \frac{\beta_{(1)}(P)}{\ell(P)^\eta} + \frac{1}{\ell_0^\eta}\right)
\end{equation}
holds, with the constant $\rho>1$ depending only on the Lipschitz character of the domain and $\eta >0$.
\end{lemma}

Note that the condition $\rho P\supset Q$ in the last sum above, implies that the cubes $P$ cannot be much smaller than $Q$, namely $\ell(P)\geq \frac{\ell(Q)}{\rho}$, and thus the number of cubes $P$ on a given scale  stays uniformly bounded for any given $Q$. Essentially $\Pi_Q$ is a half-plane whose boundary coincides with the minimizer for $\beta_{(1)}(Q)$. To be precise, we choose $(x_j,J) \in \mathcal{J}_Q$ and we choose $\Pi_Q$ so that it contains $Q$ and $\tau_j\partial \Pi_Q$ minimizes $\beta_{(n)}(A_j,J)$, see Definition \ref{defBetasBoundary}. Above we chose the Whitney constants big enough so that that $\dist(Q,\Pi_Q^c)\approx \ell(Q)$. Note that in case $\ell(Q)\geq C_\Omega$ we can chose any half-plane whose boundary is at distance from $Q$ comparable to $\ell(Q)$.

To end with beta coefficients, we write the following lemma, which will be used several times along the text.
\begin{lemma}\label{lemSummingTheBetas}
Let $\Omega\subset \C$ be a bounded $B^{s+1-\frac1p}_{p,p}$-domain, with $0<s<1$, $1<p<\infty$ and $sp>2$. For $\ell>s$, we have that
$$\sum_{Q\in\W} \ell(Q)^{2+(\ell-s)p}  \left(\left(\sum_{\rho P \supset Q: \ell(P)\leq 2R} \frac{\beta_{(1)}(P)}{\ell(P)^\ell}\right)^p + 1\right)\lesssim  \norm{\nu}_{B^{s-\frac1p}_{p,p}(\partial\Omega)}^p,$$
with constants depending on $\ell,s,p$, the Whitney constants, the Lipschitz character of the domain and its diameter. 
\end{lemma}

\begin{proof}
Fix $\varepsilon < \ell-s$. Then
\begin{align*}
\squared{S}
\nonumber	& := \sum_{Q\in\W} \ell(Q)^{2+(\ell-s)p} \left(\left(\sum_{\rho P \supset Q: \ell(P)\leq 2R} \frac{\beta_{(1)}(P)}{\ell(P)^\ell}\right)^p + 1\right)\\
\nonumber	& \lesssim  \left(C_\Omega^{2+(\ell-s)p} +\sum_{Q} \ell(Q)^{2+(\ell-s)p}  \sum_{\rho P \supset Q} \frac{\beta_{(1)}(P)^p}{\ell(P)^{(\ell-\varepsilon )p}} \left(\frac{1}{\ell(Q)^{\varepsilon p'}}\right)^{\frac{p}{p'}}\right)\\
	& \lesssim  \left(C_{\Omega,\ell-s,p} + \sum_{P} \frac{\beta_{(1)}(P)^p}{\ell(P)^{(\ell-\varepsilon )p}} \sum_{Q \subset \rho P} \ell(Q)^{2+(\ell-s-\varepsilon )p} \right) .
\end{align*}
Now, using Remark \ref{remUniformBounds}, for $P\in\mathcal{W}$, since $2+(\ell-s-\varepsilon )p>2$  we have that
$$ \sum_{Q\subset \rho P} \ell(Q)^{2+(\ell-s-\varepsilon )p}   \approx \sum_{J\subset I: I \in \mathcal{J}(P)} \ell(J)^{2+(\ell-s-\varepsilon )p}  \leq \sum_{I \in \mathcal{J}(P)}\ell(I)^{2+(\ell-s-\varepsilon )p} \approx\ell(P)^{2+(\ell-s-\varepsilon )p}.$$
 Thus, by  Lemma \ref{lemNormalVectorBetas} we get
\begin{align*}
\squared{S}
	& \lesssim \left(C_{\Omega,\ell-s,p} + \sum_{P} \frac{\beta_{(1)}(P)^p}{\ell(P)^{sp-2}} \right)
	\leq C_{\Omega,\ell-s,p}   \norm{\nu}_{B^{s-\frac1p}_{p,p}(\partial\Omega)}^p.
\end{align*}
\end{proof}
\begin{remark}Since we are in a Lipschitz domain, it is enough $2+(\ell-s-\varepsilon )p>1$, see \cite[Lemma 3.12]{PratsTolsa}. Thus, the preceding lemma holds in fact whenever $\ell>s-\frac1p$. 
\end{remark}

\subsection{Boundedness of the truncated iterates}\label{secIterates}
Consider $1<p<\infty$, $sp>2$, and $0<s< 1$ and let $\Omega$ be a bounded $B^{s+1-\frac1p}_{p,p}$-domain. Victor Cruz and Xavier Tolsa showed that for every $f\in F^s_{p,2}(\Omega)$ we have that
\begin{equation*}
\norm{\Beurling(\chi_\Omega f)}_{F^s_{p,2}(\Omega)}\leq C \norm{\nu}_{B^{s-1/p}_{p,p}(\partial\Omega)}\norm{f}_{F^s_{p,2}(\Omega)},
\end{equation*}
where $C$ depends on $p$, $s$, $\diam(\Omega)$ and the Lipschitz character of the domain.
 (see \cite[Corollary 1.3]{CruzTolsa}).

Nevertheless, this estimate is not enough, since we need to estimate the iterates of the Beurling transform, that is, Theorem \ref{theoIterates} below. Moreover, we are dealing with the larger Triebel-Lizorkin scale (with values other than $2$ allowed for $q$ as well). Thus, we proceed to give a quantitative control of $\norm{(\Beurling^m)_\Omega}_{F^s_{p,q}(\Omega)}$. The following is a fractional version of what the second author got in \cite{PratsPlanarDomains}. 
\begin{theorem}\label{theoIterates}
Consider $m\in\N$,  $0<s< 1$ and $1<p,q<\infty$ with $sp>2$,  and let $\Omega$ be a bounded $B^{s+1-\frac1p}_{p,p}$-domain. Then, for every $f\in F^s_{p,q}(\Omega)$ we have that
\begin{equation*}
\norm{\Beurling^m(\chi_\Omega f)}_{F^s_{p,q}(\Omega)}\leq C m^2 \norm{\nu}_{B^{s-1/p}_{p,p}(\partial\Omega)} \norm{f}_{F^s_{p,q}(\Omega)},
\end{equation*}
where $C$ depends on $s$, $p$, $q$, $\diam(\Omega)$ and the Lipschitz character of the domain but not on $m$.
\end{theorem}

{During the last 40 years, {research on Calder\'on-Zygmund theory} has produced a number of $T1$ and $Tb$ theorems, which consist in reducing the boundedness of an operator $T$ in a function space to fact that $T1$ (or $T(b)$ for a certain bounded function $b$) belongs to that space. The second and third authors of the present manuscript obtained a $T(1)$ Theorem in \cite{PratsSaksman} in the framework of Triebel-Lizorkin spaces on domains, which allows us to reduce the proof of Theorem \ref{theoIterates} above to checking the behavior of the operators on the constant functions.}

\begin{lemma}\label{lemBeurlingIteratesCharacteristic}
Consider $m\in\N$, let $0<s<1$, $1<p,q<\infty$ with $sp>2$, let $\Omega$ be a bounded $B^{s+1-\frac{1}{p}}_{p,p}$-domain. Then $\Beurling^m \chi_\Omega \in F^s_{p,q}(\Omega)$ and, moreover, 
$$\norm{\Beurling^m \chi_\Omega}_{{\dot{ F}}^s_{p,q}(\Omega) }\lesssim {m^2}\norm{\nu}_{B^{s-\frac{1}{p}}_{p,p}(\partial\Omega)},$$
the constant depending only on the indices {$s$, $p$ and $q$}, the Lipschitz character of the domain and its diameter.
\end{lemma}
\begin{proof}
Note that if $p<q$, then $\norm{\Beurling^m \chi_\Omega}_{ F^s_{p,q}(\Omega) }\lesssim \norm{\Beurling^m \chi_\Omega}_{ B^s_{p,p}(\Omega) }$ by Proposition \ref{theoTriebelEmbeddingsOmega}, so we will assume with no loss of generality that $p\geq q$, which in particular implies that $s>0\geq\frac{2}{p}-\frac{2}{q}$. We follow the approach given in \cite[proof of Lemma 6.3]{CruzTolsa} which gets quite shorter with the  norm given in Theorem \ref{theoNormOmegapgtrq}. Indeed, we only need to control the homogeneous seminorm
\begin{equation}\label{eqNormBeurlinCharac1}
\norm{\Beurling^m \chi_\Omega}_{\dot F^s_{p,q}(\Omega) }^p
	 :=  \sum_{Q\in\mathcal{W}} \int_Q \left(\int_{2Q}\frac{|\Beurling^m \chi_\Omega(x)-\Beurling^m \chi_\Omega(y)|^q}{|x-y|^{sq+2}} \,dy\right)^{\frac{p}{q}}dx,
 \end{equation}
and the non-local part in the aforementioned proof, which is the most difficult one to treat, is not there anymore. 
	 
Choose a half-plane $\Pi_Q$ as in Lemma \ref{lemBoundForBeurling}. By \rf{eqDorronsoroStayAway}, chosing appropriate Whitney constants we have that $x$ and $y$ are in $\Pi_Q$. Next we use that, \emph{formally},  $\Beurling \chi_{\Pi_Q}$ is constant in $\Pi_Q$ and $\overline{\Pi_Q}^c$ (see \cite[Lemma 4.2]{CruzTolsa}), i.e., that $\Beurling\chi_{\Pi_Q}=c \chi_{\Pi_Q}$ modulo constants and, by induction, $\Beurling^m \chi_{\Pi_Q}$ is constant in $\Pi_Q$ as well, so 
$|\Beurling^m \chi_\Omega(x)-\Beurling^m\chi_\Omega (y)| = |\Beurling^m \chi_\Omega(x)-\Beurling^m \chi_\Omega (y)-\Beurling^m \chi_{\Pi_Q}(x)+\Beurling^m\chi_{\Pi_Q} (y)|$
(understood in the BMO sense {as in \cite[Section 4.6]{AstalaIwaniecMartin}). {We next  make} this statement rigorous.}

It is well-known (see \cite[Section 4.1.4]{AstalaIwaniecMartin}) that 
\begin{equation}\label{eqBeurlingm}
\Beurling^m \varphi (z) = \lim_{\varepsilon \to 0} \frac{m(-1)^m}{\pi} \int_{|w-z|>\varepsilon} \frac{(\overline{z-w})^{m-1}}{(z-w)^{m+1}}\varphi(w) \, dm(w) \quad\quad \mbox{for }\varphi \in L^2.
\end{equation}
{In order to be able to write a concrete formula for the action of the Beurling transform on $L^\infty$- or $BMO$-functions, such as characteristic functions we are dealing with, one can apply  a specific kernel as in \cite[(4.91)]{AstalaIwaniecMartin}. On the other hand   under iteration such a formula  easily becomes cumbersome, and in the literature there seems not to  appear suitable formulae for this purpose.}  
However, the reader can check that
$$\int_{r\D \cap \Pi_Q\setminus B(x,\varepsilon)} \frac{(\overline{z-x})^{m-1}}{(z-x)^{m+1}} dm(z)- \int_{r\D \cap \Pi_Q\setminus B(y,\varepsilon)}\frac{(\overline{z-y})^{m-1}}{(z-y)^{m+1}}dm(z) \xrightarrow{r\to\infty} 0$$
using Green's formula \cite[Theorem 2.9.1]{AstalaIwaniecMartin} (see the proof of Lemma \ref{lemKernelVanishes} below for inspiration), and
$$\int_{(r\D)^c{\cap \Pi_Q} } \left(\frac{(\overline{z-x})^{m-1}}{(z-x)^{m+1}} - \frac{(\overline{z-y})^{m-1}}{(z-y)^{m+1}}\right)dm(z) \xrightarrow{r\to\infty} 0.$$
{Note that the preceding fact can be seen as a consequence of the evenness of the kernel.} {These formulae solve the above iteration problem,} and using \rf{eqBeurlingm} we can write
\begin{align*}
|\Beurling^m \chi_\Omega(x)-\Beurling^m\chi_\Omega (y)|
\nonumber	& \leq m \int_{\Omega \Delta \Pi_Q} \left| \frac{(\overline{z-x})^{m-1}}{(z-x)^{m+1}}- \frac{(\overline{z-y})^{m-1}}{(z-y)^{m+1}}  \right|dm(z) \\
	&  \lesssim m^2 \int_{\Omega \Delta \Pi_Q}\frac{|x-y|}{|z-x|^3} dm(z).
\end{align*}

By {formula \rf{eqBoundForBeurling} in Lemma \ref{lemBoundForBeurling}}, we can write
\begin{align*}
|\Beurling^m \chi_\Omega(x)-\Beurling^m\chi_\Omega (y)|
	&  \lesssim m^2 |x-y| \left(\sum_{P\in \mathcal{W}: Q\subset \rho P} \frac{\beta_{(1)}(P)}{\ell(P)} + \frac{1}{\diam(\Omega)}\right).
\end{align*}
Combining with \rf{eqNormBeurlinCharac1}, we get that
\begin{align*}
\norm{\Beurling^m \chi_\Omega}_{\dot F^s_{p,q}(\Omega) }^p
\nonumber 	& \lesssim  \sum_{Q\in\mathcal{W}} |Q| \left( \ell(Q)^{(1-s)q} m^{2q} \left(\sum_{P: Q\subset \rho P} \frac{\beta_{(1)}(P)}{\ell(P)} + \frac{1}{\diam(\Omega)} \right)^q  \right)^{\frac{p}{q}}\\
 	 & \approx  m^{2p}\sum_{Q\in\mathcal{W}} \ell(Q)^{2+p-sp}   \left(\sum_{P: Q\subset \rho P} \frac{\beta_{(1)}(P)}{\ell(P)} + \frac{1}{\diam(\Omega)} \right)^p
\end{align*}
and Lemma \ref{lemSummingTheBetas} with $\ell=1$ ends the proof.
\end{proof}

{
To  complete the proof of Theorem \ref{theoIterates} we need polynomial estimates for the growth of the $ F^s_{p,q}$ norm of the iterates of the Beurling transform. 
\begin{lemma}\label{theoLinearGrowth}
Let $m\in\N$, $1<p,q<\infty$ and $s>0$. Then $B^m$ is bounded on $F^s_{p,q}$ with norm
$$\norm{B^m}_{{F}^s_{p,q}\to {F}^s_{p,q}} \lesssim m,$$
where the quantifier depends only on $p,q,$ and the dimension.
\end{lemma}
\begin{proof}
E.g.  \cite[Thm 1.1]{CJJ05} gives the stated bound on  the space $\dot{F}^{s}_{p,q}$ {since}  the convolution kernel of $T^m$ obeys the $L^\infty$-bound $\lesssim m$ on the unit circle. In view of the analogous bound on $L^p$, the desired norm bound then follows for  the operator $T^m$ acting on ${F}^{s}_{p,q}$.
\end{proof}}

\begin{proof}[Proof of Theorem \ref{theoIterates}]
{We use the $T1$ theorem in \cite[Theorem 1.1]{PratsSaksman}, which says that $B_\Omega^m 1\in F^s_{p,q}(\Omega)$ implies that $B_\Omega^m$ is bounded as an operator from $F^s_{p,q}(\Omega)$ to itself. In particular we use the quantitative bound  \cite[(5.8)]{PratsSaksman},  according to which }we have that  
\begin{align*}
\norm{\Beurling_\Omega^{{m}}}_{F^s_{p,q}(\Omega)\to F^s_{p,q}(\Omega)}
	& \lesssim m^2 + \norm{\Beurling^{{m}}}_{F^s_{p,q}\to F^s_{p,q}}+\norm{\Beurling^{{m}}}_{L^p\to L^p}+\norm{\Beurling^{{m}}}_{L^q\to L^q} + \norm{\Beurling_\Omega^{{m}} 1}_{\dot F^s_{p,q}(\Omega)}.
\end{align*}
{Lemma \ref{theoLinearGrowth}} and Lemma \ref{lemBeurlingIteratesCharacteristic} imply that
\begin{align*}
\norm{\Beurling_\Omega^{{m}}}_{F^s_{p,q}(\Omega)\to F^s_{p,q}(\Omega)}
	& \lesssim_{s,p,q,\diam(\Omega),\delta} m^2  \left( \norm{\nu}_{B^{s-1/p}_{p,p}(\partial\Omega)} +1 \right) \approx_{s,p,q,\diam(\Omega),\delta} m^2  \norm{\nu}_{B^{s-1/p}_{p,p}(\partial\Omega)}. 
\end{align*}
\end{proof}

\subsection{Compactness of the commutator}\label{secCommutator}
\begin{lemma}\label{lemCompactnessCommutator}
The commutator $[\mu, \Beurling_\Omega]$ is compact in $F^s_{p,q}(\Omega)$.
\end{lemma}

\begin{proof}
Choose $s<\beta<1$. If $\mu \in C^\infty(\C)$, then $[\mu,\Beurling_\Omega]: L^\infty \to F^\beta_{p,2}$ is compact (see \cite[(18) and the subsequent paragraph]{CruzMateuOrobitg}). Precomposing with the inclusion $F^s_{p,q} \to L^\infty $ and postcomposing with  $F^\beta_{p,2}\to F^s_{p,q} $ for $\beta>s$ we get that the lemma holds for $C^\infty$ coefficients.

To show compactness for general $\mu\in F^s_{p,q}(\Omega)$, we only need to see that the commutator can be approximated in operator norm by a sequence of commutators with smooth coefficients. For this one only needs to approximate $\mu$ by $\{\mu_n\}\subset C^\infty(\bar\Omega)$ (combine the density of $C^\infty_c$ functions in $F^s_{p,q}$ in \cite[Theorem 2.3.3]{TriebelTheory} and Theorem \ref{theoExtension}, for instance). By \rf{eqAlgebra} and Theorem \ref{theoIterates} (\cite{CruzTolsa} is enough in this case), we conclude that 
$$[\mu_n,\Beurling_\Omega] \to [\mu,\Beurling_\Omega]$$
 in the operator norm.
\end{proof}

\subsection{Beurling Reflection}\label{secReflection}

\begin{theorem}\label{theoCompactnessReflection}
Let $0<s<1$, $1<p,q<\infty$ with $sp>2$, and let $\Omega$ be a bounded $B^{s+1-\frac1p}_{p,p}$-domain. For every $m$, the operator $\mathcal{R}_m$ is compact in $F^s_{p,q}(\Omega)$.
\end{theorem}

Theorem \ref{theoCompactnessReflection} is a straight consequence of the Rellich-Kondrachov compactness Theorem (see \cite[Remark 4.3.2/1]{TriebelTheory}) together with the following proposition.

\begin{proposition}\label{lemSmallNormCloseToBoundary}
Let $0<s<1$, $1<p,q<\infty$ with $sp>2$, and let $\Omega$ be a bounded $B^{s+1-\frac1p}_{p,p}$-domain, and $m\in\N$. Then
\begin{equation*}
\norm{\mathcal{R}_m f }_{\dot F^s_{p,q}(\Omega)} \lesssim_h \norm{f}_{\mathcal{C}^h(\Omega)}
\end{equation*}
for every $h>0$.
\end{proposition}

Next we take a closer look to the kernel of the Beurling reflection defined in \rf{eqDefinitionReflection}. The reflection can be written as
$$\mathcal{R}_m f(z)=\int_{\Omega^c} \frac{m(-1)^{m+1}\pi^{-2}}{(z-w)^{2}} \int_\Omega f(\xi) \frac{(\overline{w-\xi})^{m-1}}{(w-\xi)^{m+1}}\, dm(\xi)\, dm(w).$$

In a quite general setting, one can use Fubini in the former expression of $\mathcal{R}_m$ and the related kernel $\widetilde{K}_m(z,\xi)=\int_{\Omega^c} \frac{(\overline{w-\xi})^{m-1}}{(z-w)^{2}\,(w-\xi)^{m+1}}\, dw$ appears as a natural element. Mateu, Orobitg and Verdera study this kernel in \cite[Lemma 6]{MateuOrobitgVerdera} assuming the boundary of the domain $\Omega$ to be in ${\mathcal{C}}^{1+\varepsilon}$ for $\varepsilon<1$. They prove the size inequality 
$$|\widetilde{K}_m(z,\xi)|\lesssim \frac{1}{|z-\xi|^{2-\varepsilon}}$$
and a smoothness inequality in the same spirit. Cruz, Mateu and Orobitg proved an analogous result to Theorem \ref{theoCompactnessReflection} under stronger assumptions on the regularity of the boundary in \cite{CruzMateuOrobitg}, namely, that the boundary had ${\mathcal{C}}^{1+s+\varepsilon}$ parameterizations. They could show that the kernel is smoothing in this context. Their proof was based on the size and the smoothness estimates of the kernel shown in \cite{MateuOrobitgVerdera}, which could be useful for the case $F^\sigma_{p,2}(\Omega)$ with $\sigma<s-2/p$ but they are not sufficiently strong to deal with the endpoint case $F^s_{p,2}(\Omega)$ when the domain has just $B^{1+s-\frac1p}_{p,p}$ parameterizations. Nevertheless, their argument was adapted in \cite{PratsQuasiconformal} to get Proposition \ref{propoKernelExpression} below, which will be used to prove Proposition \ref{lemSmallNormCloseToBoundary}. 

Let us collect the necessary background. Given $m \in \N$,
 let us define the kernel
\begin{equation}\label{eqKernelVectorM}
K_{m}(z,\xi):=\int_{\Omega^c} \frac{m(-1)^{m+1}\pi^{-2} (\overline{w-\xi})^{m-1}}{(z-w)^{3}\,(w-\xi)^{m+1}}\, dm(w)=\int_{\partial \Omega} \frac{c_m(\overline{w-\xi})^{m}}{(z-w)^{3}\,(w-\xi)^{m+1}}\, dw
 \end{equation}
for all $z, \xi \in \Omega$, where the path integral is oriented counterclockwise. Note that for suitable $z$ and $f$ we will be able to use Fubini's Theorem to get 
$$\partial\mathcal{R}_m f(z)=\int_\Omega f(\xi) K_m(z,\xi)\, dm(\xi).$$

\begin{lemma}\label{lemKernelVanishes}
Let $\Pi$ be an open half plane, and $x,y\in \Pi$. For $m_1,m_2,m_3\in \N_0=\N \cup \{0\}$ with $m_1+m_2-m_3>2$ we have that 
$$\int_{\Pi^c}\frac{(\overline{w-y})^{m_3}}{({x-w})^{m_1}({w-y})^{m_2}}\,dm(w)
=  \int_{\partial \Pi}\frac{(\overline{w-y})^{m_3+1}}{({x-w})^{m_1}({w-y})^{m_2}} dw 
=0$$
\end{lemma}
\begin{proof}
Without loss of generality, we may assume that $\Pi$ is the upper half plane. For a suitable constant $c$, Green's and Cauchy's theorems imply that
$$\begin{aligned}
\int_{ \Pi^c}\frac{c(\overline{w-y})^{m_3}}{({x-w})^{m_1}({w-y})^{m_2}}dm(w) 
& = \int_{\partial \Pi}\frac{(\overline{w-y})^{m_3+1}}{({x-w})^{m_1}({w-y})^{m_2}} dw\\
& =\int_{ \partial \Pi}\frac{(w-\overline{y})^{m_3+1}}{({x-w})^{m_1}({w-y})^{m_2}} dw = 0.\end{aligned}
$$
\end{proof}

We will use an auxiliary function.
\begin{definition}
Let us define
\begin{equation*}
h_{m}(z):=\int_{\partial\Omega}\frac{(\overline{\tau-z})^{m}}{\tau-z} \,d\tau  \mbox{\,\,\,\, for every }z\in\Omega.
\end{equation*}
\end{definition}
By \cite[Proposition 3.6]{PratsQuasiconformal} the weak derivatives of order $m$ of $h_{m}$ are 
\begin{equation}\label{eqDerivativesBeurlings}
\partial^{j}\bar\partial^{m-j}h_{m}=c_{m,j}\Beurling^{j}\chi_\Omega \mbox{, \,\,\,\, for }0\leq j\leq m.
\end{equation} 
To shorten notation, we will write $H^j_m=\partial^{j}h_{m}$. 

Combining \rf{eqDerivativesBeurlings} with Lemma \ref{lemBeurlingIteratesCharacteristic}, one obtains the following:
\begin{lemma}\label{lemHjm}
Let $0\leq j\leq m$, $0<s<1$, $sp>2$ and let $\Omega$ be a bounded $B^{s+1-\frac1p}_{p,p}$-domain. Then 
$$\nabla^{m-j}H^j_m \in B^s_{p,p}(\Omega).$$
\end{lemma}

Given a $j$ times differentiable function $f$, we will write 
\begin{equation}\label{eqTaylor}
P^{j}_z(f)(\xi)=\sum_{|\vec{i}|\leq j}\frac{D^{\vec{i}}f(z)}{\vec{i}!} (\xi-z)^{\vec{i}}
\end{equation}
for its $j$-th degree Taylor polynomial centered in the point $z$.  Note that here we are using the standard multi-index notation for the powers.

\begin{proposition}[see {\cite[Proposition 3.6]{PratsQuasiconformal}}]\label{propoKernelExpression}
Let $\Omega$ be a bounded Lipschitz domain, and let $m\geq 1$. 
Then, for every pair $z, \xi \in \Omega$ with $z\neq \xi$, we have that
\begin{equation}\label{eqExpressionTaylorAndMore}
K_{m}(z,\xi)=c_{m} \partial \Beurling\chi_\Omega (z) \frac{(\overline{\xi-z})^{m-1}}{(\xi-z)^{m+1}} + \sum_{j\leq m} c_{m,j} \frac{H^j_m(\xi)-P^{m-j}_z H^j_m(\xi)}{(\xi-z)^{m+3-j}},
\end{equation}
\end{proposition}
Note that the Taylor polynomials are well defined because Lemma \ref{lemHjm} implies the required differentiability. 

\begin{proof}[Proof of Proposition \ref{lemSmallNormCloseToBoundary}]
We assume that $p\geq q$, since otherwise, one has that $\norm{\mathcal{R}_m f}_{F^s_{p,q}(\Omega)}\lesssim \norm{\mathcal{R}_m f}_{F^s_{p,p}(\Omega)}$ (see Proposition \ref{theoTriebelEmbeddingsOmega}). 

Let $0<\rho<1$ to be fixed later on. For $f\in F^s_{p,q}(\Omega)$, let us write
\begin{equation}\label{eqDerivativeLocal}
D^s_q f(x):= \left(\int_{B\left(x,\rho \delta_\Omega(x)\right)} \frac{|f(x)-f(y)|^q}{|x-y|^{sq+2}} \, dy \right)^{\frac{1}{q}}.
\end{equation}
We want to show that
\begin{equation*}
\squared{0}:=\norm{D^s_q \mathcal{R}_m f }_{L^p(\Omega)}=\left(\sum_{Q\in \mathcal{W}}  \norm{D^s_q \mathcal{R}_m f }_{L^p(Q)}^p\right)^\frac1p \leq C_h \norm{f}_{\mathcal{C}^h(\Omega)}.
\end{equation*}

For every Whitney cube $Q$ we choose a bump function $\chi_{3Q} \leq  \varphi_Q \leq \chi_{4Q}$ with $|\nabla\varphi_Q| \leq \frac{1}{\ell(Q)}$. Then,
\begin{align}\label{eqBreakingReflection}
\squared{0}^{\,p} 
\nonumber & \lesssim \sum_{Q} |f_Q|^p \norm{D^s_q \mathcal{R}_m 1 }_{L^p(Q)}^p + \sum_{Q} \norm{D^s_q \mathcal{R}_m (f-f_Q)\varphi_Q }_{L^p(Q)}^p \\
	&\quad  + \sum_{Q}  \norm{D^s_q \mathcal{R}_m (f-f_Q)(1-\varphi_Q)}_{L^p(Q)}^p 
	 = \squared{1}+\squared{2}+\squared{3}.
\end{align}
We will show that each term is bounded by $C \norm{f}_{\mathcal{C}^h(\Omega)}^p$. 

Let us begin by the first term in the right-hand side of \rf{eqBreakingReflection}, which is the easiest one. Indeed, for any cube $Q$, the mean $|f_Q|\leq \norm{f}_{L^\infty(\Omega)}$.  On the other hand the boundedness of $\mathcal{R}_m$ in the Triebel-Lizorkin space under consideration implies that $\mathcal{R}_m 1 \in F^s_{p,q}(\Omega)$. By Theorem \ref{theoNormOmegapgtrq}, this implies
\begin{equation*}
\squared{1} \lesssim \norm{f}_{L^\infty(\Omega)}^p \norm{ \mathcal{R}_m 1 }_{F^s_{p,q}(\Omega)}^p .
\end{equation*}

Next, let us face the local part in \rf{eqBreakingReflection}. We fix the following notation: when dealing with the difference of a function $F$ between two points, we will write
$$F[(x)-(y)]:=F(x)-F(y).$$
Let $x,y\in Q\in\mathcal{W}$. Then, since $\Beurling [\chi_{\Omega^c} \Beurling^m(\chi_\Omega f)]$ is analytic on $\Omega$, it has continuous derivatives and, thus, by the mean value theorem 
$$|\mathcal{R}_m [(f-f_Q)\varphi_Q][(x)-(y)]| \leq \norm{\partial\mathcal{R}_m [(f-f_Q)\varphi_Q]}_{L^\infty(2Q)} |x-y| $$
and, fixing a convenient $\rho$ in \rf{eqDerivativeLocal}, we get that
\begin{align}\label{eq2FirstStep}
\squared{2}
\nonumber	& \leq \sum_{Q} \int_Q \left(\int_{2Q} \frac{|\mathcal{R}_m [(f-f_Q)\varphi_Q][(x)-(y)]|^q}{|x-y|^{sq+2}} \, dy \right)^{\frac{p}{q}} dx\\
\nonumber	& \leq \sum_{Q} \norm{\partial\mathcal{R}_m [(f-f_Q)\varphi_Q]}_{L^\infty(2Q)}^p \int_Q \left(\int_{2Q} \frac{ |x-y|^q}{|x-y|^{sq+2}} \, dy \right)^{\frac{p}{q}} dx\\
	& \lesssim \sum_{Q} \ell(Q)^{2+(1-s)p} \norm{\partial\mathcal{R}_m [(f-f_Q)\varphi_Q]}_{L^\infty(2Q)}^p.
\end{align}
Take $z\in 2Q$. Then
$$\partial\mathcal{R}_m [(f-f_Q)\varphi_Q](z) 
= \int_{\Omega^c} \frac{c_m}{(z-w)^{3}} \int_\Omega\frac{(\overline{w-\xi})^{m-1} (f(\xi)-f_Q)\varphi_Q(\xi)}{(w-\xi)^{m+1}}\, dm(\xi)\, dm(w).$$
It is immediate to check that this double integral is absolutely convergent and, thus, Fubini's Theorem applies and it follows that 
$$\partial\mathcal{R}_m [(f-f_Q)\varphi_Q](z) =  \int_{4Q} (f(\xi)-f_Q)\varphi_Q(\xi) \int_{\Omega^c} \frac{c_m(\overline{w-\xi})^{m-1}}{(z-w)^{3}\,(w-\xi)^{m+1}}\, dm(w)\, dm(\xi).$$
Next, we consider the half-plane $\Pi_Q$ from Lemma \ref{lemBoundForBeurling}.
Recall that $\dist(Q,\Pi_Q^c)\approx \ell(Q)$.   Then, Lemma \ref{lemKernelVanishes} implies that
$$\begin{aligned}
\partial\mathcal{R}_m [(f-f_Q)\varphi_Q](z) 
	& =  \int_{4Q} (f(\xi)-f_Q)\varphi_Q(\xi) \\
		& \quad \cdot \left(\int_{\Omega^c} - \int_{\Pi_Q^c} \right)\frac{c_m(\overline{w-\xi})^{m-1}}{(z-w)^{3}\,(w-\xi)^{m+1}}\, dm(w) \, dm(\xi).
\end{aligned}
$$

Since $z\in 2Q$, taking absolute values we obtain
\begin{align*}
|\partial\mathcal{R}_m [(f-f_Q)\varphi_Q](z)| 
	& \leq  \int_{4Q} |f(\xi)-f_Q| \int_{\Omega\Delta \Pi_Q} \frac{c_m}{|z-w|^{3}\,|w-\xi|^2}\, dm(w)\, dm(\xi) \\
	& \lesssim \ell(Q)^2 \norm{f}_{L^\infty (\Omega)} \int_{\Omega\Delta \Pi_Q} \frac{1}{|w-z|^{5}}\, dm(w).
\end{align*}

By {formula \rf{eqBoundForBeurling} in Lemma \ref{lemBoundForBeurling}}, we get
\begin{align*}
|\partial\mathcal{R}_m [(f-f_Q)\varphi_Q](z)| 
	& \lesssim \ell(Q)^2 \norm{f}_{L^\infty (\Omega)} \left(\sum_{\rho P \supset Q: \ell(P) < R} \frac{\beta_{(1)}(P)}{\ell(P)^{3}} + R^{-3}.\right)
\end{align*}
Back to \rf{eq2FirstStep}, we have that
\begin{align*}
\squared{2}
	& \lesssim \norm{f}_{L^\infty (\Omega)}^p  \sum_{Q} \ell(Q)^{2+(3-s)p} \left(\left(\sum_{\rho P \supset Q: \ell(P) < R} \frac{\beta_{(1)}(P)}{\ell(P)^{3}}\right)^p + C\right).
\end{align*}
By Lemma \ref{lemSummingTheBetas} we get
\begin{align*}
\squared{2}
	& \lesssim \norm{f}_{L^\infty (\Omega)}^p \norm{\nu}_{B^{s-\frac1p}_{p,p}(\partial\Omega)}.
\end{align*}

It remains to control the nonlocal part in \rf{eqBreakingReflection}, that is, 
$$\squared{3}\leq \sum_{Q} \int_Q \left(\int_{2Q} \frac{|\mathcal{R}_m [(f-f_Q)(1-\varphi_Q)][(x)-(y)]|^q}{|x-y|^{sq+2}} \, dy \right)^{\frac{p}{q}} dx.$$
As in \rf{eq2FirstStep}, by the mean value property of analytic functions, we have that
\begin{align}\label{eq3FirstStep}
\squared{3}
\nonumber	& \lesssim \sum_{Q} \ell(Q)^{2+(1-s)p} \norm{\partial\mathcal{R}_m [(f-f_Q)(1-\varphi_Q)]}_{L^\infty(2Q)}^p\\
	& \lesssim \sum_{Q} \ell(Q)^{2(1-p)+(1-s)p} \norm{\partial\mathcal{R}_m [(f-f_Q)(1-\varphi_Q)]}_{L^1(\frac52Q)}^p.
\end{align}

Take $z\in \frac52Q$. Then
$$\begin{aligned}
\partial\mathcal{R}_m &[(f-f_Q)(1-\varphi_Q)](z) 
\\&= \int_{\Omega^c} \frac{c_m}{(z-w)^{3}} \int_\Omega\frac{(\overline{w-\xi})^{m-1} (f(\xi)-f_Q)(1-\varphi_Q(\xi))}{(w-\xi)^{m+1}}\, dm(\xi)\, dm(w).\end{aligned}
$$
This double integral is absolutely convergent:
$$\begin{aligned}
\int_{\Omega^c} \int_\Omega \frac{|f(\xi)-f_Q||1-\varphi_Q(\xi)|}{|z-w|^{3}|w-\xi|^{2}}\, &dm(\xi)\, dm(w)\\&\lesssim \norm{f}_{L^\infty} \int_{\Omega^c}  \frac{|\log(\delta_\Omega(w))|+|\log(\diam(\Omega))|}{|z-w|^{3}}\,  dm(w).\end{aligned}
$$
Thus, we can apply Fubini's Theorem,  \rf{eqKernelVectorM} and \rf{eqExpressionTaylorAndMore} to get
\begin{align*}
\partial\mathcal{R}_m &[(f-f_Q)(1-\varphi_Q)](z) \\
\nonumber	 &=  \int_{\Omega\setminus 3Q} (f(\xi)-f_Q)(1-\varphi_Q(\xi)) \int_{\Omega^c} \frac{c_m(\overline{w-\xi})^{m-1}}{(z-w)^{3}\,(w-\xi)^{m+1}}\, dm(w)\, dm(\xi)\\
\nonumber	& = c_m  \partial \Beurling\chi_\Omega (z)  \int_{\Omega\setminus 3Q} (f(\xi)-f_Q)(1-\varphi_Q(\xi))  \frac{(\overline{\xi-z})^{m-1}}{(\xi-z)^{m+1}} \, dm(\xi) \\
	& \quad +  \int_{\Omega\setminus 3Q} (f(\xi)-f_Q)(1-\varphi_Q(\xi)) \sum_{j\leq m} \frac{c_{m,j} (H^j_m(\xi)-P^{m-j}_z H^j_m(\xi))}{(\xi-z)^{m+3-j}} dm(\xi).
\end{align*}
Whenever $z\in \Omega\setminus \supp F$, we have that 
$$\Beurling^m_\Omega F(z)=c_m \int_{\Omega\cap \supp F} F(\xi)  \frac{(\overline{\xi-z})^{m-1}}{(\xi-z)^{m+1}} \, dm(\xi).$$ Thus, we can apply this identity in the first term of the right-hand side above, and back to \rf{eq3FirstStep}, we obtain
\begin{align}\label{eqBreak3}
\squared{3}
\nonumber	& \lesssim \sum_{Q} \ell(Q)^{2(1-p) + (1-s)p} \norm{\partial \Beurling\chi_\Omega (z)  \Beurling^m_\Omega [(f-f_Q)(1-\varphi_Q)] (z)}_{L^1_{{z}}(\frac52 Q)}^p \\
\nonumber	 &\,\,\,\,  +  \sum_{j\leq m} \sum_{Q} \ell(Q)^{2(1-p) + (1-s)p}  \norm{\int_{\Omega\setminus 3Q} |f(\xi)-f_Q| \frac{|H^j_m(\xi)-P^{m-j}_z H^j_m(\xi)|}{|\xi-z|^{m+3-j}} \,dm(\xi)}_{L^1_{z}(\frac52 Q)}^p \\
	& = \squared{3.1} + \sum_{j=0}^m \squared{3.2.j}
\end{align}

For the first term in the right-hand side, we have that 
$$\squared{3.1} \leq  \sum_{Q} \ell(Q)^{2+(1-s)p} \norm{\partial \Beurling\chi_\Omega}_{L^\infty(\frac52 Q)}^p\norm{ \Beurling^m_\Omega [(f-f_Q)(1-\varphi_Q)]}_{L^\infty(\frac52 Q)}^p.$$
Using again the half-plane $\Pi_Q$ from Lemma \ref{lemBoundForBeurling}, whose boundary minimizes $\beta_{(1)}(Q)$, for $z\in \frac52 Q$ we can write
$$\partial \Beurling\chi_\Omega(z)= c \left(\int_{\Omega\setminus B(z,\frac12\delta_\Omega)} - \int_{\Pi_Q\setminus B(z,\frac12\delta_\Omega)} \right) \frac{1}{(z-w)^3} dm(w)$$
(see \cite[Lemma 4.2]{CruzTolsa}), where we wrote again $\delta_\Omega(z)$ for $\dist(z,\partial\Omega)$. Taking absolute values, by Lemma \ref{lemBoundForBeurling} we get
$$\norm{\partial \Beurling\chi_\Omega}_{L^\infty(\frac52 Q)}\lesssim \int_{\Omega\Delta \Pi_Q} \frac{1}{\Dist(w,Q)^3}\, dm(w) \lesssim  \sum_{\rho P \supset Q: \ell(P) < R} \frac{\beta_{(1)}(P)}{\ell(P)} + R^{-1}.$$

On the other hand, for $h\leq s-\frac2p$, by \cite[Main Lemma]{MateuOrobitgVerdera} and doing some routine computations, one can check that
\begin{align*}
\norm{ \Beurling^m_\Omega [(f-f_Q)(1-\varphi_Q)]}_{L^\infty(\Omega)}
\nonumber	& \leq \norm{ \Beurling^m_\Omega [(f-f_Q)(1-\varphi_Q)]}_{\mathcal{C}^h(\Omega)}\\& \lesssim_m \norm{(f-f_Q)(1-\varphi_Q)}_{\mathcal{C}^h(\Omega)}
	 \lesssim \norm{f}_{\dot{\mathcal{C}}^h(\Omega)} .
\end{align*}

Summing up, we have seen that
\begin{align*}
\squared{3.1} 
\nonumber	& \leq \norm{f}_{\dot{\mathcal{C}}^h(\Omega)}^p \sum_{Q} \ell(Q)^{2+(1-s)p} \left(\sum_{\rho P \supset Q: \ell(P) < R} \frac{\beta_{(1)}(P)}{\ell(P)} + R^{-1}\right)^p 
\end{align*}
Using Lemma \ref{lemSummingTheBetas} again,
\begin{equation*}
\squared{3.1} \leq C \norm{f}_{ \dot{\mathcal{C}}^h(\Omega)}^p .
\end{equation*}

Consider the term $\squared{3.2.j}$ in \rf{eqBreak3}. We trivially control by the supremum norm of $f$:
\begin{align*}
\squared{3.2.j} 
\nonumber	& \leq  \norm{f}_{L^\infty(\Omega)}^p \sum_{Q} \ell(Q)^{2(1-p) + (1-s)p}  \norm{\int_{\Omega\setminus 3Q}  \frac{|H^j_m(\xi)-P^{m-j}_z H^j_m(\xi)|}{|\xi-z|^{m+3-j}} \,dm(\xi)}_{L^1_{{z}}(\frac52 Q)}^p.
\end{align*}

{By expressing the integral above as sums of integrals on cubes, we can {complete}  the proof of Proposition \ref{propoKernelExpression} just  {by} checking the following estimate
$$\begin{aligned}
&\sum_{Q\in\W} \ell(Q)^{2(1-p) + (1-s)p}  \left( \int_{\frac52Q}   \sum_{S\in\W} \int_S \frac{|H^j_m(\xi)-P^{m-j}_z H^j_m(\xi)|}{\Dist(Q,S)^{m+3-j}}\,  dm(\xi)dm(z) \right)^p  \\ & \quad\lesssim  \norm{\nabla^{m-j} H^j_m}_{B^s_{p,p}(\Omega)}.\end{aligned}
$$
To obtain this estimate, however, we need to introduce some tools, so we defer its proof to the following section. To be precise, {the above estimate will be}  a consequence of Lemma \ref{lemTaylorToBesov} below with $F=H^j_m$, $n=m-j$. }

{Once this estimate is obtained}, we get that
$$\squared{3.2.j} \lesssim  \norm{f}_{L^\infty(\Omega)}^p \norm{\nabla^{m-j}H^j_m}_{B^s_{p,p}(\Omega)},$$
and by Lemma \ref{lemHjm}, the last factor is finite.
\end{proof}

\subsection{Meyer's polynomials}\label{secMeyer}
The following lemma is true for every Whitney covering.
\begin{lemma}[{See \cite[Lemma 3.11]{PratsTolsa}}]
Let $d\geq 2$. Assume that $r>0$. If $\eta>0$, for every $Q\in\mathcal{W}$ we have
\begin{equation}\label{eqMaximalFar}
 \sum_{S\in\W}  \frac{\ell(S)^d}{D(Q,S)^{d+\eta}}\lesssim \frac{1}{\ell(Q)^\eta}.
 \end{equation}
\end{lemma}

\begin{definition}\label{defChain}
 If $\Omega$ is a Lipschitz domain, for every $Q,S\in\W$, we can find a chain $[Q,S]$, that is, a sequence of cubes $(Q_1,\cdots ,Q_N)$ satisfying $\overline{Q_j}\cap\overline{Q_{j+1}}\neq \emptyset$ for all $j<N$ with $Q_1=Q$, $Q_N=S$, and a central cube $Q_S:=Q_{j_0}$ for $j_0\leq N$ such that the following holds:
\begin{equation}\label{eqChainDistances}
\mbox{If $j\leq j_0$, then $\ell(Q_j)\approx \Dist(Q, Q_j)$, while $\ell(Q_j)\approx \Dist(Q_j,S)$ otherwise,}
\end{equation}
and
\begin{equation}\label{eqChainLength}
\sum_{j=1}^N \ell(Q_j)\lesssim \Dist(Q,S)\approx \ell(Q_S).
\end{equation}
\end{definition}
The constants involved depend on the Whitney constants and the Lipschitz character of the domain. The interested reader may find more information in \cite[Section 3]{PratsTolsa}.
In that paper one shows that the number of cubes in a chain of a given side-length is uniformly bounded, that is 
\begin{equation}\label{eqChainNumberOfCubesUniformlyBounded}
\#\{P \in [Q,S]: \ell(P) = \ell_0 \} < C.
\end{equation}

More generally, a \emph{uniform} domain is a domain having a Whitney covering such that for every pair of cubes there exists a chain satisfying \rf{eqChainDistances} and  \rf{eqChainLength}. Moreover, as a consequence it also satisfies  \rf{eqChainNumberOfCubesUniformlyBounded} (see \cite{PratsSaksman}).

The proof of Proposition \ref{lemSmallNormCloseToBoundary} above depends on the following {estimate:}
\begin{lemma}\label{lemTaylorToBesov}
Let $n\in\N$ and $0<s<1$. Let $\Omega$ be a uniform domain with Whitney covering $\mathcal{W}$, and let $F\in C^n(\Omega)$ such that its weak derivatives $\nabla^n F\in B^s_{p,p}(\Omega)$. Then
$$\begin{aligned}
\squared{\emph{N}} & :=\sum_{Q\in\W} \ell(Q)^{2(1-p) + (1-s)p}  \left( \int_{\frac52Q}   \sum_{S\in\W} \int_S \frac{|F(\xi)-P^{n}_z F(\xi)|}{\Dist(Q,S)^{n+3}}\, dm(\xi)dm(z) \right)^p\\&   \lesssim  \norm{\nabla^n F}_{B^s_{p,p}(\Omega)} 
\end{aligned}
$$
\end{lemma}

Meyers' approximating polynomials are very useful to deal which such a situation: consider the set $\mathcal{P}^n$ of polynomials of degree at most $n$. Given a cube $Q$ and a function $f\in L^1_{\rm loc}(\frac52Q)$, the Meyers polynomial $P^{n}_Q f\in \mathcal{P}^n$ is the unique polynomial in $\mathcal{P}^n$ satisfying that $\int_{\frac52Q} \nabla^j f = \int_{\frac52Q} \nabla^j P^n_Q f$ for $j\leq n$.  It satisfies the Poincar\'e inequality 
\begin{equation}\label{eqPoincare}
\norm{\nabla^k (f-P^{n}_Qf)}_{L^1(\frac52 Q)}
	\lesssim \ell(Q)^{n-k} \norm{\nabla^{n}f- (\nabla^{n} f)_Q}_{L^1(\frac52 Q)},
\end{equation}
whenever $f\in W^{n,1}(3Q)$, $k\leq n$.

\begin{proof}[Proof of Lemma \ref{lemTaylorToBesov}]
We change the Taylor polynomial centered at $z\in Q$ by the corresponding Meyers' polynomial as follows:
\begin{align*}
\squared{N} 
\nonumber	& \lesssim   \sum_{Q\in\mathcal{W}} \ell(Q)^{-2\frac{p}{p'} + (1-s)p}  \left(\sum_{S\in\mathcal{W}}  \frac{\norm{\norm{P^{n}_Q F-P^{n}_z F}_{L^1(S)}}_{L^1_{{z}}(\frac52 Q)}}{\Dist(Q,S)^{n+3}}\right)^p \\
\nonumber	& +   \sum_{Q\in\mathcal{W}} \ell(Q)^{-2\frac{p}{p'} + (1-s)p}  \left(\sum_{S\in\mathcal{W}}  \frac{\ell(Q)^2 \norm{F-P^{n}_Q F}_{L^1(S)}}{\Dist(Q,S)^{n+3}}\right)^p = \squared{E}+\squared{M}.
\end{align*}

The error term $\squared{E}$ may be addressed using the following facts. First, given a polynomial $P$ of degree at most $n$ and disjoint cubes $Q$ and $S$, we have that
\begin{equation}\label{eqPolyLluny}
\norm{P}_{L^1(S)}\lesssim \frac{\ell(S)^2\Dist(Q,S)^{n}}{\ell(Q)^{2+n}}\norm{P}_{L^1(Q)}.
\end{equation}
(use the fact that all norms on $\mathcal{P}^n$ are equivalent and appropriate rescaling factors). Thus,
$$\squared{E}\lesssim \sum_{Q\in\mathcal{W}} \ell(Q)^{-2\frac{p}{p'}+ (1-s)p}  \left(\sum_{S\in\mathcal{W}}  \frac{\norm{\norm{P^{n}_Q F-P^{n}_z F}_{L^1(Q)}}_{L^1_{{z}}(\frac52 Q)}}{\Dist(Q,S)^{3}}\frac{\ell(S)^2}{\ell(Q)^{2+n}}\right)^p$$

Using Fubini, we can change the order of integration and since the Taylor polynomial of a polynomial of the same degree is itself, we get 
$$\norm{\norm{P^{n}_Q F-P^{n}_z F}_{L^1(Q)}}_{L^1_{{z}}(\frac52 Q)} \leq \norm{\norm{P^{n}_z(F-P^{n}_Q F)}_{L^1_{{z}}(\frac52 Q)}}_{L^\infty(Q)}\ell(Q)^2. $$
But using the expression  \rf{eqTaylor} of the Taylor Polynomial of degree $n$, for $\xi,z\in 3Q$ we have that
$$\norm{P_z^{n} f (\xi)}_{L^1_{{z}}(\frac52 Q)}\leq \sum_{0\leq |\vec{i}|\leq n} \frac{1}{\vec{i}!}\norm{D^{\vec{i}} f(z) (\xi-z)^{\vec{i}} }_{L^1_{{z}}(\frac52 Q)} \lesssim \sum_{k=0}^{n} \norm{\nabla^k f}_{L^1(\frac52 Q)} \ell(Q)^{k}.$$
Plugging the Poincar\'e inequality \rf{eqPoincare} in, we get
$$\norm{\norm{P^{n}_Q F-P^{n}_z F}_{L^1(Q)}}_{L^1_{{z}}(\frac52 Q)} \lesssim \ell(Q)^{2+n} \norm{\nabla^{n}F- (\nabla^{n} F)_Q}_{L^1(\frac52 Q)}. $$

Back to the error term, we get that
$$\squared{E}\lesssim     \sum_{Q\in\mathcal{W}} \ell(Q)^{-2\frac{p}{p'} + (1-s)p}  \left(\norm{\nabla^{n}F- (\nabla^{n} F)_Q}_{L^1(\frac52 Q)}\sum_{S\in\mathcal{W}}  \frac{\ell(S)^2}{\Dist(Q,S)^{3}}\right)^p.$$
Using \rf{eqMaximalFar} and the H\"older inequality, 
$$\squared{E}\lesssim     \sum_{Q\in\mathcal{W}} \ell(Q)^{-2\frac{p}{p'} + (1-s)p}  \norm{\nabla^{n}F- (\nabla^{n} F)_Q}_{L^p(\frac52 Q)}^p \ell(Q)^{2\frac{p}{p'}} \ell(Q)^{-p},$$
so 
$$\squared{E}\lesssim    \sum_{Q\in \mathcal{W}} \frac{\norm{\nabla^{n}F- (\nabla^{n} F)_Q}_{L^p(\frac52 Q)}^p}{\ell(Q)^{sp}}\lesssim \norm{\nabla^{n}F}_{B^s_{p,p}(\Omega)}.$$

To estimate the main term $\squared{M}$ we will argue by duality. Writing 
$$h_Q:= \ell(Q)^{1-s + \frac2{p}}  \sum_{S\in\mathcal{W}}  \frac{\norm{F-P^{n}_Q F}_{L^1(S)}}{\Dist(Q,S)^{n+3}},$$
it follows that 
$$\begin{aligned}
\squared{M}^\frac1p
	& = \norm{ \{h_Q\}_{Q\in\mathcal{W}}}_{\ell^p(\mathcal{W})} = \left(\sup_{\{g_Q\}} \sum_{Q\in\mathcal{W}} h_Q g_Q \right)
	\\&=  \sup_{\{g_Q\}} \sum_{Q\in\mathcal{W}}  \ell(Q)^{1-s + \frac2{p}}g_Q  \sum_{S\in\mathcal{W}}  \frac{\norm{F-P^{n}_Q F}_{L^1(S)}}{\Dist(Q,S)^{n+3}}  
\end{aligned}
$$
where the supremum is taken over the sequences $\{g_Q\}_{Q\in\mathcal{W}}$ satisfying that 
$$\norm{\{g_Q\}_{Q\in\mathcal{W}}}_{\ell^{p'}(\mathcal{W})}= 1.$$
Fix Whitney cubes $Q$ and $S$. Next we use a telescoping summation following the chain of cubes $[Q,S]$ introduced in Definition \ref{defChain}: 
\begin{align*}
\norm{F-P^{n}_Q F}_{L^1(S)}
	& \leq \norm{F-P^{n}_S F}_{L^1(S)} + \sum_{P\in [Q,S)}\norm{P^{n}_P F-P^{n}_{\mathcal{N}(P)} F}_{L^1(S)},
\end{align*}
where $\mathcal{N}(P)$ stands for the next cube in the chain $[Q,S]$.  By Definition \ref{defWhitney}, the side of a given Whitney cube is at most twice as long as the side of its neighbors. Thus, for $P\in [Q,S]$ we have $P\subset 5\mathcal{N}(P)$. Using \rf{eqPolyLluny},
\begin{align*}
\norm{F-P^{n}_Q F}_{L^1(S)}
	& \lesssim \sum_{P\in [Q,S]} \norm{F-P^{n}_P F}_{L^1(5P)} \frac{\ell(S)^2\Dist(P,S)^{n}}{\ell(P)^{2+n}}.
\end{align*}

Using the Poincar\'e inequality \rf{eqPoincare} and the H\"older inequality, for $P\in [Q,S]$ we get that
\begin{align*}
\norm{F-P^{n}_P F}_{L^1(5P)}
	&  \lesssim \norm{\nabla^{n} F-(\nabla^{n} F)_P}_{L^1(5P)}\ell(P)^{n} \\
	& \lesssim \norm{\nabla^{n} F-(\nabla^{n} F)_P}_{L^p(5P)}\ell(P)^{n+\frac2{p'}} \\		& \lesssim \norm{D^s_p \nabla^{n} F}_{L^p(5P)} \ell(P)^{s+n+\frac2{p'}}
\end{align*}
(see \rf{eqDerivativeLocal}).
Since $\Dist(P,S)\lesssim \Dist(Q,S)$, we obtain
$$\squared{M}^\frac1p\lesssim \sup_{\{g_Q\}} \sum_{P\in\mathcal{W}} \norm{D^s_p \nabla^{n} F}_{L^p(5P)}  \sum_{Q,S: P\in[Q,S]}  \frac{\ell(Q)^{1 + \frac2{p}-s}g_Q\ell(S)^2}{\ell(P)^{\frac2p-s}\Dist(Q,S)^{3}}.$$
To {complete  the proof we need to} use the maximal Hardy-Littlewood operator
{$$MG(x):=\sup_{Q: x\in Q} \frac{1}{|Q|}\int_Q |G(y)|\, dy, \quad\quad\text{for every }G\in L^1_{\rm loc}$$}
which is bounded on Lebesgue spaces, so we define  $G(x):=\sum_{Q\in\mathcal{W}}\frac{g_Q}{\ell(Q)^{2/p'}} \chi_Q(x)$. It is clear that $\norm{G}_{L^{p'}(\Omega)}=\norm{\{g_Q\}_{Q\in\mathcal{W}}}_{\ell^{p'}(\mathcal{W})}= 1$, and $\ell(Q)^\frac2p g_Q=\int_Q G$. Thus, 
$$\squared{M}^\frac1p\lesssim \sup_{G\in L^{p'}(\Omega)} \sum_{P\in\mathcal{W}} \frac{\norm{D^s_p \nabla^{n} F}_{L^p(5P)} }{{\ell(P)^{\frac2p-s}}} \sum_{Q,S: P\in[Q,S]}  \frac{\ell(Q)^{1 -s}  \ell(S)^2}{\Dist(Q,S)^{3}} \int_Q G(x) \, dx.$$

{Next we make the following claim, inspired in \cite[Lemma 2.5]{PratsQuasiconformal}.}
\begin{lemma}\label{lemTwoFoldedFunctionGuay}
Consider a uniform domain $\Omega\subset \R^d$ with Whitney covering $\mathcal{W}$, a cube $P\in\mathcal{W}$, a function $G\in L^{1}(\Omega)$ and two real numbers $0<s<\ell$. Then
$$ \sum_{Q,S: P\in[Q,S]}  \frac{\ell(Q)^{\ell-s}\ell(S)^2}{\Dist(Q,S)^{2+\ell}}\int_Q G(x) \, dx \lesssim \frac{1}{\ell(P)^s}\int_P (M G(x)+M^2 G(x))\, dx.$$
\end{lemma}

{Before proving  Lemma \ref{lemTwoFoldedFunctionGuay}, let us see how it can be used to complete the proof of Lemma \ref{lemTaylorToBesov}: Take $\ell=1$ in the lemma.} By H\"older's inequality and the finite overlapping of Whitney cubes, we get that 
$$\squared{M}^\frac1p \lesssim  \sup_{G\in L^{p'}(\Omega)} \sum_{P\in\mathcal{W}} \norm{D^s_p \nabla^{n} F}_{L^p(5P)}   \norm{MG+ M^2G}_{L^{p'}(P)} \lesssim \norm{D^s_p \nabla^{n} F}_{L^p(\Omega)}$$
{by the boundedness of the maximal operator in $L^{p'}$, which verifies Lemma \ref{lemTaylorToBesov}.}

{Let us then turn to proving Lemma \ref{lemTwoFoldedFunctionGuay}.}We divide the chain $[Q,S] =[Q,Q_S] \cup [Q_S,S]$, in such a way that if $P$ is in the ascending path $[Q,Q_S]$ then $\ell(P)\approx \Dist(Q,P)$ and $\Dist(P,S)\approx \Dist(Q,S)\approx \ell(Q_S)$, and if $P$ is in the descending path we get analogous conditions, see \rf{eqChainDistances}. Thus, we write
$$\left(\sum_{Q,S: P\in[Q,Q_S]}+\sum_{Q,S: P\in[Q_S,S]}\right)  \frac{\ell(Q)^{\ell-s}\ell(S)^2}{\Dist(Q,S)^{2+\ell}}\int_Q G(x) \, dx  =:\squared{A}+\squared{D}.$$

For every cube $P$ we get
\begin{align}\label{eqBoundCubes}
\sum_{Q:\Dist(Q,P)\lesssim \ell(P)}  \ell(Q)^{\ell-s} \int_Q G(x) \, dx  
	& \leq \ell(P)^{\ell-s} \int_P MG(x)\, dx .
\end{align}
In the ascending path, thus, using \rf{eqMaximalFar} and \rf{eqBoundCubes} we obtain
\begin{align*}
\squared{A}	
	& \leq \sum_{S\in\mathcal{W}} \frac{\ell(S)^2}{\Dist (S,P)^{2+\ell}}\sum_{Q:\Dist(Q,P)\lesssim \ell(P)}  \ell(Q)^{\ell-s} \int_Q G(x) \, dx  
	\approx \ell(P)^{-\ell} \ell(P)^{\ell-s} \int_P MG(x)\, dx
\end{align*}

In the descending path, we divide the sum in $Q$ in ``dyadic annuli'' just by setting $R=Q_S$:
$$\squared{D}=\sum_{S:\Dist(S,P)\lesssim \ell(P)} \ell(S)^2 \sum_{R: \Dist(R,P)\lesssim \ell(R)}\frac{1}{\ell(R)^{2+\ell}} \sum_{Q: \Dist(Q,R)\lesssim \ell(R)} \ell(Q)^{\ell-s} \int_Q G(x)\, dx.$$
Again, first we will use \rf{eqBoundCubes} to get
\begin{align*}
 \squared{D}
 	& \lesssim \ell(P)^2 \sum_{R: \Dist(R,P)\lesssim \ell(R)} \frac{1}{\ell(R)^{2+s}} \int_R MG(x)\, dx \lesssim   \int_PM^2 G(y) dy\sum_{R: \Dist(R,P)\lesssim \ell(R)}\frac{1}{\ell(R)^{s}}.
	\end{align*}
Since the last sum above is a geometric series, we obtain that
\begin{align*}
 \squared{D}
 	& \lesssim  \frac{1}{\ell(P)^s}\int_P M^2 G(y)\, dy .
 \end{align*}

\end{proof}

\begin{appendices}
\section{Appendix: A universal extension operator}
In \cite{Rychkov}, for any given index {$\raja>0$} the author defined  an operator $\mathcal{E}=\mathcal{E}_\raja$ on the space of distributions on a $(\delta,\infty)$-Lipschitz domain $\Omega\subset \R^d$ (aka special Lipschitz domain) such that it is an extension operator for $F^s_{p,q}(\Omega)$ for $s<\raja$ and for all the admissible values of $p$ and $q$. In this section we check that the same operator maps also $L^\infty(\Omega)$ to $L^\infty$:
\begin{theorem}\label{theoRychkov}
Given a special Lipschitz domain $\Omega$ and $\raja\in \N$, any operator $\mathcal{E}:=\mathcal{E}_\raja$ as defined in  {\rm \cite[Theorem 2.2]{Rychkov}} maps $L^\infty(\Omega)$ to $L^\infty$.
\end{theorem}

\noindent Before proving the theorem we  recall the definition of  Rychkov's extension $\mathcal{E}$. Due to the fact that we are dealing with a $(\delta,\infty)$-Lipschitz domain, there exists an open cone $K' $ such that its translates satisfy $x+K'\subset \Omega$ for every $x\in\Omega$. We may assume that $K'=\bigcup_{t>0}B(tx_0,tr_0)$, where $|x_0|>r_0>0$. Denote by $K:=\bigcup_{t>0}B(tx_0,tr_0/2)$ the smaller cone 'compactly' contained in $K'$.
Take $\varphi_0\in C^\infty_c(-K)$ with integral one and write
$$\varphi(x)=\varphi_0(x)-2^{-d}\varphi_0(x/2), \quad\quad \varphi_j(x):=2^{jd}\varphi(2^j x) \, \mbox{ for } j\in\N.$$
Let $\raja\in \N$ be given (actually, when considering Besov or Tribel spaces one demands that $L>\max\{s-1,d/p, d/q\}$), and assume that $\varphi$ has vanishing moments up to order $\raja$, i.e.
$$\int_{\R^n}x^\alpha \varphi(x)\, dx =0\qquad\textrm{for all}\quad \alpha\in \N^d \;\;\textrm{with}\;\; |\alpha|\leq \raja.$$	
According to \cite[Proposition 2.1]{Rychkov}, there exist functions $\psi_0,\psi\in C^\infty_0(-K)$ depending only on $K$ and $\raja$ such that $\psi$ has vanishing moments up to order $\raja$ and denoting $\psi_j:=2^{jd}\psi(2^j \cdot)$ for $j\in\N$ we have
\begin{equation}\label{eqIdentity}
f\mapsto \sum_{j=0}^\infty \psi_j*\varphi_j* f \mbox{ is the identity in $\mathcal{D}'(\Omega)$}.
\end{equation}
Finally, one simply sets
$$\mathcal{E} f:= \sum_{j=0}^\infty \psi_j* \big((\varphi_j* f)|_0\big),$$
where $(\varphi_j* f)|_0$ stands for the zero extension of the locally integrable function $(\varphi_j* f)$ (the latter one is defined on the domain $\Omega$).
\begin{proof}[Proof of Theorem \ref{theoRychkov}]
The extension acts as the identity inside the domain and the boundary of $\Omega$ has zero Lebesgue measure, so it is enough to  fix $x\in {\overline{\Omega}}^c$ and check that  $\mathcal{E}f(x)\leq C\norm{f}_{L^\infty(\Omega)}$, with $C$ independent of $x$ or $f$. For that end first first note that
\begin{claim}\label{claimFiniteSum}
Let $f\in L^\infty(\Omega)$. There exists $M\in\N$ such that for every $x\in \overline{\Omega}^c$ we may write 
$$\mathcal{E} f(x)= \sum_{j=0}^{N_x-1} \psi_j*\varphi_j*f_0 (x)+  \sum_{j=N_x}^{N_x+M} \psi_j*((\varphi_j*f)_0)(x)$$
for a suitable index $N_x\in\{0,1,2,\ldots\}$. Above  $f_0$ is the extension of $f$ by zero to the whole of $\R^d$ .
\end{claim}
Indeed, by simple geometry, if $d(x,\Omega)\sim 2^{-j_0},$ then $\psi_j*((\varphi_j*f)_0)(x)=0$ for $j\geq j_0+M_0$ with a fixed $M_0$ just by the definition of a convolution.  On the other hand, by increasing $M_0$ if needed, we see that  $\psi_j(x-\cdot)$ is fully supported in $\Omega$
for $j\leq j_0-M_0$, and hence for these indices $j$ it holds that  
$$
\psi_j*((\varphi_j*f)_0)(x)= \psi_j*\varphi_j*f(x)=\psi_j*\varphi_j*f_0(x).
$$
Here $M_0$ does not depend on $x$. In view of the definition the extension $\mathcal{E}$ we may thus take $N_x=(j_0-M_0)\vee 0$ and $M:=2M_0$.

By e.g. noting that \eqref{eqIdentity} remains valid if $\Omega$ is replaced by its translates, any $g\in L^\infty(\R^d)$ satisfies 
$$
g = \sum_{j=0}^\infty \psi_j*\varphi_j* g \qquad\textrm{on}\quad \R^d.
$$
Especially, by choosing  $g=f_0$ and looking again at the supports we see that for our fixed $x$ it holds that
$$
0= f_0(x)=  \sum_{j=0}^{N_x+M} \psi_j*\varphi_j*f_0(x).
$$
As we subtract this from the claim \ref{claimFiniteSum} it finally follows that
$$\mathcal{E} f(x)=   \sum_{j=N_x}^{N_x+M}\Big( \psi_j*((\varphi_j*f)_0)(x)-\psi_j*\varphi_j*f_0(x)\Big)$$
The desired boundedness clearly follows since $\psi_j$:s and $\varphi_j$:t have uniformly bounded $L^1$-norms and the number of summands does not depend on $f$ or $x$.
\end{proof}
\begin{corollary}\label{coroRychkov}
Given a Lipschitz domain $\Omega$ and $s\in \N$, there exists an operator $\mathcal{E}:=\mathcal{E}_s$ defined in $\mathcal{D}'(\Omega)$ that is an extension operator from $L^\infty(\Omega)$ to $L^\infty$ and from $F^\sigma_{p,q}(\Omega)$ to $F^\sigma_{p,q}$ for every $\sigma\leq s$, every $1/s<p<\infty$ and every $1/s\leq q\leq \infty$.
\end{corollary}
\end{appendices}

\renewcommand{\abstractname}{Acknowledgements}
\begin{abstract}  We are grateful for H. Triebel for helpful comments regarding Proposition \ref{propoRiemannTriebel}.
The first author was supported by Academy of Finland project SA13316965. The second author was supported by the Spanish State Research Agency, through the Severo Ochoa and Mar\'ia de Maeztu Program for Centers and Units of Excellence in R\&D (CEX2020-001084-M), by the Spanish government under the grant IJC2018-035373-I and by ERC grants 320501 (FP7/2007-2013), 307179-GFTIPFD and partially supported by MTM-2016-77635-P, PID2020-114167GB-I00, PID2021-125021NA-I00, PID2021-123405NB-I00 (Spain) and projects 2017-SGR-395, 2021 SGR 00087 and 2021-SGR-00071 (Catalonia). The two first authors were supported also by the ERC grant 834728-QUAMAP. The third author was supported by the Academy of Finland CoE ``Analysis and Dynamics'', as well as the Academy of Finland Project ``Conformal methods in analysis and random geometry''.
\end{abstract}

\bibliography{LlibresAPS}

\end{document}